%% file: main_TransAMS.tex
\documentclass[11pt]{amsart}
\usepackage[T1]{fontenc}
\usepackage{graphicx}
\usepackage{url}
\usepackage{hyperref}
\input{xy}
\xyoption{all}
\xyoption{poly}
\usepackage[all]{xy}
\usepackage{float}
\usepackage{enumerate}
\usepackage{spverbatim}
\usepackage{breqn}
\usepackage{graphicx}
\usepackage[usenames]{color}
 \usepackage{amsmath,amsthm,amsfonts,amssymb}
      \theoremstyle{plain}
      \newtheorem{theorem}{Theorem}[section]
      \newtheorem*{theorem*}{Theorem}
      \newtheorem{lemma}[theorem]{Lemma}
      \newtheorem{corollary}[theorem]{Corollary}
      \newtheorem{proposition}[theorem]{Proposition}
      
      \theoremstyle{definition}
	  \newtheorem{example}[theorem]{Example}
      \newtheorem{definition}[theorem]{Definition}
     \theoremstyle{remark}
      \newtheorem{remark}[theorem]{Remark}
\usepackage{graphicx}
\usepackage[dvipsnames]{xcolor}
\definecolor{darkcerulean}{rgb}{0.03, 0.27, 0.49}
\usepackage{accents}
\newcommand{\dbtilde}[1]{\tilde{\raisebox{0pt}[0.85\height]{$\tilde{#1}$}}}

 \newcommand\RR{{\mathbb{R}}}
 
 \newcommand\ZZ{{\mathbb{Z}}}
 \newcommand\NN{{\mathbb{N}}}

 \def\A{{\mathcal A}}
 \def\B{{\mathcal B}}

 \def\G{{\mathcal G}}
 \def\H{{\mathcal H}}

 \def\K{{\mathcal K}}
 \def\L{{\mathcal L}}
 \def\M{{\mathcal M}}
 
 \def\P{{\mathcal P}}
 \def\Q{{\mathcal Q}}
 
 \def\S{{\mathcal S}}

 \def\X{{\mathcal X}}

 \def\ZZ{{\mathbb Z}}

 \def\Img{{\mathrm {Im}}}

\newcommand\nplusonedef{\hspace{2 pt}\diagup\hspace{-4.8 pt} \searrow\hspace{-9pt}^{^{n+1}} \hspace{5 pt}} 
\newcommand\ndef{\hspace{2 pt}\diagup\hspace{-4.8 pt} \searrow\hspace{-8pt}^n \hspace{5 pt}}
\newcommand\threedef{\hspace{2 pt}\diagup\hspace{-5 pt} \searrow\hspace{-8pt}^3 \hspace{5 pt}}

\newcommand\ce{{\hspace{2.4 pt} \searrow\hspace{-8 pt}^e \hspace{5 pt}}}
\newcommand\ee{{\hspace{3 pt} \nearrow\hspace{-13 pt}^e \hspace{8 pt}}}
\newcommand\se{{\hspace{2 pt}\diagup\hspace{-4.8 pt} \searrow\hspace{5 pt}}}
\newcommand\co{{\hspace{2 pt}\searrow \hspace{3 pt}}}
\newcommand\ex{{\nearrow \hspace{3 pt}}}

      \makeatletter
      \def\@setcopyright{}
      \def\serieslogo@{}
      \makeatother
   
\begin{document}

\title {Morse theory for group presentations}
 \author{Ximena Fern\'andez}
\address{Department of Mathematics, Swansea University, UK and Departamento de Matem\'atica, FCEN, Universidad de Buenos Aires, Argentina.}
\email{xfernand@dm.uba.ar}

   \begin{abstract}
We introduce a novel combinatorial method to study  $Q^{**}$-transformations of group presentations or, equivalently, 3-deformations
of CW-complexes of dimension 2.
Our procedure is based on a refinement of discrete Morse theory that gives a Whitehead simple homotopy equivalence from a regular CW-complex to the simplified Morse CW-complex, with an explicit description of the attaching maps and bounds on the dimension of the complexes involved in the deformation.
We apply this technique to show that some known potential counterexamples to the Andrews--Curtis conjecture do satisfy the conjecture.
   \end{abstract}

\subjclass[2020]{
57M07, %Topological methods in group theory
57Q10, % Simple homotopy type, Whitehead torsion, Reidemeister-Franz torsion, etc.\\
57Q70,  %Discrete Morse theory and related ideas in manifold topology\\
20F05,  %Generators, relations, and presentations\\
55-04. %Explicit machine computation and programs
}

\keywords{discrete Morse theory, 3-deformations, Andrews--Curtis conjecture, Q**-transformations, posets}

\maketitle

Any finite presentation  $\P=\langle x_1,\dots,x_n ~|~  r_1,\dots,r_m\rangle $ of a group $G$ 
can be transformed into any other presentation of the same group by a 
finite sequence of the following operations \cite{MR1547755, MR1812024}:
\begin{enumerate}
 \item replace some relator $r_i$ by $r_i^{-1}$;
 \item replace some relator  $r_i$ by $r_i r_j$ for some $j \neq i$;
 \item replace some relator $r_i$ by a conjugate $wr_i w^{-1}$ for some $w$ in the free group $F(x_1, x_2, \dots, x_n)$;
\item add a  generator $x_{n+1}$ and a relator $r_{m+1}$ that coincides with $x_{n+1}$, or the inverse of this operation;
\item replace each relator $r_i$ by $\phi(r_i)$ where $\phi$ is an automorphism of $F(x_1, x_2,\dots, x_n)$;
\item add a relator $1$, or the inverse of this operation.
\end{enumerate}
For balanced presentations (i.e., $m=n$), J. Andrews and M. Curtis conjectured in 1965 \cite{MR0173241}  that any presentation of the trivial group $\P=\langle x_1,\dots,x_n ~|~  r_1,\dots,r_n\rangle $
can be transformed into the empty presentation $\langle ~|~ \rangle $ by a 
finite sequence of the  operations $(1)$ to $(5)$, called \textit{$Q^{**}$-transformations} \cite{MR564434, MR813099}. This question still remains open and it has become one of the most notorious problems in group theory, as well as in low-dimensional topology due to its topological consequences.
Indeed, although the original conjecture is in the area of combinatorial group theory,  it has an equivalent formulation -- first noticed by the anonymous referee of the foundational article \cite{MR0173241} -- in terms of  Whitehead's simple homotopy theory \cite{MR0005352, MR0035437}. It states that any (finite) contractible 2-dimensional CW-complex $K$ 3-deforms to a
point. That is, $K$ can be transformed into a point by a sequence of elementary collapses and expansions in which the dimension of the complexes involved is not greater than 3 (see \cite{MR0380813} and \cite[Sect. 2.3. Ch. I]{hog1993two} for a detailed proof of this equivalence). This conjecture is closely related to other relevant  problems in algebraic topology,
such as Whitehead asphericity conjecture \cite{MR0004123},  Zeeman conjecture \cite{MR0156351} and the smooth 4-dimensional Poincar\'e conjecture \cite{kirby_problems} (see also \cite{hog1993two}).
The Andrews--Curtis conjecture is known to be true for some classes of complexes
(such as the \textit{standard spines} \cite{MR710105} and the 
\textit{quasi-constructible complexes} \cite{MR3024764}),
but the problem still remains unsolved for general 2-complexes or, equivalently, balanced presentations of the trivial group. Moreover, there is a list of balanced presentations of the trivial group for which no $Q^{**}$-trivialization is known. They serve as \textit{potential counterexamples} to disprove the conjecture (see \cite[Sect. 1.1 Ch. XII]{hog1993two} and \cite{myasnikov2002andrews}).

Computational approaches to this problem have proven to be limited by the exponential complexity of the algorithms \cite{MR2253006,  bridson2015complexity, MR1970867, KS16, MR1727164, MR1829485, MR1921712}. Most of them are based on the exploration and exhibition of possible transformations of type (1)--(3) from a given presentation.
However, in \cite{bridson2015complexity} M. Bridson exhibited examples of rather small balanced presentations of the trivial group for which %it is physically impossible to exhibit an explicit 
the minimum length of any $Q^{**}$-simplification sequence to $\langle~|~\rangle$ is super-exponential in the total length of the relators (and hence, computationally intractable with standard methods). 
In this article, we present a
method which combines topological and combinatorial tools,  that allows the computational 
exploration of presentations which are $Q^{**}$-equivalent from a given one without the need of exhibiting the actual list of transformations. 
This alternative technique to find $Q^{**}$-transformations, based on a refinement of discrete Morse theory, enables us to show that some well known potential counterexamples to the conjecture can be easily $Q^{**}$-trivialized. Moreover, we also apply this theory to presentations of non-trivial groups, proving that some potential counterexamples to the \textit{generalized Andrews-Curtis conjecture} \cite{Ba18} do satisfy it.

Forman's discrete Morse theory \cite{MR1358614, MR1612391} introduces a combinatorial tool to simplify the cell decomposition of a given (regular) CW-complex up to homotopy equivalence, in terms of the critical cells of discrete Morse functions. Although this theory results in an efficient way to compute the homology of a regular CW-complex, it does not provide sufficient information to recover its homotopy type due to the lack of a combinatorial description of the simplified CW-complex (the \textit{Morse CW-complex}). Moreover, the procedure does not ensure the preservation of the simple homotopy type.
We extend the scope of Forman's theory 
to Whitehead deformations and simple homotopy classes.
Concretely, given a regular $n$-dimensional complex $K$ and a discrete Morse function on it, we construct an explicit and algorithmically computable
 cell decomposition of the Morse CW-complex and we prove that it $(n+1)$-deforms to $K$. 
We deduce a computational method for handling 3-deformations of 2-complexes and thus study  the  Andrews--Curtis conjecture from a new point of view.
We present an algorithm to obtain new presentations $Q^{**}$-equivalent to a given one without requiring to specify the exhaustive list of movements to transform one into the other.
 
Independently, in \cite{MR3320904, ellis2} the authors also used discrete Morse theory to present
an algorithm to describe a presentation of the fundamental group of a regular CW-complex. They applied it in a classification problem of prime knots, and also to compute the fundamental group of point clouds.

The article is organized as follows. In Section  \ref{section morse},
we present the refinement of discrete Morse theory in terms of internal collapses and Whitehead deformations with bounds in the dimension of the complexes involved.
In Section \ref{section Q**-transformations}, we deduce a method that associates to each presentation $\P$ and acyclic matching in a poset obtained from  $\P$, a presentation $\Q$ such that $\P\sim_{Q^{**}}\Q$.
In Section \ref{counterexamples}, we exhibit applications of our method
to investigate potential counterexamples to the Andrews--Curtis conjecture. Appendix \ref{appendix gordon} contains the proof of an intermediate result that is necessary in Section \ref{counterexamples}. 
The implementation in {\fontfamily{lmss}\selectfont SAGE} \cite{sagemath} of the algorithms of Section \ref{section Q**-transformations}, as well as its application to the potential counterexamples at Section \ref{counterexamples}, can be found at \cite{finite-spaces} as part of the package {\fontfamily{lmss}\selectfont Finite Topological Spaces}. An implementation in {\fontfamily{lmss}\selectfont GAP} \cite{GAP4} is also available at the package {\fontfamily{lmss}\selectfont Posets} \cite{GAP-posets}. An outline of the computational procedure is described in Appendix \ref{appendix code sage}.

{\bf Note.} Most of the results of this article appeared originally in the author's PhD Thesis \cite{phdthesis}. 

\section{Discrete Morse theory and Whitehead deformations}\label{section morse}

Discrete Morse theory was introduced by R. Forman  \cite{MR1358614, MR1612391} as a discrete approach to classical Morse theory for smooth manifolds.
It is based on a combinatorial notion of Morse functions for regular CW-complexes.
Critical cells of Morse functions on a CW-complex $K$ are linked to the number of cells in each dimension of a new CW-complex, \textit{the Morse CW-complex}, which is homotopy equivalent to $K$ \cite[Cor. 3.5]{MR1358614}.
Although discrete Morse theory is a relevant tool to obtain information about the homotopy type of a CW-complex, 
it does not give information about its Whitehead \textit{simple} homotopy type.
In this section
we present a refinement of the theory, showing that discrete Morse theory actually provides a method 
to simplify the cell structure of an $n$-dimensional complex through an $(n+1)$-deformation.

We briefly recall here the main concepts in simple homotopy theory and refer the reader to \cite{MR0362320, hog1993two} for 
a more complete exposition. All the CW-complexes in this article will be finite and connected.
Given  a CW-complex $K$  and  a subcomplex $L\leq K$, we say that $K$ 
\textit{elementary collapses} to $L$ (or $L$ \textit{elementary expands} to $K$) and denote it by $K\ce L$ (resp. $L\ee K$) if
$K=L\cup e^{n-1}\cup e^n$ with $e^{n-1}, e^n\notin L$ and 
there exists a map
$\psi:D^n\to K$ such that $\psi$ is the characteristic map of $e^n$,
$\psi|_{\overline{\partial D^n\smallsetminus D^{n-1}}}$ is the characteristic
map of $e^{n-1}$
and $\psi(D^{n-1})\subseteq L^{(n-1)}$, where $L^{(n-1)}$ is the $(n-1)$-skeleton of $L$.
In general, $K$ 
 \textit{collapses} to $L$ (or $L$ \textit{expands} to $K$) if
 there is a finite sequence of elementary collapses from $K$ to $L$. We denote it by
 $K\co L$ (resp. $L\ex K$).
We say that
a CW-complex $K$ \textit{$n$-deforms}
to $L$, and denote it by $K\ndef L$, if there is a 
sequence of CW-complexes $K=K_0,K_1,\dots K_r=L$ such that $K_i\ce K_{i+1}$ or $K_i\ee K_{i+1}$  for each $0\leq i\leq r-1$, and $\dim(K_i)\leq n$ for all $1\leq i\leq r$. For every $0\leq i\leq r-1$, there is a homotopy equivalence $f_i:K_i\to K_{i+1}$ which is an inclusion or a retraction depending on whether $K_i\ee K_{i+1}$ or $K_{i+1}\ce K_{i}$ respectively.
If $K$ $n$-deforms to $L$, then $K$ and $L$ are then related by a \textit{deformation} $f:K \to L$ defined as the composition of the retractions and inclusions as above, i.e, $f=f_{r-1}\dots  f_1 f_0$.
Notice that if $K\ndef L$, then $K$ and $L$ are homotopy equivalent. The converse is not true in general and this obstruction is measured by the Whitehead group \cite{MR0362320}.

We now outline the main definitions and results in discrete Morse theory. We refer the reader to \cite{MR1939695, MR2361455} for more details. A CW-complex  $K$ is said to be \textit{regular} if   for every open cell $e^n$, the characteristic map $D^n\to \overline{e^n}$
is a homeomorphism. A cell $e$ of a regular complex $K$ is a \textit{face} of a cell $e'$ if $e\subseteq \overline{e'}$. We  denote the face relation by $e\leq e'$.
Given a regular CW-complex $K$, a map $f : K \to  \RR$ is a \textit{discrete 
Morse function} if  for every cell $e^n$ in $K$, the number of faces and cofaces 
of $e^n$ for which the value of $f$ does not increase with dimension is at most one.
An $n$-cell $e^n \in  K^{(n)}$
is a \textit{critical cell of index $n$} if the values of $f$ in 
every face and coface of $e^n$ increase with dimension.
A discrete Morse function induces an ordering in the cells, which determines
level subcomplexes of $K$.
For every $c \in \RR$, the \textit{level subcomplex $K(c)$ of $K$} is the subcomplex 
of closed cells $\bar e$ of $K$ such that $f(e)\leq c$ in $\RR$.
Discrete Morse functions serve as a tool to study the homotopy type of $K$.
The following theorem summarizes the main results in discrete Morse theory.

\begin{theorem}\cite{MR1612391, MR1939695} \label{teo morse}
Let $K$ be a regular CW-complex and  let $f:K\to \RR$ be a discrete Morse function. Let $a < b$ be real numbers.
\begin{enumerate}[(a)]
 \item If  every cell $e\in K$ such that  $f (e) \in (a, b]$ is not critical, then 
$K(b)\co K(a)$.
\item If $e^n\in K$ is the only critical cell with
$f (e^n ) \in (a, b]$, then there is a continuous map 
$\varphi : \partial D^n \to K(a)$ such that
$K(b)$ is homotopy equivalent to $K(a) \cup_{\varphi} D^n$.
\item $K$ is homotopy equivalent to a CW-complex with exactly one cell of dimension $k$ for every critical cell of index $k$. 
\end{enumerate}
\end{theorem}

Given $K$ a regular CW-complex, denote by $\X(K)$ its \textit{face poset}, that is, the poset of cells of $K$ ordered by the face relation $\leq$. Let $\H(\X(K))$ be the \textit{Hasse diagram} of the poset $\X(K)$,  a digraph whose vertices are the cells of $K$ and whose edges are the ordered pairs $(e, e')$ such that $e< e'$ and there exists no $e''\in K$ such that $e< e''< e'$ (in that case, we say that $e'$ covers $e$ and we denote this by $e\prec e')$.

Every discrete Morse function $f:K\to \RR$ has an associated set $M_f$ of pairings of the cells of $K$,  where 
 $(e, e') \in M_f\text{ if and only if }e\prec e'\text{ and }
f (e) \geq f (e').$ Moreover, $M_f$ is an acyclic matching in $\X(K)$.
Recall that a pairing $M$ of cells in $K$  is said to be an \textit{acyclic matching}
if 
each cell of
$K$ is involved in at most one pair of $M$ and the 
directed graph $\H_M(\X(K))$ obtained by reversing  the orientation of the edges of $\H(\X(K))$ associated to 
matched pairs of cells is acyclic. 
In \cite{MR1766262}, M. Chari proved that a subset $C$ of cells of $K$ is the set of critical cells of a discrete Morse function $f$ on $K$ if and only if there is an acyclic matching $M$ in $\X(K)$ such that $C$ is the set of nodes of $\X(K)$ not incident to any edge in $M$.

We next present a series of results in simple homotopy theory as preliminary steps to our refinement of Theorem \ref{teo morse}. Our theory is built over the idea of 
\textit{internal collapses}, a generalization of the standard collapses (see \cite[Ch. 11]{MR2361455} and \cite[Ch. 11]{kozlov2021organized}).
We will show that internal collapses can be thought of as a way of performing $(n+1)$-deformations from a CW-complex of dimension less than or equal to $n$ to another one with fewer number of cells. This general notion  of collapse will be the key to a better understanding of the close connection between discrete Morse theory and Whitehead deformations of bounded dimension.

The cornerstone of the concept of internal collapses is the following fact. If $K\co L$ and we attach a cell $e$ to $K$, then we still have a deformation from $K\cup e$ to $L\cup \tilde e$, where $\tilde e$ is attached with an inherited attaching map (cf. \cite[Prop. 7.1]{MR0362320}).
When needed, we may emphasize the attaching map $\varphi: \partial D^n\to K$ of an $n$-cell by writing $K\cup e^n$ as $K\cup_\varphi D^n$.

\begin{lemma}\label{elementary internal}
 Let $K$ be a CW-complex of dimension less than or equal to $n$. Let $\varphi:\partial D^n\to K$ be the 
 attaching map of an $n$-cell $e^n$.
 If $K\co L$, then $K\cup e^n \nplusonedef L\cup \tilde e^n$, 
 where the attaching map $\tilde \varphi:\partial D^ n\to L$ of $\tilde e^n$ is defined as $\tilde \varphi=r \varphi$ with $r:K\to L$ the canonical
 strong deformation retract induced by the collapse $K\co L$. 
\end{lemma}

\begin{proof} Let $\imath:L\to K$ be the inclusion map and let $r:K\to L$ be the strong deformation retract induced by the collapse $K\co L$. There is a homotopy $H:\partial D^n \times I\to K$,
$\imath r\varphi \cong_H \varphi$ that allows to
perform the following  elementary moves
 \[K\cup_{\varphi}D^n \ee  
 \left(K\cup_{\varphi}D^n\right) \cup_{\imath r \varphi}D^n \cup_{H} D^{n}\times I 
 \ce  K\cup_{\imath r \varphi}D^n.\]
 Finally, the collapse $K\co L$ induces a collapse $ K\cup_{\imath r \varphi}D^n\co L\cup _{r \varphi} D^n$, since the image of the attaching map $\imath r \varphi$ is included in $L$.
 Given that the dimension of the complex $ \left(K\cup_{\varphi}D^n\right) \cup_{\imath r \varphi}D^n \cup_{H} D^{n+1}$ is $n+1$,
 we conclude that 
 \[K\cup_{\varphi} D^n \nplusonedef L\cup_{\tilde \varphi} D^n. \qedhere\] 
\end{proof}

The next result is a generalization of the previous statement. It says that  when one attaches a finite sequence of cells to a CW-complex that collapses to a subcomplex, it is possible to recover an iterative procedure for the deformation with bounds in the dimension of the involved complexes.

\begin{proposition}\label{internal collapse}
Let $ \displaystyle K \cup \bigcup_{i=1}^d e_i$ be a CW-complex where
$\dim (K) \leq \dim (e_{i})\leq  \dim (e_{i+1})\leq n$ for all $i=1, 2, \dots, d-1$.
Let $\displaystyle \varphi_j:\partial D_j\to K\cup \bigcup_{i<j} e_i$  be the 
 attaching map of $e_j$. 
 If $K\co L$, then there exist CW-complexes $Z_1\leq Z_2\leq \dots \leq Z_d$ of dimension less than or equal to $n+1$ such that for every $j=1, 2, \dots, d,$
 \[K\cup \bigcup_{i=1}^j e_i \ex Z_j \co L\cup 
 \bigcup_{i=1}^j \tilde e_i\] where the attaching map
$\displaystyle \tilde \varphi_j:\partial D_j\to L\cup \bigcup_{i<j}
\tilde e_i$  of the cell $\tilde e_j$ is defined inductively as follows:
$\displaystyle \tilde \varphi_1=r_0\varphi_1$  with $r_0:K\to L$ the canonical strong deformation retract and for $j>1$, $\tilde \varphi_{j}=\tilde r_{j-1}\imath_{j-1}\varphi_{j}$ where $\displaystyle \tilde r_{j-1}:Z_{j-1}\to L\cup \bigcup_{i<j} \tilde e_i$ is the strong deformation retract and $\imath_{j-1}:\displaystyle K\cup \bigcup_{i<j} e_i \to Z_{j-1}$ is the inclusion.
\end{proposition}

\begin{proof} We proceed by induction on the number of cells $d$. If $d=1$, the assertion follows by Lemma \ref{elementary internal}.
 Suppose that the statement holds for $d\geq 1$. Thus, we have the following sequence of expansions and collapses
 \[K\cup \bigcup_{i=1}^d e_i \ex Z_d \co L\cup 
 \bigcup_{i=1}^d \tilde e_i.\]
There is a strong deformation retract $\displaystyle \tilde r_{d}:Z_{d}\to L\cup 
 \bigcup_{i=1}^d \tilde e_i$ and an inclusion $\imath_{d}:\displaystyle K\cup \bigcup_{i=1}^d
  e_i \to Z
  _{d}$. Let $\displaystyle \varphi_{d+1}:\partial D_{d+1}\to K\cup \bigcup_{i=1}^d e_i$  be the 
 attaching map of the cell $e_{d+1}$.
Define 
 $\displaystyle \tilde \varphi_{d+1}:\partial D_{d+1}\to L\cup \bigcup_{i=1}^d
 \tilde e_i$ as 
$\tilde \varphi_{d+1}:=\tilde r_{d}\imath_d\varphi_{d+1}$.
If $\displaystyle \tilde \imath_{d}: L\cup 
 \bigcup_{i=1}^d \tilde e_i\to Z_d$ denotes the inclusion, then the maps $\imath_{d} \varphi_{d+1} $ and $\tilde \imath_{d}\tilde \varphi_{d+1} $ are homotopic.
Let $H:\partial D_{d+1}\times I\to Z_d$ be the homotopy $\imath_{d} \varphi_{d+1}\simeq_H \tilde \imath_{d}\tilde \varphi_{d+1} $. Define \[Z_{d+1} := Z_d \cup_{\imath_{d+1}\varphi_{d+1}} D_{d+1} \cup_{\tilde \imath_{d+1}\tilde \varphi_{d+1}} D_{d+1} \cup_H  (D_{d+1}\times I).\]
By definition,  there are elementary moves
\[Z_d \cup_{\imath_{d+1}\varphi_{d+1}} D_{d+1}  \ee  Z_{d+1}\ce Z_d \cup_{\tilde \imath_{d+1}\tilde \varphi_{d+1}} D_{d+1}.\]
Since $\displaystyle \Img(\imath_{d+1}\varphi_{d+1})\subseteq K \cup \bigcup_{i=1}^d e_i$, the collapse $\displaystyle Z_d\co K \cup \bigcup_{i=1}^d e_i$ induces a collapse $\displaystyle Z_d \cup_{\imath_{d+1}\varphi_{d+1}} D_{d+1}\co K \cup \bigcup_{i=1}^{d+1} e_i$.
Analogously, $\displaystyle Z_d \cup_{\tilde \imath_{d+1}\tilde \varphi_{d+1}} D_{d+1}\co L \cup \bigcup_{i=1}^{d+1} e_i$ and the result follows.
\end{proof}

\begin{corollary}\label{K def K/L}
Let $K$ be an $n$-dimensional CW-complex and let $L\leq K$ be a subcomplex of $K$ such that $L\co *$. Then $K\nplusonedef K/L$.
\end{corollary}

\begin{definition}
Under the conditions of Proposition \ref{internal collapse}, we say that there is a \textit{internal collapse} from
$\displaystyle K\cup \bigcup_{i=1}^d e_i$ to $\displaystyle L\cup 
 \bigcup_{i=1}^d \tilde e_i$.
\end{definition}

We next prove that the composition of a sequence of
 internal collapses in $n$-dimensional complexes is an $(n+1)$-deformation.

\begin{theorem}\label{several internal collapses}
Let $L$ be a CW-complex on dimension $n$. Let  $L_0\leq K_0\leq L_1\leq K_1\leq\dots \leq L_{N}\leq K_{N}\leq L_{N+1}=L$
be a sequence of CW-subcomplexes of $L$ such that 
$K_j\co L_j$ for all $j=0, 1,\dots N$. If $\displaystyle L_{j+1}=K_j\cup \bigcup_{i=1}^{d_j} 
e_i^j$, then there are internal collapses from 
$L_{j+1}$ to $\displaystyle L_{j}\cup \bigcup_{i=1}^{d_j} 
\tilde e_i^j$
 for all $j=0,1, \dots, N$ that
induce an $(n+1)$-deformation 
\[
L\nplusonedef
L_0\cup \bigcup_{j=0}^{N} \bigcup_{i=1}^{d_j} \tilde e_i^j.\footnote{Here, the attaching maps of the cells $\tilde e_i^j$ are induced by the internal collapses (see Proposition \ref{internal collapse}).}
\]
\end{theorem}
\begin{proof}
We proceed by induction on $N$. For $N=0$, the statement follows from Proposition \ref{internal collapse}. For $N\geq 1$, by inductive hypothesis 
 there is an $(n+1)$-deformation
  \[
\displaystyle L_0\cup \bigcup_{j=0}^{N-1} \bigcup_{i=1}^{d_j} \tilde e_i^j \nplusonedef L_N.
 \]
Moreover, there exist a CW-complex $Z$ of dimension $n+1$ such that 
 \[
\displaystyle L_0\cup \bigcup_{j=0}^{N-1} \bigcup_{i=1}^{d_j} \tilde e_i^j \ex Z\co L_N
 \]
(see for instance \cite[Ch.II]{MR0362320}). Suppose that $L_N\ex K_N \subseteq L_{N+1}$, where $\displaystyle L_{N+1}=K_N\cup \bigcup_{j=1}^{d_N} 
e_j^N$. 
By Proposition \ref{internal collapse}, the internal collapse from $L_{N+1}$ to $\displaystyle L_{N}\cup \bigcup_{i=1}^{d_N} 
\tilde e_i^N$ induces an $(n+1)$-deformation (where the attaching maps of $\tilde e_j^N$ are induced from the attaching maps of $e_j^N)$.
On the other hand, an argument analogous to the proof of Proposition \ref{internal collapse} shows that it is possible to construct a sequence of CW-complexes $Z\leq Z_1\leq Z_2\leq \dots \leq Z_{d_N}$ of dimension less than or equal to $n+1$ such that for every $d=1, 2, \dots, d_N$
 \[
 \displaystyle \left(L_0\cup \bigcup_{j=0}^{N-1} \bigcup_{i=1}^{d_j} \tilde e_i^j\right)\cup \bigcup_{i=1}^d
\tilde{\tilde{e}}_i^N \ex Z_j\co L_N \cup \bigcup_{i=1}^d 
\tilde  e_i^N
 \] where the attaching map $\displaystyle \dbtilde{\varphi}_d^N:\partial D_j\to L\cup \bigcup_{i<d}
\dbtilde{e}_i$ of the cell $\dbtilde{e}_d^N$ is induced by the attaching map of $\tilde e_d^N$.
 
We have proved that there is an $(n+1)$-deformation from $L_{N+1}$ to $\displaystyle \left(L_0\cup \bigcup_{j=0}^{N-1} \bigcup_{i=1}^{d_j} \tilde e_i^j\right)\cup \bigcup_{i=1}^{d_N} 
\dbtilde{e}_i^N$ (equivalently, to $\displaystyle L_0\cup \bigcup_{j=0}^{N} \bigcup_{i=1}^{d_j} \tilde e_i^j$ by abuse of notation).
\end{proof}

In what follows, we interpret discrete Morse theory in terms of internal collapses. In particular, we deduce a combinatorial method to simplify the cell structure of a regular CW-complex preserving its simple homotopy type, with bounds in the deformation.
Every discrete Morse function $f$ on 
a regular CW-complex $K$ can be described  as a sequence of internal collapses 
in its CW-structure as follows.

\begin{lemma}\label{matching - internal collapses}
 Let $K$ be a regular CW-complex. Then, $M$ is an acyclic matching in $\X(K)$ with unmatched set of cells $C$ if and only if there exist a sequence of subcomplexes of $K$
\begin{equation}\tag{$*_1$}\label{internal collapses sequence}
 K_0\leq L_1\leq K_1\dots \leq K_{N-1}\leq L_{N-1}\leq K_N = K
\end{equation}
 such that $K_j\co L_{j}$ for all $1\leq j\leq N$ and the set of cells of $K$ that was not collapsed in any of the collapses $K_j\co L_j$ is equal to $C$.
\end{lemma}

\begin{proof}
Given an acyclic matching $M$ in $\X(K)$, there is a linear extension $\L$ of the (order induced by the) directed acyclic graph $\H_M(\X(K))$ such that if $(e,e')\in M$, the cells $e,e'$ follow consecutively in total order $\L$ (see \cite[Thm. 11.2.]{MR2361455}). Let $\L_M$ be the ordering induced by $\L$ with the additional relations $e=_{\L_M}e'$ if $(e,e')\in M$.
It is easy to see that $f_M:K\to \RR$ defined by $f_M(e) = |\{e'\in K: e'\leq_{\L_M} e\}|$ is a discrete Morse function with critical cells the set $C$ of unmatched cells of $M$.
If $ c_1 \leq c_2\leq \dots \leq c_N\in \NN_0$ are the images under $f_M$ of the critical cells of $K$, define $L_j = K(c_j)$ for $1\leq j\leq N$, $K_{j-1} = K(c_{j}-1)$ for $1\leq j\leq N-1$ and $K_N = K$. Notice that $L_1\leq K_1\leq\dots \leq L_{N-1}\leq K_{N-1}\leq L_{N}\leq K_N$ is a sequence of CW-subcomplexes of $K$ and by Theorem \ref{teo morse},  $K_j\co L_j$ for all $1\leq j\leq N$. The set of cells of $K$ that are not involved in any of the collapses $K_{j}\co L_j$ is $\{e\in K(c_j)\smallsetminus K(c_j-1): 1\leq j \leq N\} = C$, the set of critical cells of $f_M$.

Conversely, given a sequence \eqref{internal collapses sequence} of subcomplexes of a regular CW-complex $K$, define $M$ as the set of pairings of cells $(e^{k-1},e^k)$ of $K$ associated to each elementary collapse in  $K_j\co L_{j}$ for all $1\leq j\leq N$.
Since each cell is involved in at most one elementary collapse,  $M$ is a matching in $\X(K)$. Suppose that there is a simple cycle 
\begin{equation}\tag{$*_2$}\label{cycle}
e^{k-1}_1\prec e^k_1\succ e^{k-1}_2\prec e^k_2\succ\dots
e^{k-1}_l\prec e^k_l = e^k_0 \succ e^{k-1}_1
\end{equation}
in $\H_M(\X(K))$ where $(e^{k-1}_i, e^k_i) \in M$ for all $1\leq i\leq l$.
Let $K_j$ be the minimal subcomplex of $K$ in the sequence  \eqref{internal collapses sequence} that contains all the cells of the cycle \eqref{cycle}. For every
pair of cells $(e^{k-1}_i, e^k_i)$ in  \eqref{cycle}, $e^{k-1}_i$ is also a face of  $e^{k}_{i-1}$, which in turn is not  part of any other elementary collapse along with a higher dimensional cell. So none of the pairs $(e^{k-1}_i, e^k_i)$  determines an elementary collapse in $K_j\co L_j$.  This contradicts the minimality of $K_j$. Hence, $M$ is an acyclic matching in $\X(K)$.
\end{proof}

\begin{remark}\label{decreasing dimension}
Given a regular CW-complex $K$ and an acyclic matching in $\X(K)$, the associated sequence of internal collapses \eqref{internal collapses sequence} in $K$ of Lemma \ref{matching - internal collapses} can be performed in a way that all the collapses involved are performed in decreasing dimension. That is, if $(e^{k-1}, e^k)$ and $(e^{k'-1}, e^{k'})$ are pair of cells involved in collapses $K_j\co L_j$ and $K_{j'}\co L_{j'}$ respectively with $j'\leq j$, then the dimensions $k, k'$ satisfy $k'\leq k$. This construction can be achieved by considering in the proof of Lemma \ref{matching - internal collapses} a linear extension $\L$ of the order induced by $\H_M(\X(K))$ such that it also respects the order given by the dimension (the existence of such $\L$ follows by a slight modification of proof of \cite[Thm. 11.2.]{MR2361455}).
\end{remark}

We next state the simple homotopy version of Theorem \ref{teo morse}, that provides an explicit construction of the 
Morse CW-complex and establishes bounds on the dimension of the deformation.

\begin{theorem}\label{teo morse deformacion}
 Let $K$ be a regular CW-complex of dimension $n$ and let $M$  be an acyclic matching in $\X(K)$. Then $M$ induces a sequence of subcomplexes of $K$
\[
 K_0\leq L_1\leq K_1\dots \leq K_{N-1}\leq L_{N-1}\leq K_N = K\]
 such that  $K_j\co L_{j}$ for all $1\leq j\leq N$ and the set of cells of $K$ non-collapsed in $K_j\co L_j ~\forall j$  corresponds to the unmatched cells in $M$. 
Moreover, 
 if $K_M$ is
the CW-complex obtained after performing in $K$ the sequence of internal collapses induced by $M$, then $K\nplusonedef K_M$. In particular, $K_M$ has exactly one cell of dimension $k$ for every unmatched cell of dimension $k$ in $M$. 
\end{theorem}

\begin{proof}
It is consequence of Lemma \ref{matching - internal collapses} and Theorem \ref{several internal collapses}.
\end{proof}

We focus now in the explicit combinatorial description of the deformation provided by  Theorem \ref{teo morse deformacion} in the case of 2-complexes.

\begin{remark}\label{combinatorial}
Notice that internal collapses transform 2-dimensional regular CW-complexes into combinatorial complexes. Recall
that a CW-complex $K$ of dimension 2 is called \textit{combinatorial} if
for each 2-cell $e^2$, its attaching map $\varphi:S^1\to K^{(1)}$, seen as a  cellular map by assigning a CW-structure on $S^1$, is a \textit{combinatorial map} (that is, $\varphi$ sends each open
1-cell of $S^1$ either homeomorphically onto an open 1-cell of $K$ or it collapses it to a 0-cell of $K$, see \cite[Ch.II]{hog1993two}). Thus, one
can think of the attaching map of a 2-cell in a combinatorial complex of dimension 2
just as the ordered list of  oriented 1-cells.
Suppose
that there is an elementary internal collapse from the combinatorial 2-complex 
$K \cup \bigcup_{i=1}^n e_i$ to $L \cup \bigcup_{i=1}^n \tilde{e}_i$, in which  $K\ce L$. We endow the CW-complex 
$L \cup \bigcup_{i=1}^n \tilde{e}_i$ with a combinatorial structure as follows.
If $K = L \cup \{x,v\}$, 
where $v$ is a 0-cell and $x$ a 1-cell, then for every 2-cell in $\bigcup_{i=1}^n e_i$ containing an edge $x^\epsilon$ with $\epsilon = 1$ or
$-1$ (where by $x^{-1}$ we mean the 1-cell $x$ traversed with the opposite orientation), modify its  attaching 
map by replacing each occurrence of $x^{\epsilon}$ by 1.
Similarly, if $K = L \cup \{x,e\}$ with $x$ a 1-cell and $e$ a 2-cell with attaching map
$x x_1 \dots x_r$, then
for every 2-cell in $\bigcup_{i=1}^n e_i$ containing edge $x^{\epsilon}$ modify its attaching 
map by replacing each occurrence of $x^{\epsilon}$ by $(x_1 \dots x_r)^{-\epsilon}$
and leaving the remaining cells unchanged.
As a result, after performing a sequence of internal collapses to 
a regular CW-complex $K \cup \bigcup_{i=1}^n e_i$ we obtain a CW-complex $L \cup \bigcup_{i=1}^n \tilde{e}_i$ with a natural combinatorial structure.
\end{remark}

\begin{example}[Torus]\label{example:torus}
 Let $K$ be the oriented regular CW-complex of Figure \ref{fig:torus}.
 \begin{figure}[htb!]
     \centering
     \def\svgwidth{0.28\textwidth}
     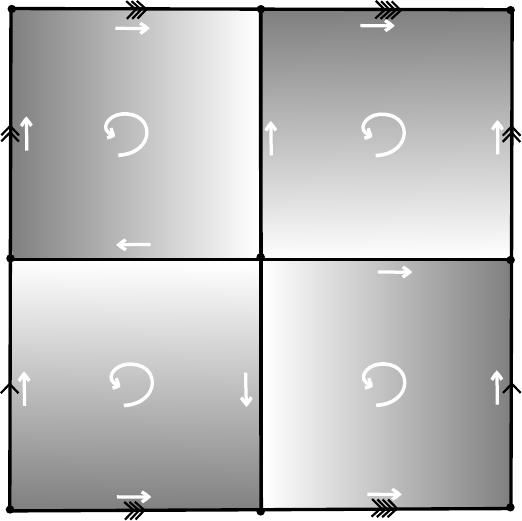
     \caption{CW-regular structure of the torus $T$. Arrows (in white) indicate the orientation of cells. }
     \label{fig:torus}
 \end{figure}
Let
$M$ be the acyclic matching in $\X(K)$ described in Figure \ref{fig:torus_matching}.
By Theorem \ref{teo morse}, the Morse CW-complex $K_M$ has one cell of dimension 0, two cells of dimension 1 and one cell of dimension 2. There is also a formula to compute the incidence number of critical cells of consecutive dimension by means of $M$ (see \cite{MR1612391} and \cite[Thm. 11.13.]{MR2361455}).
However, this information does not determine the homotopy type of $K_M$. In fact, the incidence number of the critical 2-cell into each of the critical 1-cells is 0. But the same number of cells in each dimension and incidence numbers can be observed in the minimal CW-structure of $S^1\vee S^1 \vee S^2$. 
We give next an explicit description of the attaching map of the 2-cell using Theorem \ref{teo morse deformacion}.

 \begin{figure}[h]
     \centering
    \begin{minipage}{0.5 \textwidth}
    \xymatrix@C=1.4em@R=4.1em{
    &&& \circ_{10}&\bullet_9&\bullet_8&\bullet_7\\
    &\bullet_9\ar@{-}[urr]
    \ar@{-}@*{[red]}[urrr]\ar@{-}@*{[red]}[urrr]\ar@{-}@*{[red]}[urrr]\ar@{-}@*{[red]}[urrr]\ar@{-}@*{[red]}[urrr]
    &\bullet_8 \ar@{-}[urr]
    \ar@{-}@*{[red]}[urrr]\ar@{-}@*{[red]}[urrr]\ar@{-}@*{[red]}[urrr]\ar@{-}@*{[red]}[urrr]\ar@{-}@*{[red]}[urrr]
    &\bullet_7\ar@{-}[urr]
    \ar@{-}@*{[red]}[urrr]\ar@{-}@*{[red]}[urrr]\ar@{-}@*{[red]}[urrr]\ar@{-}@*{[red]}[urrr]\ar@{-}@*{[red]}[urrr]
    &\bullet_2\ar@{-}[ul]\ar@{-}[urr]&
    \bullet_3\ar@{-}[ur]\ar@{-}[ull]
    &\circ_4 \ar@{-}[ull]\ar@{-}[ul]
    &\bullet_5\ar@{-}[ullll]\ar@{-}[ulll]
    &\circ_6\ar@{-}[ulll]\ar@{-}[ull]\\
    &&&\bullet_2\ar@{-}[ull]\ar@{-}[ul]\ar@{-}[u]\ar@{-}@*{[red]}[ur]\ar@{-}@*{[red]}[ur]\ar@{-}@*{[red]}[ur]\ar@{-}@*{[red]}[ur] &
    \bullet_3\ar@{-}[ulll]\ar@{-}[u]\ar@{-}@*{[red]}[ur]\ar@{-}@*{[red]}[ur]\ar@{-}@*{[red]}[ur]\ar@{-}@*{[red]}[ur]\ar@{-}@*{[red]}[ur]\ar@{-}[urr]
    &\bullet_5\ar@{-}[ulll]\ar@{-}[ul]\ar@{-}[urrr]\ar@{-}@*{[red]}\ar@{-}@*{[red]}[urr]\ar@{-}@*{[red]}[urr]\ar@{-}@*{[red]}[urr]\ar@{-}@*{[red]}[urr]&
    \circ_1\ar@{-}[u]\ar@{-}[urr]\ar@{-}[ur]\ar@{-}[ul]
    }
\end{minipage}
\hspace{5pt}
     \begin{footnotesize}
     \begin{minipage}{0.34 \textwidth}
    \def\svgwidth{\textwidth}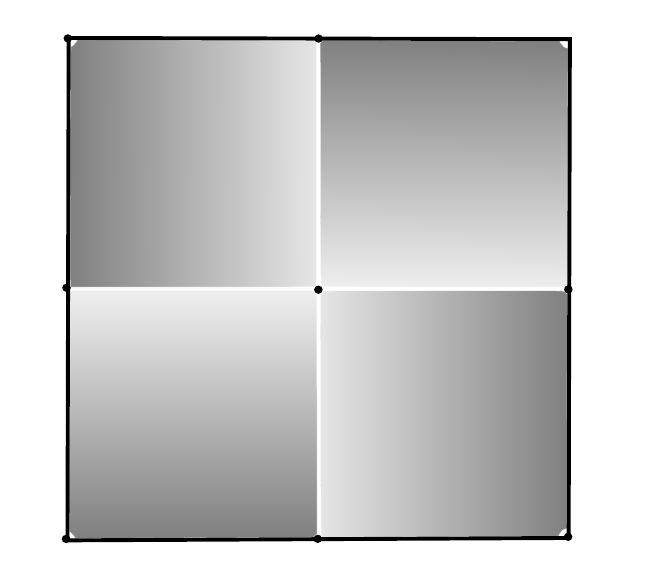
    \end{minipage}
    \end{footnotesize}

     \caption{ Left: Face poset associated to the regular CW-structure of the torus $T$, an acyclic matching $M$ in red and labels corresponding to an induced discrete Morse function $f_M$ (critical cells are the empty bullets). Right: Regular CW-structure of the torus $T$.  The arrows in red indicate pairs of cells in the matching  $M$.
     }
     \label{fig:torus_matching}
 \end{figure}

We denote by $v_i, x_i, e_i$ the 0-cells, 1-cells and 2-cells of $K$ respectively.
Initially, the attaching map of the critical 2-cell $e_{10}$ is \[x_3^{-1}x_5x_{2}^{-1}x_{9}.\] 
After the internal collapse associated to the pair $(x_{9}, e_{9})$ (or 9 for short), the attaching map of $\tilde e_{10}$ becomes \[x_3^{-1}x_5x_{2}^{-1}x_{8}x_5^{-1} x_{4}^{-1}\] (see Figure \ref{fig:torus_collapses} (a)). After performing  the internal collapses associated to the pairs $8$ and $7$, the induced attaching map of the 2-cell $\tilde e_{10}$ \footnote{By abuse of notation, we  denote by $\tilde e$ to all the intermediate stages of the cell obtained from a critical cell $e$ after performing a subsequence of internal collapses.} is 
\[x_3^{-1}x_{5}x_{2}^{-1}x_2x_6x_3x_4x_6^{-1}x_5^{-1}x_4^{-1}\]  (see Figure \ref{fig:torus_collapses} (b) and (c)). Now,  the internal collapses 5, 3, 2 associated to pairs of 0-cells and 1-cells mean that the corresponding edges are collapsed into a point, and they translate into the elimination of each occurrence of $x_2^{\epsilon}, x_3^{\epsilon}, x_5^{\epsilon}$ (with $\epsilon = \pm 1)$  from the description of the attaching map of $\tilde e_{10}$, that ends up being equal to \[x_6x_4x_6^{-1}x_{4}^{-1}.\]
See Figure \ref{fig:torus_collapses} for a picture of the complete procedure of deformation.

 \begin{figure}[htb!]
     \centering
    \begin{footnotesize}

    \begin{minipage}{0.32\textwidth}
        \centering
        \def\svgwidth{\textwidth}
        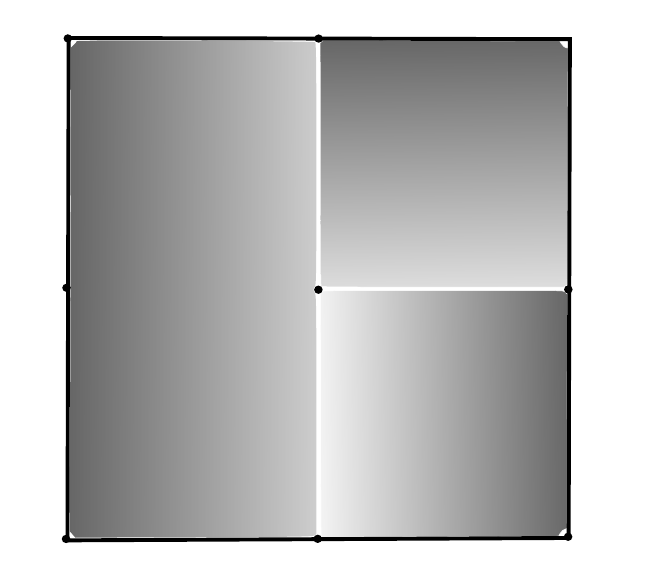
        (a)
    \end{minipage}
    \hspace{3pt}
        \begin{minipage}{0.32\textwidth}
        \centering
        \def\svgwidth{\textwidth}
        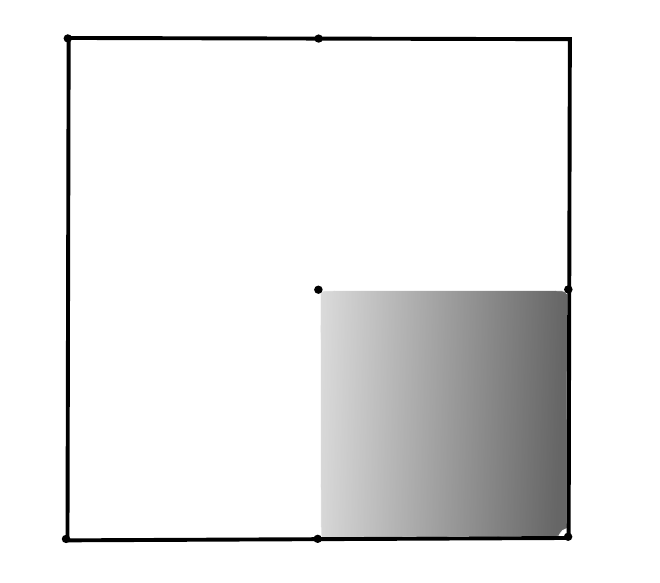
        (b)
        \end{minipage}
        \hspace{3pt}
        \begin{minipage}{0.32\textwidth}
        \centering
        \def\svgwidth{0.99\textwidth}
        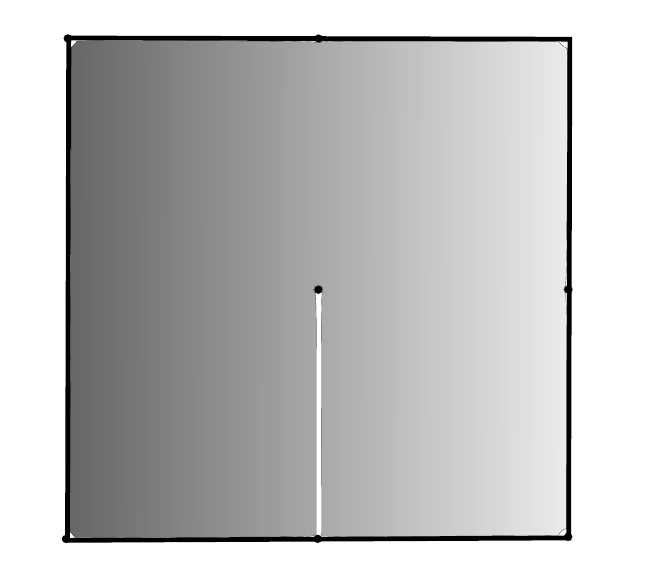
        \vspace{2pt}
        (c)
        \end{minipage}
        
    \vspace{10pt}
    
    \begin{minipage}{0.31\textwidth}
    \centering 
        \def\svgwidth{0.9\textwidth}
    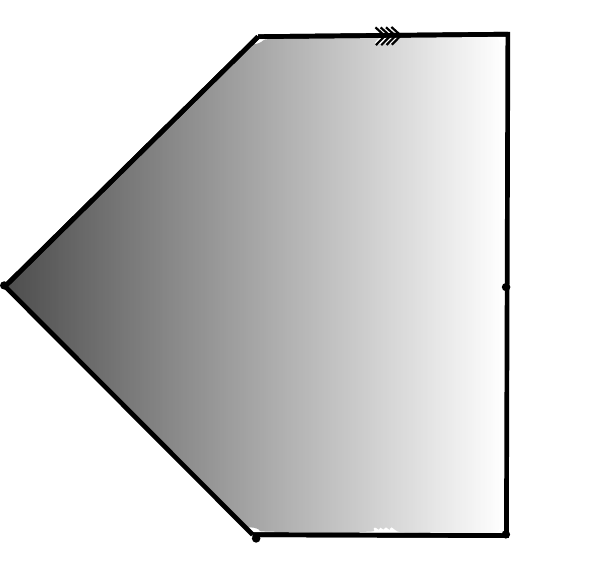
    (d)
    \end{minipage}
    \begin{minipage}{0.36\textwidth}
        \centering 
        \def\svgwidth{0.95\textwidth}
        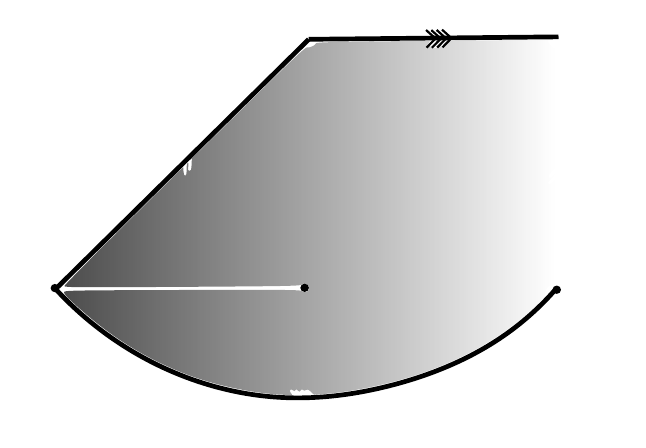
        \vspace{14pt}
        
        (e)
    \end{minipage}
    \begin{minipage}{0.3\textwidth}
        \centering 
        \def\svgwidth{0.7\textwidth}
         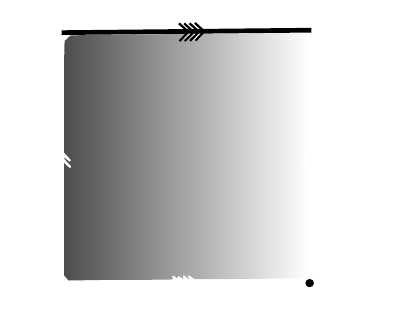
         
         \vspace{44pt}
        (f)
    \end{minipage}
    \end{footnotesize}
     \caption{ Internal collapses in the torus $T$. The CW-complex obtained after performing successively the internal collapses associated to the matched pairs 9 (a), 8 (b),  7 (c), 5 (d),  3 (e)  and  2 (f).}
     \label{fig:torus_collapses}
 \end{figure}

\end{example}

\begin{remark}\label{order internal collapses} Given a regular CW-complex $K$ of dimension 2 and an acyclic matching $M$ in $\X(K)$, the associated sequence of internal collapses \eqref{internal collapses sequence} 
can be performed in any order in the process of construction of $K_M$.
Concretely, if $(v,x)\in M$ is a pair of matched cells of dimension 0 and 1, the associated  internal collapse induces the replacement of $x^{\epsilon}$ by 1 in the combinatorial description of the attaching map of every  2-cell. This operation commutes 
with the transformation induced by  the rest of the internal collapses, since $x$ is not part of any other element of $M$.
Let $(x_1, e_1), (x_2, e_2)\in M$ be two pairs of matched cells of dimension 1 and 2. The attaching maps of $e_i$ in $K$ can be described as $x_i w_i$ where $w_i$ does not contain $x_i$.
It is impossible that both $w_1$ contains an occurrence of $x_2$ and $w_2$ an occurrence of $x_1$, since this would contradict the acyclicity of $\H_M(\X(K))$.
In particular, after performing any of the transformations induced by these internal collapses, the modified attaching map $\tilde e_i$ will still contain a single occurrence of $x_i$, $i=1,2$. So it is possible to perform these two transformations in any order, and they yield the same result.
An inductive application of this step completes the proof for an arbitrary number of transpositions.
\end{remark}

\iffalse
\begin{remark}(Higher dimensions)
\textcolor{magenta}{For higher dimensions, internal collapses also preserve some of the rigid combinatorial structure of regular CW-complexes. Indeed, they transform 
regular complexes into higher dimensional \textit{combinatorial complexes}, characterized by a discrete description of the attaching maps.
The attaching map $\phi:S^{n-1}\to K$ if a $n$-cell is a \textit{combinatorial map}  if for some CW-structure on $S^{n-1}$, the ????
The building block of internal collapses is the (multiple) composition of the (original) attaching maps of non-matched cells with the deformation retract $r:L\cup  $ induced by an elementary collapse $L\cup e^{n-1}\cup e^n\ce L$
Product cell
Recall that a \textit{combinatorial map} $f:K\to L$ is a cellular map for which every open $k$-cell of $K$ is either homeomorphically carried onto an open $k$-cell of $L$ or it is a \textit{product cell} that is collapsed by the map $f$ (that is, ????, see \cite[Ch.II]{hog1993two}).
The attaching maps of regular complexes are combinatorial maps and, in general, internal collapses preserve combinatorial attaching maps.
However, the algorithmic approach is limited in higher dimensions by the lack of a simple combinatorial description of the attaching map of $n$-cells for $n>2$. }
\end{remark}
\fi

\begin{example}[Higher dimensions] For dimension $n>2$, internal collapses also preserve some of the rigid combinatorial structure of regular CW-complexes. Indeed, they transform 
regular complexes into higher dimensional \textit{combinatorial complexes}, characterized by a discrete description of the attaching maps (see \cite[Ch.II]{hog1993two} for the definitions).
A general algorithmic description of the Morse CW-complex in high dimensions is, however, limited by the lack of a simple combinatorial description of the attaching map of $n$-cells for $n>2$. Notwithstanding, the Morse CW-complex can be still described in  particular examples.

For instance, let $K$ be the 3-dimensional regular CW-complex  constructed as two solid pyramids $K_1$ and $K_2$ whose boundary is jointly identified as in Figure \ref{fig:complex_dim3}. Each pyramid can be seen as the cone on a regular CW-structure $P$ of the projective plane $\RR P^2$.
%Denote by $e_1, \dots, e_4$ and $e_5, \dots, e_8$ the `internal' 2-cells in $K_1$ and $K_2$ respectively with vertices  $\{v_1, v_4, v_i\}$ for $i=2,3$. Let $e_{10}, \dots, e_{12}$ be the `external' 2-cells  with vertices  $\{v_1, v_2, v_3\}$, and let $e_{13}, \dots, e_{16}$ be the 2-cells in the basis of the pyramids with vertices $\{v_2, v_3, v_4\}$.
%We finally denote by $b_1, \dots, b_4$ and  $b_5, \dots, b_8$ the 3-cells in $K_1$ and $K_2$ respectively.
\begin{figure}[htb!]
    \centering
        \def\svgwidth{0.7\textwidth}
    \begin{footnotesize}
         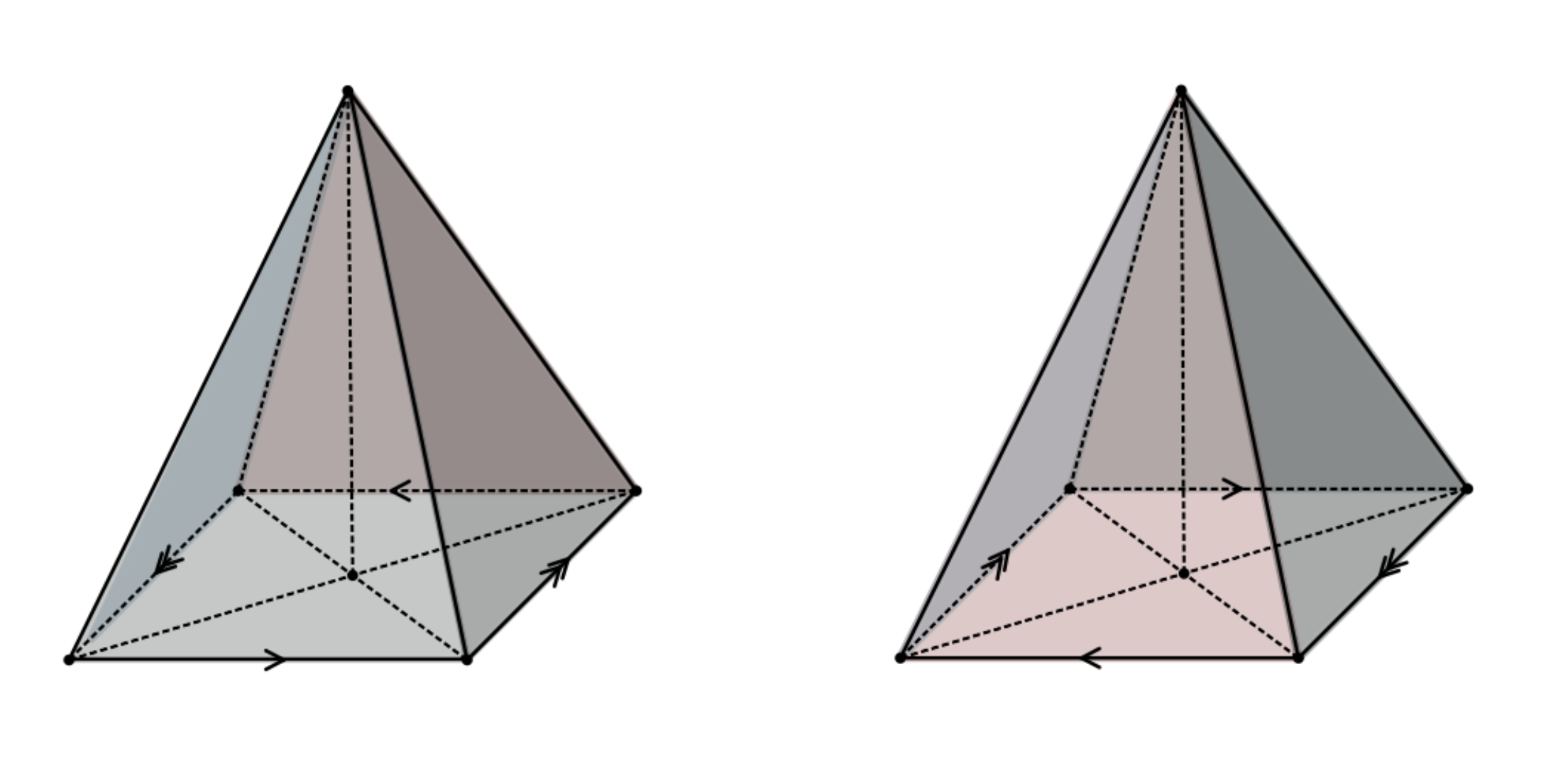
    \end{footnotesize}
    \def\svgwidth{0.7\textwidth}
    \begin{footnotesize}
         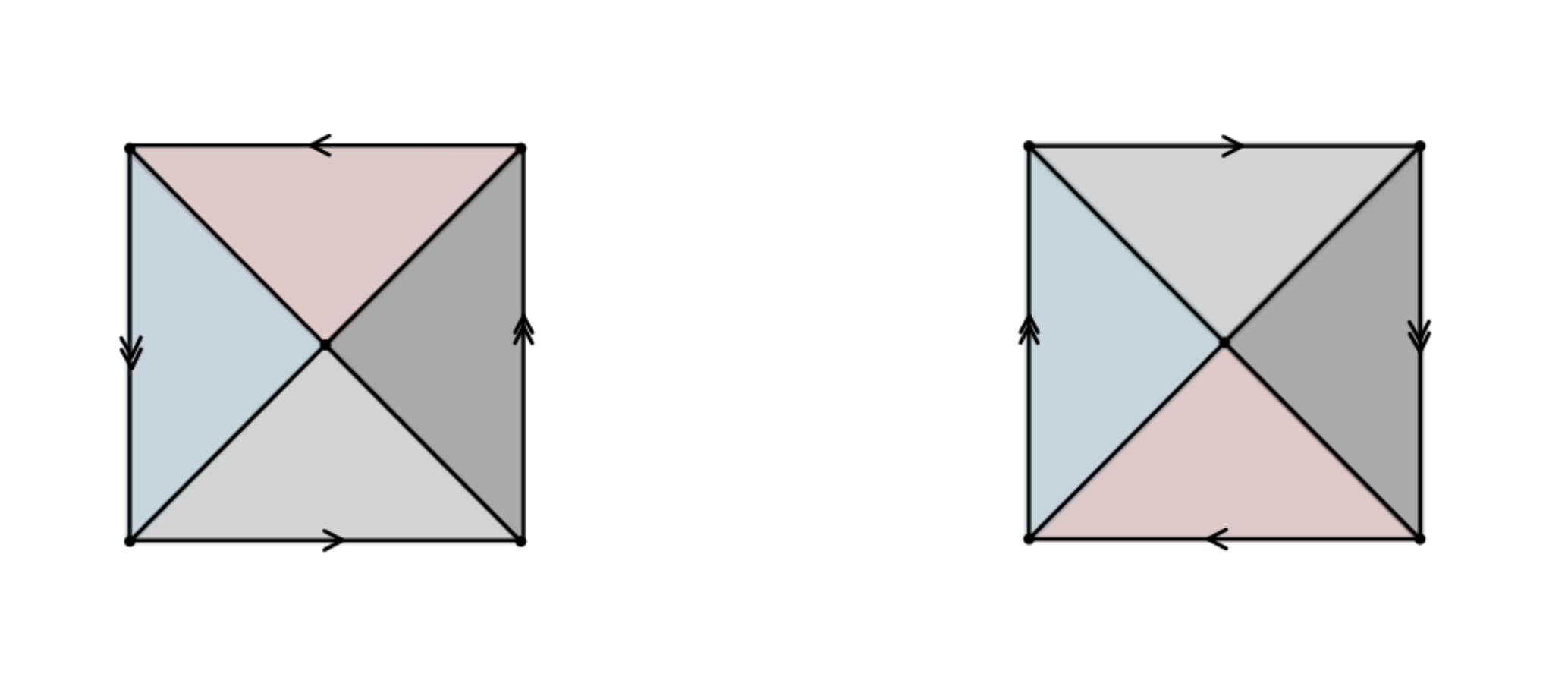
    \end{footnotesize}
    \caption{Top: Two subdivided solid pyramids $K_1$ and $K_2$ whose boundaries are jointly identified. Each pyramid is the cone on a regular CW-structure $P$ of the projective plane $\RR P^2$. Bottom: The top view of $K_1$ and $K_2$ showing the identification of its `top' cells.}
    \label{fig:complex_dim3}
\end{figure}

Let $M$ be the acyclic matching in $\X(K)$ described at Figure \ref{fig:matching_3dim}.
The associated internal collapses, depicted in Figure \ref{fig:internal_collapses_dim3}, simplify the 3-skeleton of $K$ to obtain a CW-complex with a single (critical) 3-cell whose attaching map can be concretely described as follows.
Let  $S$ be the triangulation of $S^2$ induced by the construction of $K$ illustrated at Figure \ref{fig:attaching_map_dim3}. The attaching map of the critical 3-cell in the Morse CW-complex $K_M$ is the quotient of the canonical homeomorphism $\varphi:S^2\to S$ under the identification in $S$ of antipodal points (via the homeomorphism). This construction shows a simple CW-structure of the 3-dimensional real projective space $\RR P^3$.
A minimal CW-structure of $\RR P^3$ (with a single 0-cell $v_2$, a 1-cell $x_{10}$, a 2-cell $e_5$ and a 3-cell $b_8$) is obtained after further performing the internal collapses associated to the matched cells in solid black at Figure \ref{fig:matching_3dim}.

\begin{figure}[htb!]
\centering

 \def\svgwidth{\textwidth}
    \begin{footnotesize}
         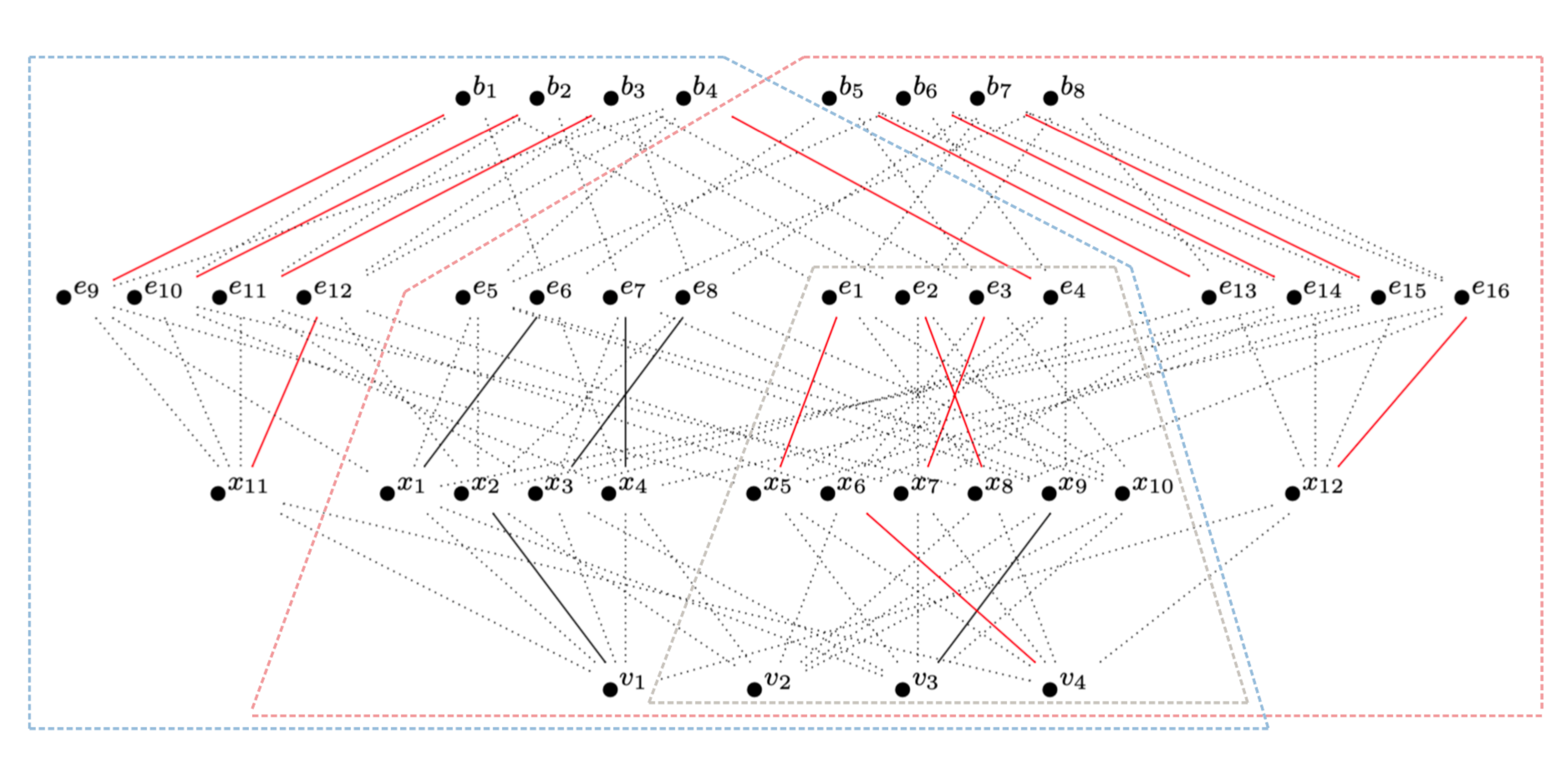
    \end{footnotesize}
    \caption{Acyclic matching $M$ in $\X(K)$ in solid red (and a maximal acyclic matching $M_{\max}$ in solid red and black). In dashed lines, the subposets $\X(K_1)$, $\X(K_2)$ and $\X(P)$ with $P$ the regular CW-structure of the projective plane $\RR P^2$ in the basis of $K_1$ and $K_2$.\\
    Here,  $e_9, \dots, e_{12}$ and $e_{13}, \dots, e_{16}$ denote the `internal' 2-cells in $K_1$ and $K_2$ respectively with vertices  $\{v_1, v_4, v_i\}$ for $i=2,3$. The 3-cells in $K_1$ and $K_2$ are denoted by  $b_1, \dots, b_4$ and  $b_5, \dots, b_8$ respectively. The 1-cells $x_{11}$ and $x_{12}$ joins the vertices $v_1$ and $v_4$ in $K_1$ and $K_2$ respectively. The `internal' 1-cells in the CW-structure $P$ of the projective plane in the basis of $K_1$ and $K_2$ are labelled as $x_7, \dots, x_{10}$.}
    \label{fig:matching_3dim}
\end{figure}

\begin{figure}[htb!]
    \centering
    \def\svgwidth{0.95\textwidth}
    \begin{scriptsize}
         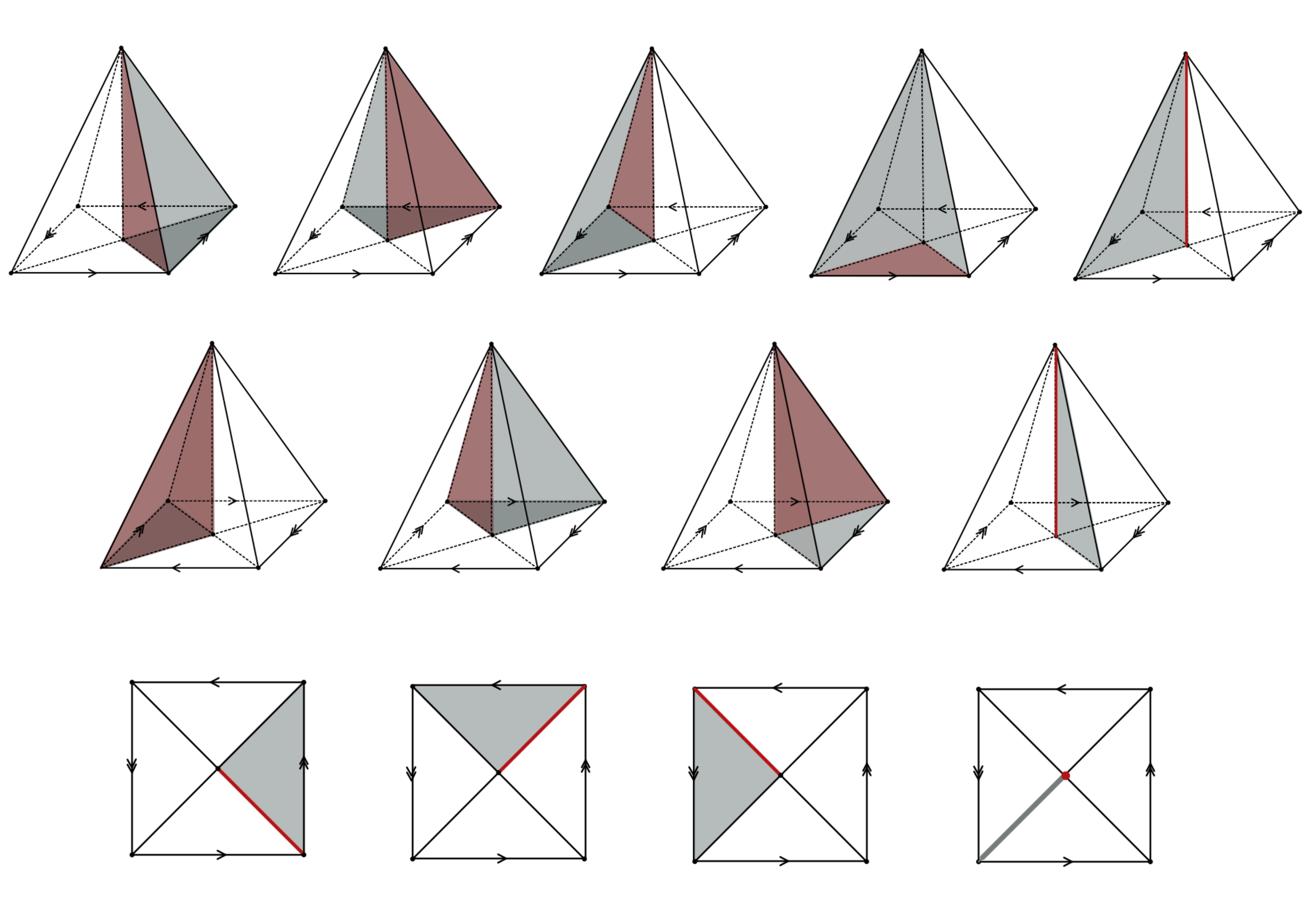
    \end{scriptsize}
    \caption{Internal collapses in $K$ determined by the acyclic matching $M$. Here $P$ is the CW-structure of projective plane $\RR P^2$ at the basis of the pyramids.}
    \label{fig:internal_collapses_dim3}
\end{figure}

\begin{figure}[htb!]
    \centering
    %\includegraphics[width = 0.5\textwidth{figures/projective_dim3.png}
    %\def\svgwidth{0.7\textwidth}
    %\begin{footnotesize}
    %     \input{projective_3dim.pdf_tex}
    %\end{footnotesize}
    \def\svgwidth{0.65\textwidth}
    \begin{footnotesize}
         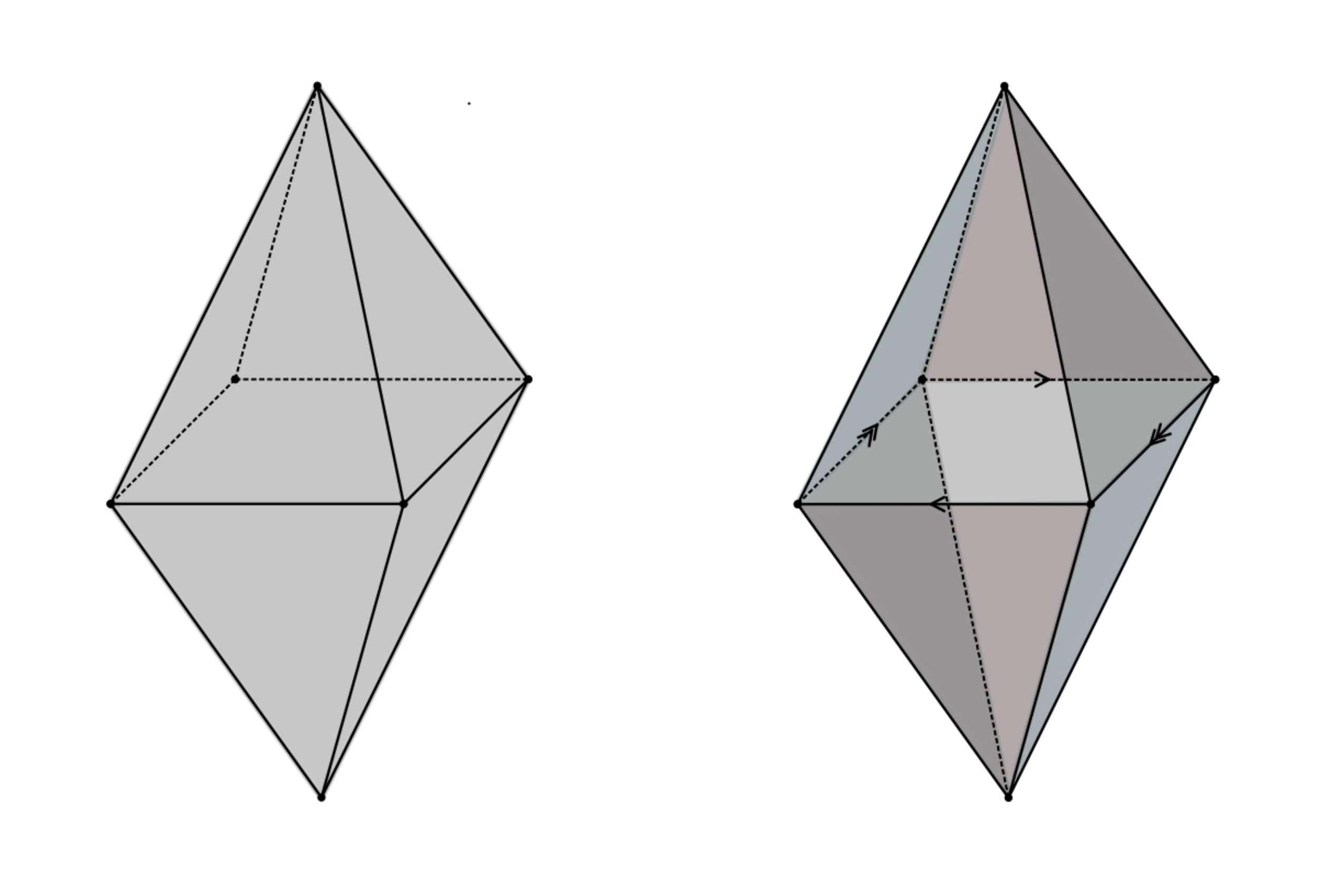
    \end{footnotesize}
    \caption{Left: A triangulation $S$ of $S^2$ (the boundary of $D^3$). Right: The Morse CW-complex $K_M$ obtained as the identification of antipodal points in $S\cup_\varphi D^3$, with $\varphi:\partial D^3\to S$ the canonical homeomorphism.
    }
    \label{fig:attaching_map_dim3}
\end{figure}

\end{example}

\section{ $Q^{**}$-transformations of group presentations} \label{section Q**-transformations}

In this section, we 
present an algorithm based in discrete Morse theory to obtain finite presentations that are $Q^{**}$-equivalent to a given one without tracking down explicitly the $Q^{**}$-transformations involved. The method has a geometric core based in the connection between group presentations and 2-dimensional CW-complexes \cite[Ch.I]{hog1993two}.

In the same way as the collapses and expansions determine simple homotopy classes  of
CW-complexes, the transformations (1)--(5) define classes of group presentations, the
\textit{$Q^{**}$-classes}. There is a correspondence between 3-deformation classes of CW-complexes of dimension 2 and $Q^{**}$-classes of group presentations. Such correspondence is based on a standard way to map a given complex $K$ to an associated presentation $\P_{K}$, and a presentation $\P$ to a 2-complex $K_\P$.
Concretely, if $\P=\langle x_1,\dots,x_n ~|~  r_1,\dots,r_m\rangle$ is a finite group presentation, 
its associated \textit{standard complex} $K_\P$ is a
CW-complex with a single vertex and an oriented 1-cell $e_i^1$ for each generator $x_i$. The 2-cells $e_j^2$ of $K_\P$ correspond to the relators $r_j$, which determine closed edge paths on the 1-skeleton as the attaching maps.
Conversely, if $K$ is a finite CW-complex of dimension 2 and $T$ is a 
spanning tree of $K^{(1)}$ (the 1-skeleton of $K$), then $K\threedef K/T$ and the \textit{standard presentation}
$\P_K$ is defined as a presentation of the fundamental group of $K/T$. Different
choices of spanning trees result in $Q^{**}$-equivalent presentations. 
It is important to note that $\P_K$ contains not only information about the fundamental group of $K$, but also about its 3-deformation class. 
Given group presentations $\P,\Q$, if
$\P$ can be transformed into  $\Q$ by a finite 
sequence of $Q^{**}$-transformations, we say that $\P$ is \textit{$Q^{**}$-equivalent} to
$\Q$ and we denote  $\P\sim_{Q^{**}} \Q$.
It can be shown that $\P\sim_{Q^{**}} \Q$
if and only
$K_{\P}\threedef K_{\Q}$. Similarly, if $K, L$ are CW-complexes of dimension 2,
$K\threedef L$ if and only if $\P_K \sim_{Q^{**}} \P_L$. Moreover, $\P\sim_{Q^{**}} \P_{K_{\P}}$ and $K\threedef K_{\P_{K}}$.
Balanced presentations of the trivial group correspond to contractible complexes.

Theorem \ref{teo morse deformacion} implies the following result in terms of group presentations when the dimension of the CW-complex is 2.

\begin{corollary}\label{morse AC}
Let $K$ be a  regular  CW-complex of dimension 2 and let $M$ be 
 an acyclic matching in $K$. 
Then, $\P_K\sim_{Q^{**}}\P_{K_M}$.
\end{corollary}

We will show next an algorithm to construct $\P_{K_M}$ for any matching $M$ in $\X(K)$. 
We start with the case when the matching $M$ only pairs cells of dimension 0 and 1.

\begin{lemma} \label{matching tree}
 Let $K$ be a regular  CW-complex  of dimension $n$ and let $M$ be a matching in the subposet of $\X(K)$ of cells of 
 dimension 0 and 1 with only one critical cell of dimension 0. Let $T$ be the subcomplex of $K$ of the closed matched cells
 $T= \bigcup_{e \in M} \bar e$.  Then,
$M$ is acyclic if and only if $T$ is a spanning 
tree in the 1-skeleton $K^{(1)}$.
\end{lemma}

\begin{proof} Recall that given an acyclic matching $M$ in the Hasse diagram of $\X(K)$, the graph $\H_M(\X(K))$ is the 
(acyclic) directed graph obtained by reversing  in $\H(X)$ the orientation of the edges corresponding to pairs in $M$.
Since the induced subgraph of $\H(\X(K))$ whose vertices are the cells of dimension 0  and 1
 is the barycentric subdivision of the 1-skeleton of $K$, the result follows directly from the following fact:
 any simple cycle \[v_1\prec x_1\succ v_2\prec x_2\succ\dots
v_k\prec x_k \succ v_1\] in $\H_M(\X(K))$, with $v_i$ of dimension 0 and $x_i$ of dimension 1,
is in correspondence with a simple cycle \[v_1,v_2,\dots,v_k,v_1\] in the 1-dimensional 
subcomplex of $K$ of matched pairs of cells of dimensions 0 and 1. 
\end{proof}

As a consequence of the previous lemma and Theorem \ref{teo morse deformacion}, we obtain an alternative proof of the fact that an $n$-dimensional CW-complex $K$ $(n+1)$-deforms to $K/T$, where $T$ is a spanning tree of $K^{(1)}$. Indeed, it reduces to noting that $K/T$ is homeomorphic to $K_M$, where $M$ is the acyclic matching induced by $T$. 

We  now provide a simple description of the group presentation associated to $K_M$ for any acyclic matching $M$,
which will be easily tractable through computer assistance.
We need first some definitions.

\begin{definition} [Rewriting rule]
Let $\P$ be a group presentation and let $r$ be a relator of $\P$ given by the word 
$w_1x^{\epsilon}w_2$, where $w_1$ and $w_2$ are words on the generators, the generator $x$ appears neither in
$w_1$ nor in $w_2$ and $\epsilon=+1$ or $-1$. The \textit{equivalent expression of $x$ induced by $r$} 
is defined as $(w_1^{-1}w_2^{-1})^{\epsilon}$.
\end{definition}

\begin{remark} \label{isolate}
If $\P$ is a group presentation and 
$x$ is a generator of $\P$ such that it appears only once in a relator $r$, then $\P$ is $Q^{**}$-equivalent to the presentation obtained after eliminating the generator $x$ and  the relator $r$ and applying the rewriting rule to every occurrence of $x$ in the rest of the relators. Indeed, if $r'$ is another relator containing $x$, by cyclically permuting $r'$ if necessary (this is operation (3)), we can assume $r'$ reads as $x^{\epsilon}u$, with $\epsilon = 1$ or $-1$. We replace $r'$ by the product $sr'$ where $s$ is a suitable cyclic permutation of $r$ (or its inverse) to eliminate this occurrence of $x$ (here we apply operations (1), (2) and (3)). We iterate this procedure until no occurrence of $x$ is left.
\end{remark}
\begin{example}\label{case n=1}
If $\P=\langle x,y~ |~ xyxy^{-1}x^{-1}y^{-1}, x y^{-2}\rangle$, then the equivalent expression
of $x$ induced by the relator $r=xy^{-2}$ is $y^2$.
By Remark \ref{isolate}, $\P$ is $Q^{**}$-equivalent to the presentation
$\tilde \P=\langle y ~ | ~ 
y^2yy^2y^{-1}y^{-2}y^{-1}\rangle$, i.e. $\langle y ~ | ~ 
y\rangle$.
\end{example}

\begin{definition}\label{pres matching}
 Let $K$ be a regular CW-complex of dimension 2.
 Let $M$ be an acyclic matching in $\X(K)$ such that 
 there is only one critical cell of dimension 0. 
   Denote by $M_0$  the subset of matched pairs of cells of dimension 0 and 1 and by $M_1=\{(x_1, e_1), \dots, (x_m, e_m)\}$  the subset of matched pairs of cells of dimension 1 and 2. 
 The \textit{Morse presentation} $\Q_{K,M}$ is the presentation $\Q_m$ defined by the following iterative procedure:
\begin{itemize}
 \item  $\Q_0$ is the standard presentation $\P_K$ constructed using 
 the spanning tree $T$  induced by $M_0$ (see Lemma \ref{matching tree}).
 The generators of $Q_0$ are the unmatched 1-cells of $K$ with respect to the matching $M_0$, and its relators are the words induced by the attaching maps of the 2-cells of $K$. We associate to each relator $r$ in $Q_0$ the 2-cell that produced it.
 \item For $0\leq i <m$, let $\Q_{i+1}$ be the presentation obtained from $\Q_{i}$ after removing the relator associated to $e_i$ and the generator $x_i$, and applying the rewriting rule on every occurrence of the generator $x_i$ in the rest of the relators.
\end{itemize}
\end{definition}

We will see next that the Morse presentation $\Q_{K, M}$ 
is an algorithmic description of $\P_{K_M}$. Concretely, there is a correspondence between the internal collapses in $K$ induced by $M$ and the combinatorial transformations in $\P_K$ described in Definition \ref{pres matching}. The acyclicity of $M$ implies that  the construction of $\Q_{K, M}$ does not depend on the choice of the order in $M_1$; that is, the internal collapses determined by $M_1$ can be performed in any order (see Remark \ref{order internal collapses}).

\begin{proposition}\label{presentation K_M}
 Let $K$ be a regular 2-dimensional CW-complex, and let $M$ be  an acyclic matching in $\X(K)$
 with only one critical cell of dimension 0. Then, $\Q_{K, M} = \P_{K_M}$ 
 for a suitable choice of orientations and basepoints in $K_M$.
\end{proposition}

\begin{proof}  Given an acyclic matching $M$ in $\X(K)$, by Lemma  \ref{matching - internal collapses} and Remarks \ref{decreasing dimension} and \ref{order internal collapses} there is a sequence of internal collapses \eqref{internal collapses sequence} of decreasing dimension such that the collapses at each dimension can be performed in any order. 
If $M_0$ and $M_1 = \{(x_1, e_1), \dots, (x_m, e_m)\}$ are as in Definition \ref{pres matching}, then they describe an (ordered) sequence of internal collapses to transform $K$ into $K_M$. For each $1\leq i <m$, the combinatorial transformation performed to get $\Q_{i+1}$ from $\Q_i$ parallels exactly the geometric description of the elementary internal collapse indicated by the pair $(x_i, e_i)$ (see Remark \ref{combinatorial}).
  Denote by $K_{M_1}$ the CW-complex obtained from $K$ after performing the
  internal collapses induced by the pairs in $M_1$. Now, let $T$ be the spanning tree induced by $M_0$. Since $K_M = K_{M_1} / T$, it
  follows that $\Q_m = \Q_{K, M}$ is the standard presentation associated to
  $K_M$ for the right choice of orientations and basepoints.
\end{proof}

By Corollary \ref{morse AC}, 
Proposition \ref{presentation K_M} implies that $Q_{K,M}$ is another
representative of the
$Q^{**}$-class of $\P_K$.

\begin{theorem} \label{matching pres} Let $K$ be a regular 2-dimensional CW-complex, and let $M$ be 
 an acyclic matching
  in $\X(K)$ with only one critical cell of dimension 0. Then, 
$\Q_{K,M}\sim_{Q^{**}} \P_{K}$.
\end{theorem}

We deduce next a
combinatorial
technique to study the $Q^{**}$-class of a given presentation $\P$, by means of acyclic matchings.

\begin{theorem} \label{teo morse presentation} Let $\P$ be a finite presentation of a finitely presented group and  $K'_{\P}$ the barycentric subdivision of $K_{\P}$.
Let $M$ be an acyclic matching in $\X(K'_{\P})$ with only one critical cell of dimension 0.
Then $\P\sim_{Q^{**}}\Q_{K'_{\P},M}$.
\end{theorem}

\begin{proof} 
The result follows from the sequence of $Q^{**}$-equivalences
\[\P\sim_{Q^{**}}\P_{K_{\P}}\sim_{Q^{**}}\P_{K_{\P}'}
\sim_{Q^{**}}\P_{(K_{\P}')_M}=\Q_{K_{\P}',M},\]
where the equivalence $\P\sim_{Q^{**}}\P_{K_{\P}}$ is well known (see for instance \cite{MR0380813}), $\P_{K_{\P}}\sim_{Q^{**}}\P_{K_{\P}'}$ holds since every 2-complex 3-deforms to its barycentric subdivision, and  $\P_{K_{\P}'}\sim_{Q^{**}}\P_{(K_{\P}')_M}=\Q_{K_{\P}',M}$ is consequence of Theorem \ref{matching pres} applied to the regular 2-dimensional CW-complex $K'_{\P}$.
\end{proof}

\begin{example}
Let $\P=\langle x,y ~|~x^2,  xy^{-2}\rangle$. The poset $\X(K'_{\P})$ associated to $\P$ is depicted in Figure \ref{fig:face poset}. This poset admits an explicit algorithmic description from $\P$. If $o$ is the minimal point that represents the only vertex of $K_\P$, for every generator $g$ of $\P$ there is a subposet of $\X(K_\P')$, a model of $S^1$, consisting of four points $o, g, g^{+1}, g^{-1}$.  We refer to these points  as \textit{generator elements}.
For every relator $r_i$ of $\P$, there is a subposet of $\X(K_\P')$, a model of $D^2$, given by a cone with apex $v_r$ over a model of $S^1$ that represents the subdivision of the 2-cell associated to $r_i$ in $K_\P$. We refer to the points of this subposet as \textit{relator elements}.
Finally, there are cover relations between the generator elements and the relator elements. For each $g^{\epsilon}$ in $r_i$ there is an edge from the generator elements $g$, $g^{+1}$, $g^{-1}$ to the homonym relator elements  $g$, $g^{+1}$, $g^{-1}$.
Note that $\X(K'_{\P})$ has 29 points and, in general, $|\X(K'_{\P})|=4(l+m)+1$ where $l$ is the total relator length and $m$ the number of relators of $\P$.

\begin{figure}[htb!]
\begin{minipage}{0.55\textwidth}
    \begin{scriptsize}
     \def\svgwidth{\textwidth}
    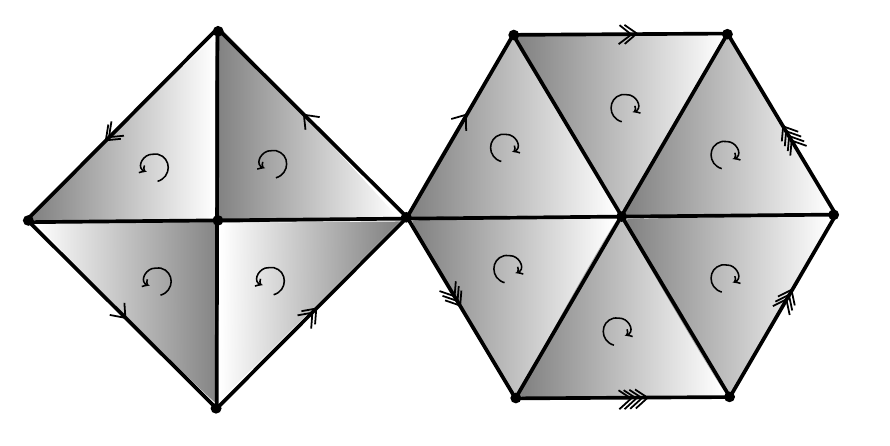
    \end{scriptsize}
\end{minipage}
\centering
\xymatrix@C=0.0001em@R=3em{
\bullet_{x^{+1}}&\bullet_{x^{-1}}&\bullet_{x^{+1}}&\bullet_{x^{-1}}&&&&&&
\bullet_{x^{+1}}&\bullet_{x^{-1}}&\bullet_{y^{-1}}&\bullet_{y^{+1}}&\bullet_{y^{-1}}&\bullet_{y^{+1}}\\
\bullet_o\ar@{-}[u]\ar@{-}[urrr]&
\bullet_{x}\ar@{-}[u]\ar@{-}[ul]&
\bullet_o\ar@{-}[u]\ar@{-}[ul]&
\bullet_{x}\ar@{-}[u]\ar@{-}[ul]&
%here starts the part related to the generators
\bullet_{x^{+1}}\ar@{-}[ull]\ar@{-}[ullll]\ar@{-}[urrrrr]&
\bullet_{x^{-1}}\ar@{-}[ull]\ar@{-}[ullll]\ar@{-}[urrrrr]&
&\bullet_{y^{-1}}\ar@{-}[urrrr]\ar@{-}[urrrrrr]
&\bullet_{y^{+1}}\ar@{-}[urrrr]\ar@{-}[urrrrrr]&
%here starts the secord relator
\bullet_o\ar@{-}[u]\ar@{-}[urrrrr]&
\bullet_{x}\ar@{-}[u]\ar@{-}[ul]&
\bullet_o\ar@{-}[u]\ar@{-}[ul]&
\bullet_{y}\ar@{-}[u]\ar@{-}[ul]&
\bullet_o\ar@{-}[u]\ar@{-}[ul]&
\bullet_{y}\ar@{-}[u]\ar@{-}[ul]\\
%0-cells
&\bullet_{v_{r_1}}\ar@{-}[u]\ar@{-}[ul]\ar@{-}[urr]\ar@{-}[ur]& &&&
\bullet_{x}\ar@{-}[ull]\ar@{-}[ullll]
\ar@{-}[urrrrr]\ar@{-}[u]\ar@{-}[ul]&
\bullet_o\ar@{-}[ullll]
\ar@{-}[ullllll]
\ar@{-}[urrr]
\ar@{-}[urrrrr]
\ar@{-}[urrrrrrr]
\ar@{-}[ull]\ar@{-}[ul]\ar@{-}[ur]\ar@{-}[urr]&
\bullet_{y} \ar@{-}[urrrrr]\ar@{-}[urrrrrrr]\ar@{-}[u]\ar@{-}[ur]
&&&&\bullet_{v_{r_2}}\ar@{-}[u]\ar@{-}[ul]\ar@{-}[ur]\ar@{-}[ull]\ar@{-}[urr]\ar@{-}[urrr]
}
    \caption{Top: The barycenter subdivision of $K_\P$. Bottom: The face poset $\X(K'_\P)$.}
    \label{fig:face poset}
\end{figure}

Consider an acyclic matching $M$ in $\X(K'_\P)$ with a single critical cell of dimension 0, Figure \ref{fig:matching}.

\begin{figure}[htb!]
    \begin{minipage}{0.52\textwidth}
    \begin{scriptsize}
     \def\svgwidth{\textwidth}
    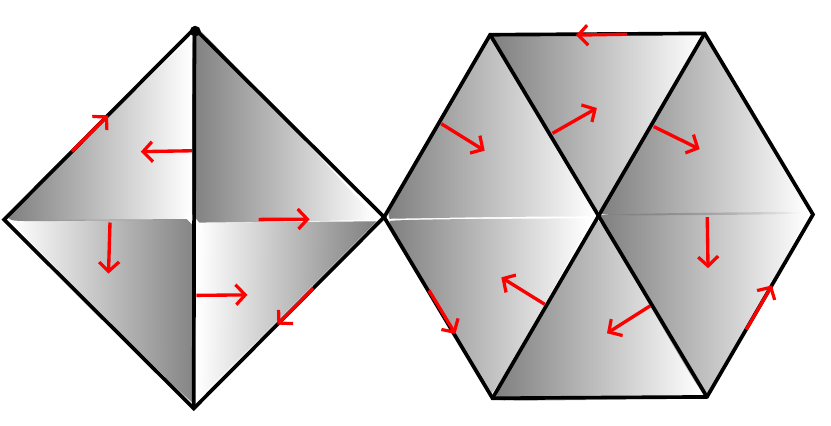
    \end{scriptsize}
\end{minipage}
\centering
\xymatrix@C=0.6em @R=3em{
\circ_{16}&\bullet_{15}&\bullet_{14}&\bullet_{13}&&&&&&
\bullet_{12}&\bullet_{11}&\bullet_{10}&\bullet_{9}&\bullet_{8}&\bullet_{7}\\
%1-cells
\bullet_4\ar@{-}[u]\ar@{-}[urrr]&
\bullet_{15}\ar@{-}@*{[red]}[u]\ar@{-}[ul]&
\bullet_{14}\ar@{-}@*{[red]}[u]\ar@{-}[ul]&
\bullet_{13}\ar@{-}@*{[red]}[u]\ar@{-}[ul]&
%here starts the part related to the generators
\bullet_{12}\ar@{-}[ull]\ar@{-}[ullll]\ar@{-}@*{[red]}[urrrrr]&
\bullet_{3}\ar@{-}[ull]\ar@{-}[ullll]\ar@{-}[urrrrr]&
&\circ_{6}\ar@{-}[urrrr]\ar@{-}[urrrrrr]
&\bullet_{2}\ar@{-}[urrrr]\ar@{-}[urrrrrr]&
%here starts the secord relator
\bullet_5\ar@{-}[u]\ar@{-}[urrrrr]&
\bullet_{11}\ar@{-}@*{[red]}[u]\ar@{-}[ul]&
\bullet_{10}\ar@{-}@*{[red]}[u]\ar@{-}[ul]&
\bullet_{9}\ar@{-}@*{[red]}[u]\ar@{-}[ul]&
\bullet_8\ar@{-}@*{[red]}[u]\ar@{-}[ul]&
\bullet_{7}\ar@{-}@*{[red]}[u]\ar@{-}[ul]\\
%0-cells
&\bullet_{4}\ar@{-}[u]\ar@{-}@*{[red]}[ul]\ar@{-}[urr]\ar@{-}[ur]& &&&
\bullet_{3}\ar@{-}[ull]\ar@{-}[ullll]
\ar@{-}[urrrrr]\ar@{-}@*{[red]}[u]\ar@{-}[ul]&
\circ_1\ar@{-}[ullll]
\ar@{-}[ullllll]
\ar@{-}[urrr]
\ar@{-}[urrrrr]
\ar@{-}[urrrrrrr]
\ar@{-}[ull]\ar@{-}[ul]\ar@{-}[ur]\ar@{-}[urr]&
\bullet_{2} \ar@{-}[urrrrr]\ar@{-}[urrrrrrr]\ar@{-}[u]\ar@{-}@*{[red]}[ur]
&&&&\bullet_{5}\ar@{-}[u]\ar@{-}[ul]\ar@{-}[ur]\ar@{-}@*{[red]}[ull]\ar@{-}[urr]\ar@{-}[urrr]
}
     \caption{Top: The barycenter subdivision of $K_\P$ with arrows and labels associated to an acyclic matching $M$. Bottom: The face poset $\X(K'_\P)$ with the matching $M$ in red. The labels correspond to an induced discrete Morse function $f_M$. Critical cells are empty bullets.}
    \label{fig:matching}
\end{figure}
First, we compute the presentation $\Q_0$ associated to $K'_\P/T$, $T$ being the spanning tree in ${K'_{\P}}^{(1)}$ induced by $M$, see \eqref{Q0} \footnote{Here, the subscript of each critical 1-cell corresponds to the label in Figure \ref{fig:matching}.}. Notice that, although $\Q_0$ is $Q{^{**}}$-equivalent to $\P$, it has a larger  number of generators and relators.
\begin{footnotesize}
\begin{equation} \label{Q0}
\begin{split}\Q_{0} 
= &\langle x_{6}\dots, \mathbf{x_{15}}~|~\\
&x_7, ~x_6^{-1}x_7^{-1}x_8, ~x_8^{-1}x_9, ~x_6^{-1}x_9^{-1}x_{10}, ~x_{10}^{-1}x_{11}, ~x_{11}x_{12}^{-1}, ~x_{13}, ~x_{12}x_{13}^{-1}x_{14}, ~\mathbf{x_{14}^{-1}x_{15}}, ~x_{12}x_{15}^{-1}\rangle
\end{split}\tag{$\ast_3$}
\end{equation}
\end{footnotesize}
We perform below the iterative procedure of reduction of $\Q_0$ described in Definition \ref{pres matching}.  This is a sequence of  $Q^{**}$-transformations induced by $M$ to obtain $\Q_{K,M}$, a new presentation derived from $\Q_0$ with fewer generators and relators that is also $Q^{**}$-equivalent to $\P$ (see Proposition \ref{presentation K_M}). At each step, the generator and relator involved in the transformation of the next step is in bold.
%\textcolor{red}{The sequence of intermediate $Q^{**}$-equivalent presentations is described below. }
\begin{footnotesize}
\begin{align*}\Q_{1}
= &\langle x_{6}, \dots, \mathbf{x_{14}}~|~
x_7, ~x_6^{-1}x_7^{-1}x_8, ~x_8^{-1}x_9, ~x_6^{-1}x_9^{-1}x_{10}, ~x_{10}^{-1}x_{11}, ~x_{11}x_{12}^{-1}, ~x_{13}, ~\mathbf{x_{12}x_{13}^{-1}x_{14}}, ~x_{12}x_{14}^{-1}\rangle\\
\Q_{2} 
= &\langle x_{6},\dots, \mathbf{x_{13}}~|~
x_7, ~x_6^{-1}x_7^{-1}x_8, ~x_8^{-1}x_9, ~x_6^{-1}x_9^{-1}x_{10}, ~x_{10}^{-1}x_{11}, ~x_{11}x_{12}^{-1}, ~\mathbf{x_{13}}, ~ x_{12}x_{13}^{-1}x_{12}\rangle\\
\Q_{3} 
= &\langle x_{6}, \dots, \mathbf{x_{12}}~|~
x_7, ~x_6^{-1}x_7^{-1}x_8, ~x_8^{-1}x_9, ~x_6^{-1}x_9^{-1}x_{10}, ~x_{10}^{-1}x_{11}, ~\mathbf{x_{11}x_{12}^{-1}}, ~x_{12}^2\rangle\\
\Q_{4} 
= &\langle x_{6}, \dots, \mathbf{x_{11}}~|~
x_7, ~x_6^{-1}x_7^{-1}x_8, ~x_8^{-1}x_9, ~x_6^{-1}x_9^{-1}x_{10}, ~\mathbf{x_{10}^{-1}x_{11}}, ~x_{11}^2\rangle\\
\Q_{5} 
= &\langle x_{6}, \dots, \mathbf{x_{10}}~|~
x_7, ~x_6^{-1}x_7^{-1}x_8, ~x_8^{-1}x_9, ~\mathbf{x_6^{-1}x_9^{-1}x_{10}}, ~x_{10}^{2}\rangle\\
\Q_{6} 
= &\langle x_{6}, \dots, ~\mathbf{x_9}~|~
x_7, ~x_6^{-1}x_7^{-1}x_8, ~\mathbf{x_8^{-1}x_9}, ~(x_9x_6)^2\rangle\\
\Q_{7} 
= &\langle x_{6}, ~x_{7}, ~\mathbf{x_{8}}~|~
x_7, ~\mathbf{x_6^{-1}x_7^{-1}x_8}, ~(x_8x_6)^2\rangle\\
\Q_{8} 
= &\langle x_{6}, ~\mathbf{x_{7}}~|~
\mathbf{x_7}, ~(x_7x_6^2)^2\rangle\\
\Q_{9} 
= &\langle x_{6}~|~
 ~x_6^{4}\rangle
\end{align*}
\end{footnotesize}
Hence, $\P \sim_{Q^{**}} \langle x_{6}~|~
 ~x_6^{4}\rangle$.
\end{example}

\begin{remark}
Given a presentation $\P$ and a matching $M$ in $K'_\P$ with only one critical 0-cell, we showed that $\P\sim_{Q^{**}}\Q_{K'_\P,M}$. We now estimate a sufficient number of $Q^{**}$-transformations to obtain $\Q_{K'_\P,M}$ from $\P$.
Let $n$ be the number of generators of $\P$, $m$ the number of relators and $k$
the total relator length. We base our reasoning in the proof of Theorem \ref{teo morse presentation}. The equivalences
$\P\sim_{Q^{**}} \P_{K_{\P}}$ and $\P_{K_{\P}}\sim_{Q^{**}}\P_{K_{\P}'}$ can be achieved in 
$O(n+m)$ and $O(k)$ $Q^{**}$-transformations respectively.
By Corollary \ref{morse AC}, $K_{\P}'\threedef (K_{\P}')_M$ and  the estimated number
of $Q^{**}$-transformations required to 
obtain $\P_{(K_{\P}')_M}$ from $\P_{K_{\P}'}$ is proportional to the number of elementary expansions and collapses needed to deform $K_{\P}'$ into $(K_{\P}')_M$. 
This is bounded by the square of the number of cells of $K_{\P}'$, 
which is proportional to $k$.
\end{remark}

\section{Applications to the Andrews--Curtis conjecture}\label{counterexamples}

In this section, we apply the techniques developed in Section \ref{section Q**-transformations} to prove that some of the potential counterexamples to the Andrews--Curtis conjecture do satisfy the conjecture. 

There is a common belief that the Andrews--Curtis conjecture is false.  However, not a single counterexample could be found yet.
Over the last fifty years, a list of examples of balanced presentations 
of the trivial group which are not known to be trivializable via 
$Q^{**}$-transformations has been compiled.
They serve as \textit{potential counterexamples} to disprove the Andrews--Curtis conjecture (see \cite[Ch. XII]{hog1993two} for a detailed reference). 

For a potential counterexample $\P$, the general outline we will adopt to prove that 
$\P\sim_{Q^{**}}\langle~|~\rangle$ is to find an appropriated acyclic matching $M$ in $\X(K_{\P'})$ such that the presentation $\Q_{K'_\P,M}$, which is in the same $Q^{**}$-class as $\P$, is computationally tractable and easily transformed into the trivial presentation $\langle~|~\rangle$. Note that, given $\P$ and the acyclic matching $M$, the complexity of the computation of the presentation $\Q_{K'_\P,M}$, the \textit{Morse presentation} associated to $\P$, is $O(k)$ with $k$ the total length relator of $\P$. We emphasize that our method of $Q^{**}$-transformation makes use of the transformation (4), the only transformation in (1)--(5) that increases the total length relator of the presentation. This point is central for the posterior manageable reduction of the $Q^{**}$-equivalent Morse presentation.

We say that a presentation $\P$ is \textit{greedily trivializable} if the
 \textit{reduction algorithm} for simplification of presentations described in 
\cite{MR760651}  can transform $\P$ into $\langle ~|~ \rangle$. 
This procedure was originally developed for Tietze simplification of presentations and it consists of a loop of two phases which is iterated until no further operation is possible. The \textit{search phase} attempts to reduce the length-relator by replacing  long substrings of relators  by shorter equivalent ones.   That is, if there are relators $r_1$ and $r_2$ such that a suitable cyclic permutation of $r_1$ reads $uv$ and a cyclic permutation of $r_2$, or its inverse reads as $wv$, and the length of $u$ is greater than the length of $w$, then $r_2$ is replaced by $wu^{-1}$. The \textit{elimination phase} involves the elimination of generators $x$ which occur only once in some relator $r$ (in which case the reduction consist of the elimination of the generator $x$ and the relator $r$). 

\begin{remark}\label{remark simplified}
There is only one situation in which the \textit{reduction algorithm} \cite{MR760651}
performs the transformation (6). Namely, suppose that the presentation $\P$ has a relator $r_i$ which is equal to another
relator $r_j$. In that case, the algorithm replaces relator $r_i$ by a 1 and then eliminates the latter 1.
 This transformation changes the deficiency of the presentation and does not preserve the (simple) homotopy type of the associated CW-complex $K_{\P}$. 
 
Note that the previous situation is not possible if $\P$ is a balanced
presentation of the trivial group, since if $\P$ has (after possibly a sequence of $Q^{**}$-transformations)
one relator equal to another, then it is in the same $Q^{**}$-class as a presentation with a relator equal to 1. Thus, $K_{\P}$ has non-trivial
second homology group, which is not possible.
Therefore, the reduction algorithm performs only $Q^{**}$-transformations if the input is a balanced presentation of the trivial group. 
\end{remark}

\subsection{Experimental results.} \label{examples} We next present an application of our method to computationally prove that some potential counterexamples of the Andrews--Curtis conjecture can be $Q^{**}$-trivialized. Concretely, given a balanced presentation $\P$, we computationally generate an acyclic matching $M$ in $\X(K'_\P)$ such that the presentation $\Q_{K'_\P,M}$ is greedily trivializable.
All the computations are made using {\fontfamily{lmss}\selectfont SAGE} \cite{sagemath}.  The code to replicate the following examples can be found at \cite{finite-spaces}, whereas the outline of the routine is described in Appendix \ref{appendix code sage}.
We remark that in all of these examples of balanced presentations $\P$ of the trivial group, the reduction algorithm applied to the presentation $\Q_{K'_\P,T}$, with $T$ the spanning tree  induced by the matching $M$, is not able to trivialize it (cf. Lemma \ref{matching tree}).

\begin{example}[Akbulut \& Kirby \cite{MR816520}]\label{example AK}
The family of balanced presentations of the trivial group $\A\K_n$ was inspired by the case $n=4$, which corresponds to a handle decomposition of the Akbulut-Kirby 4-sphere: \[\A\K_n=\langle x, y~|~ xyx = yxy, x^n = y^{n+1}\rangle,~n\geq 1.\]  We prove that the presentation 
\[
\A\K_2=\langle x, y~|~ xyx = yxy, ~x^2 = y^{3}\rangle
\] 
satisfies the Andrews--Curtis conjecture.
Indeed, it is simple to find an acyclic matching $M$ in $\X(K'_{\A\K_2})$ such that
$\Q_{K'_{\A\K_2},M}$ is greedily trivializable (see Appendix \ref{appendix code sage}). 
Since $\A\K_2 \sim_{Q^{**}} \Q_{K'_{\A\K_2},M}$, this shows that $\A\K_2$ satisfies the Andrews--Curtis conjecture.
This fact was previously proved in \cite{MR1727164} using genetic algorithms  (also by S. M. Gersten in an unpublished work \cite{gersten87}). However, their methods were focused on the explicit sequence of transformations to trivialize the presentation, rather than to give a proof of its existence.
For $n=1$, 
$\A\K_1$ is simply $Q^{**}$-trivializable (see Example \ref{case n=1}). For $n>2$, the question remains open.
\end{example}

\begin{example}[Miller III \& Schupp \cite{MR1732210}]\label{miller}
Given $w$ a word in $x$ and $y$ with exponent sum 0 on $x$ and $n>0$,
\[
\M\S_n (w) = \langle x,y~|~x^{-1}y^n x = y^{n+1}, x = w\rangle
\]
is a balanced presentation of the trivial group.
A well-studied subfamily of the presentations $\M\S_n(w)$ is that given by $w_*=y^{-1}xyx^{-1}$ (see \cite{MR1970867, MR1727164}).  In these works, the authors proved that $\M\S_n(w_*)$ is $Q^{**}$-trivializable for $n\leq 2$. The case $n=3$ was the smallest potential counterexample of this family. With our algorithmic method, we prove that 
\[
\M\S_3(w_*) = \langle x,y~|~x^{-1}y^3 x = y^4, x = y^{-1}xyx^{-1}\rangle\]
does satisfy the Andrews--Curtis conjecture. Precisely, we computationally generate an acyclic matching $M$ such that the Morse presentation $\Q_{K'_{\M\S_3(w_*)},M}$ is greedily trivializable (see Appendix \ref{appendix code sage}).
\end{example}

\begin{example}[Barmak \cite{Ba18}] \label{example barmak} The \textit{generalized Andrews--Curtis conjecture} states that any two presentations $\P$ and $\Q$ with simple homotopy equivalent standard complexes $K_{\P}, K_{\Q}$ are $Q^{**}$-equivalent \cite[Ch.I]{hog1993two}. A strong version of the generalized conjecture asserts that under the previous conditions, $\P$ can be obtained  from $\Q$ by performing only operations (1) to (3) (called \textit{$Q^{*}$-transformations}). 
Recently, Barmak found a counterexample to this strong formulation of the conjecture \cite{Ba18}. He proved that \[\B_1=\langle x,y~|~[x,y],1\rangle\]
is not $Q^*$-equivalent to \[\B_2=\langle x,y~|~[x,[x,y^{-1}]]^2y[y^{-1},x]y^{-1},[x,[[y^{-1},x],x]]\rangle,\]
but $K_{\B_1}\se K_{\B_2}$ (here $[x,y] = xyx^{-1}y^{-1}$). He asked whether these presentations are $Q^{**}$-equivalent or not, being $\B_1$ and $\B_2$ a potential counterexample to the generalized Andrews--Curtis conjecture.

We produce an acyclic matching $M$ in $\X(K'_{\B_2})$ such that the associated Morse presentation is greedily transformed into $\B_1$ via the reduction algorithm (see Appendix \ref{appendix code sage}). By Theorem \ref{teo morse presentation} and Remark \ref{remark simplified}, we conclude that $\B_1\sim_{Q^{**}}\B_2$ 
and hence, that this counterexample to the strong version of the conjecture does not disprove the generalized Andrews--Curtis conjecture.
\end{example}

\subsection{Gordon's potential counterexamples}
 In \cite[Sect. 3.8--3.10]{MR770569}, the author brings to the attention the following family of balanced presentations of the trivial group proposed by Gordon
\[\G_{n,m,p,q}=\langle x,y~| ~x=[x^m,y^n], y=[y^p,x^q]\rangle ~n,m,p,q\in \ZZ,\]
where $[x,y] = xyx^{-1}y^{-1}.$
Many instances of this family of presentations are potential counterexamples of the Andrews--Curtis conjecture
\cite{metzler1985andrews}. For instance, the presentation $\M\S_3(w_*)$ of Example \ref{miller} (see also \cite{MR1732210}) is a representative of this family for $n=m=q = -1, p = -3$.
It is not known in general
whether a presentation $\G_{n,m,p,q}$ is $Q^{**}$-trivializable.
In \cite{MR2253006}, the authors proved that any presentation
in this sequence with total length-relator up to 14  can be transformed into $\langle~|~\rangle$ by exploring the space of possible sequences of transformations (1)--(3).
We focus our attention on the subfamily
 \begin{align*}\G_q :=\G_{-1,-1,-1, -q}&=\langle x,y~| ~x=[x^{-1},y^{-1}], ~ y=[y^{-1},x^{-q}]\rangle. \\
 & = \langle x,y~|~x=x^{-1}y^{-1}x y, ~y = y^{-1}x^{-q} y x^{q}\rangle. \end{align*}
We show that $\G_q$ is $Q^{**}$-trivializable for all  $q\in \NN$, by an inductive procedure based on %\textcolor{red}{our refinement of discrete Morse theory} 
our theory from Section \ref{section Q**-transformations}.

\begin{theorem}
The presentation $\G_q= \langle x,y~|~x=x^{-1}y^{-1}x y, ~y = y^{-1}x^{-q} y x^{q}\rangle$ satisfies the Andrews--Curtis conjecture for all $q\in \NN$.
\end{theorem}

\begin{proof}
If $q$ is even, $\G_{q}$  results greedily trivializable. Indeed, given $q = ~ 2k$ with $k\in \NN$, the reduction algorithm replaces $k$ times the string $x^{-1}yx^2$ by its equivalent expression $y$ in the second relator $y^{-2}x^{-2k}yx^{2k}$, which results in $y^{-2}x^{-k}y$ or, equivalently, $y^{-1}x^{-k}$.
Now, the first relator $x^{-2}y^{-1}xy$ is transformed into $x$ after replacing the string $y$ by $x^{-k}$. The resulting presentation $\langle  x, y ~ | ~ x, y^{-1}x^{-k}\rangle$ is clearly $Q^{**}$-trivializable. Hence, $\G_{2k}\sim_{Q^{**}}\langle ~ | ~ \rangle$.

For $q$ odd, that is $q = 2k-1$,
a similar procedure
transforms $\G_{2k-1}$ into 
\[\tilde \G_k = \langle x,y ~ |~ x^{-2}y^{-1}x y, ~y^{-1}x^{-k+1}yx^{-1}y^{-1}\rangle, ~  k\in \NN.\] For $k=1$, it is easy to check that $\tilde \G_1$
is trivializable. For $k=2$, 
\[\tilde \G_2 = \langle x,y ~ |~ x^{-2}y^{-1}x y, ~y^{-1}x^{-1}yx^{-1}y^{-1}\rangle.
\]
If we replace the first relator $r_1$ by $r_1r_2$, we obtain $\tilde r_1 = x^{-3}y^{-1}$. Now, $y$ occurs only once at $\tilde r_1$. After replacing  $y$ by its equivalent expression $x^{-3}$ in $r_2$ and eliminating the generator $y$ and relator $\tilde r_1$, the resulting $Q^{**}$-equivalent presentation
\[\langle x ~ | ~ x \rangle\] is trivializable.

For $q>3$ odd (that is $q = 2k-1$ with $k>2$), the reduction algorithm is not able to trivialize $\G_q$ (neither $\tilde \G_k$). We prove by induction on $k$ that $\G_k\sim_{Q^{**}}\langle ~ | ~  \rangle$. For $k\geq 3$, consider the acyclic matching $M_k$ in the poset $X_k = \X(K'_{\tilde \G_k})$ of Figure \ref{fig:matching G_k}.
\begin{figure}[htb!]
    \centering
\begin{small}
\xymatrix@C=0.25em@R=4em{
\bullet_{x^{-1}}&\bullet_{x^{+1}}&\bullet_{x^{-1}}&\bullet_{x^{+1}}&
\bullet_{y^{-1}}&\bullet_{y^{+1}}&\circ_{x^{+1}}&\bullet_{x^{-1}}&\bullet_{y^{+1}}&\circ_{y^{-1}}
\\
\bullet_o\ar@{-}@*{[red]}[u]\ar@{-}[urrrrrrrrr]&
\bullet_{x}\ar@{-}@*{[red]}[u]\ar@{-}[ul]&
\bullet_o\ar@{-}@*{[red]}[u]\ar@{-}[ul]&
\bullet_{x}\ar@{-}[u]\ar@{-}[ul]&
\bullet_o\ar@{-}@*{[red]}[u]\ar@{-}[ul]&
\bullet_{y}\ar@{-}@*{[red]}[u]\ar@{-}[ul]&
\bullet_o\ar@{-}[u]\ar@{-}[ul]&
\bullet_{x}\ar@{-}@*{[red]}[u]\ar@{-}[ul]&
\bullet_o\ar@{-}@*{[red]}[u]\ar@{-}[ul]&
\bullet_{y}\ar@{-}[u]\ar@{-}[ul]&
%here starts the part related to the generators
\bullet_{x^{+1}}\ar@{-}@*{[red]}[ulllllll]&
\circ_{x^{-1}}&&
\bullet_{y^{+1}}&
\bullet_{y^{-1}}
\\
%0-cells
&&&&\bullet_{v_{r_1}}\ar@{-}[u]\ar@{-}[ul]\ar@{-}[ull]\ar@{-}[ulll]\ar@{-}[ullll]
\ar@{-}@*{[red]}[urr]\ar@{-}[urrr]\ar@{-}[urrrr]\ar@{-}[urrrrr]\ar@{-}[ur]&&&&&&&
\bullet_{x}\ar@{-}[ul]\ar@{-}[u]\ar@{-}@*{[red]}[ullllllll]&
\circ_o
\ar@{-}[ul]\ar@{-}[ull]\ar@{-}[ur]\ar@{-}[urr]&
\bullet_{y}\ar@{-}[u]\ar@{-}[ur]\ar@{-}@*{[red]}[ullll]
}

\vspace{15pt}

\xymatrix@C=0.0001em@R=4em{
&&&&&
\bullet_{y^{-1}}&\bullet_{y^{+1}}&\bullet_{x^{-1}}&\bullet_{x^{+1}}&\dots&\bullet_{x^{-1}}&\bullet_{x^{+1}}&\bullet_{y^{+1}}&\bullet_{y^{-1}}&\bullet_{x^{-1}}&\bullet_{x^{+1}}&\circ_{y^{-1}}&\bullet_{y^{+1}}\\
%here starts the part related to the generators
\bullet_{x^{+1}}&
\circ_{x^{-1}}&&
\bullet_{y^{+1}}\ar@{-}@*{[red]}[urrrrrrrrr]&
\bullet_{y^{-1}}\ar@{-}@*{[red]}[urrrrrrrrr]&
%here starts the secord relator
\bullet_o\ar@{-}@*{[red]}[u]\ar@{-}[urrrrrrrrrrrr]&
\bullet_{y}\ar@{-}@*{[red]}[u]\ar@{-}[ul]&
\bullet_o\ar@{-}@*{[red]}[u]\ar@{-}[ul]&
\bullet_{x}\ar@{-}@*{[red]}[u]\ar@{-}[ul]&\dots&
\bullet_o\ar@{-}@*{[red]}[u]\ar@{-}[ull]|\cdots&
\bullet_{x}\ar@{-}@*{[red]}[u]\ar@{-}[ul]
&
\circ_o\ar@{-}[u]\ar@{-}[ul]&
\bullet_{y}\ar@{-}[u]\ar@{-}[ul]&
\bullet_o\ar@{-}@*{[red]}[u]\ar@{-}[ul]&
\circ_{x}\ar@{-}[u]\ar@{-}[ul]&
\bullet_o\ar@{-}[u]\ar@{-}@*{[red]}[ul]&
\bullet_{y}\ar@{-}@*{[red]}[u]\ar@{-}[ul]
\\
%0-cells
&\bullet_{x}\ar@{-}[ul]\ar@{-}[u]&
\circ_o
\ar@{-}[ul]\ar@{-}[ull]\ar@{-}[ur]\ar@{-}[urr]&
\bullet_{y}\ar@{-}[u]\ar@{-}[ur]
&&&&&&&&\bullet_{v_{r_2}}\ar@{-}[u]
\ar@{-}[ul]\ar@{-}[ulll]\ar@{-}[ullll]\ar@{-}[ulllll]\ar@{-}[ullllll]
\ar@{-}[ur]\ar@{-}@*{[red]}[urr]\ar@{-}[urrr]\ar@{-}[urrrr]\ar@{-}[urrrrr]\ar@{-}[urrrrrr]
}
\end{small}
    \caption{Acyclic matching in $X_k$ (edges in red). Critical points are the empty bullets. Edges between generator elements and relator elements are not depicted unless they belong to the matching.
    The figure illustrates the subposets of $X_k$ corresponding to the generator elements and the first relator elements (top) and the second relator elements (bottom).}
    \label{fig:matching G_k}
\end{figure}

The Morse presentation $\Q_{K'_{\tilde \G_k}, M_k}$ is ($Q^{**}$-equivalent to) \[\begin{small}\langle a_0, a_1, a_2 ~ | ~ a_0 a_1^{-1} a_0, 
~a_2 a_1^{-1} a_2^{-1} a_0^{-1} a_1, 
~a_0a_2^{-1}a_0^{-1} a_1a_2^{-1} (a_1^{-1}a_0)^{k-1}a_2a_1^{-1}
\rangle\end{small}\]
(see Appendix \ref{appendix gordon})
and it can be inductively trivialized.
In fact, there exists a sequence of $Q^{**}$-equivalent presentations
\begin{align*}
\Q_{K'_{\tilde \G_k}, M_k}&\simeq_{Q^{**}} \langle a_0, a_2 ~ | ~ a_2 a_0 ^{-2}a_2^{-1}a_0, ~ a_2^{-1}a_0a_2^{-1}a_0^{-k+1}a_2a_0^{-1} \rangle \\
&\simeq_{Q^{**}} \langle a_0, a_2 ~ | ~ a_2 a_0 ^{-2}a_1^{-1}a_0, ~ a_2^{-1}a_0a_2^{-1}a_0^{-k+2}a_2a_0^{-3} \rangle \\
&\simeq_{Q^{**}} \langle a_0, a_2 ~ | ~ a_2 a_0 ^{-2}a_2^{-1}a_0, ~ a_0^{-1}a_2^{-2}a_0^{-k+2}a_1 \rangle\\
&\simeq_{Q^{**}} \tilde \G_{k-1}.
\end{align*}
For the first $Q^{**}$-equivalence, notice that the generator $a_1$ appears only once in the first relator of $\Q_{K'_{\tilde \G_k}, M_k}$. Thus, after replacing $a_1$ by its equivalent expression $a_0^2$ in the rest of the relators, the generator $a_1$ and the relator $a_0 a_1^{-1} a_0$ can be removed.
Now, for the second equivalence, the second relator can be rewritten as $a_2^{-1}a_0(a_2^{-1}a_0^{-k+2}a_2)(a_2^{-1}a_0^{-1}a_2)a_0^{-1}$. It can be deduced from the first relator that the string  $a_2^{-1}a_0^{-1}$ is equivalent to $a_0^{-2}a_2^{-1}$. After replacing $a_2^{-1}a_0^{-1}$ by its equivalent expression in the second relator, the latter is transformed into $a_2^{-1}a_0(a_2^{-1}a_0^{-k+2}a_2)(a_0^{-2}a_2^{-1}a_2)a_0^{-1}$ or equivalently, $a_0(a_2^{-1}a_0^{-k+2}a_2)a_0^{-3}a_2^{-1}$. By replacing this time $a_0^{-2}a_2^{-1}$ by $a_2^{-1}a_0^{-1}$, it is transformed into $a_2^{-1}a_0^{-k+2}a_2a_0^{-1}a_2^{-1}$ and, hence, the resulting presentation
is $Q^{**}$-equivalent to $\tilde \G_{k-1}$.
By inductive hypothesis, $\tilde \G_{k-1}$ is $Q^{**}$-trivializable and therefore, so is $\Q_{K'_{\tilde \G_k}, M_k}$.
\end{proof}

\appendix
\section{Computational procedure for examples in Section \ref{examples}} \label{appendix code sage}

The algorithm to compute a presentation $Q^{**}$-equivalent to a given one is implemented in {\fontfamily{lmss}\selectfont SAGE} \cite{sagemath}.  The code can be found at the repository \cite{finite-spaces}.\footnote{The package {\fontfamily{lmss}\selectfont Posets} \cite{GAP-posets} in {\fontfamily{lmss}\selectfont GAP} \cite{GAP4} also contains an implementation of the functions described in this section.}

We describe next the main functions. Given as input the lists \texttt{gens} and \texttt{rels} of generators and relators, the function \texttt{group\_presentation(gens, rels)} computes the induced group presentation $\P = \langle \texttt{gens} ~|~ \texttt{rels}\rangle$.
The method \texttt{simplified()} applies the reduction algorithm described in Section \ref{counterexamples} \cite{MR760651}.
It can be verified that the method \texttt{simplified()} is not able to reduce the number of generators and relators of the original presentation in any of the Examples of Section \ref{examples}.

The function \texttt{presentation\_poset(gens, rels)} returns $\X(K'_\P)$, that is, the face poset of the barycentric subdivision of the standard complex induced by $\P$.

Given a poset $\texttt X$, we randomly generate an acyclic matching with a single critical point at level 0 with the function \texttt{spanning\_matching(X)}.

Now, given an acyclic matching \texttt{M} in the poset $\X(K'_{\P})$, the function \texttt{Morse\_presentation (gens, rels, M)} computes the associated Morse presentation $\Q_{K'_\P, \texttt{M}}$ (see Definition \ref{pres matching}). By Theorem \ref{teo morse presentation}, the latter is $Q^{**}$-equivalent to $\P$.

In Examples \ref{example AK} and \ref{miller}, we find in short runtime an appropriate matching \texttt{M} such that the method \texttt{simplified()} applied to \texttt{Morse\_presentation(gens, rels, M)}  produce as output the presentation $\langle ~ | ~ \rangle$. In Example \ref{example barmak}, given \texttt{gens} and \texttt{rels} the generators and relators of $\B'$, we generate a matching \texttt{M} such that \texttt{Morse\_presentation(gens, rels, M).simplified()} is the presentation $\langle x,y~|~[x,y]\rangle$. By Remark \ref{remark simplified}, this implies that the original presentation $\B' = \langle \texttt{gens} ~|~ \texttt{rels}\rangle$ is $Q^{**}$-equivalent to $\B = \langle x,y~|~[x,y], 1\rangle$.\footnote{The list of transformations involved in the simplification of the Morse presentation $Q_{K'_{\B'}, \texttt{M}}$ can be verified using the function \texttt{SimplifyPresentation} in {\fontfamily{lmss}\selectfont GAP} \cite{GAP4} (equivalent to the method \texttt{simplified()} in {\fontfamily{lmss}\selectfont SAGE} \cite{sagemath}) and it is included in the repository \cite{finite-spaces} for completeness.}

The following script summarizes the procedure:

\smallskip 

\begin{spverbatim}
P = group_presentation(gens, rels)
X = presentation_poset(gens, rels)
l = len(P.simplified().generators())
while l>1:
    M  = spanning_matching(X)
    Q  = Morse_presentation(gens, rels, M)
    S = Q.simplified()
    l = len(S.generators())
print(S)
\end{spverbatim}

\section{The Morse presentation $Q_{K'_{\G_{2k-1}}, M_k}$}.\label{appendix gordon}
In this appendix we give a proof of the construction of the Morse presentation $Q_{K'_{\G_{2k-1}}, M_k}$ for all $k\in \NN$. Recall that $\G_{2k-1} = \tilde \G_k = \langle x,y ~ |~ x^{-2}y^{-1}x y, ~y^{-1}x^{-k+1}yx^{-1}y^{-1}\rangle, ~  k\in \NN$. The poset $X_k$ is the face poset of the regular CW-complex $K_{\tilde \G_k}'$ associated to $\tilde \G_k$ (see Figure \ref{fig:CW_gordon}). 
\begin{figure}[htb]
    \centering
    \begin{footnotesize}
    \def\svgwidth{0.89\textwidth}
     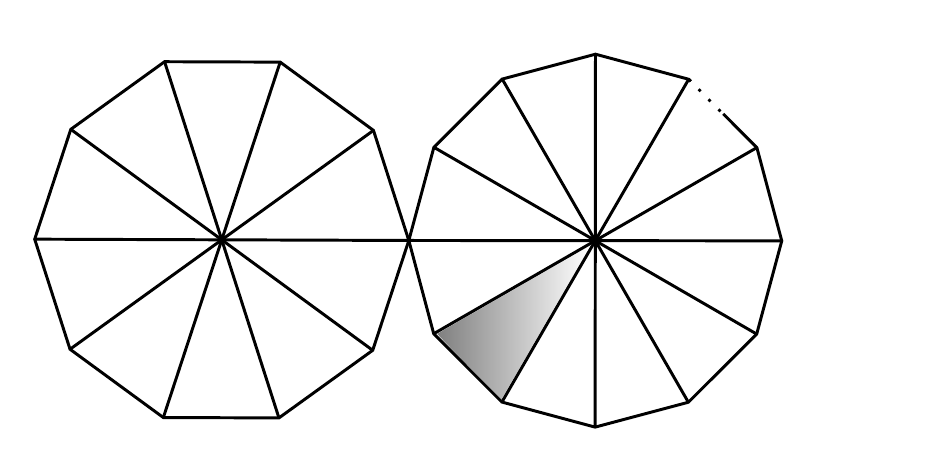
     \end{footnotesize}
    \caption{The regular CW- complex $K_{\tilde \G_k}'$. The 1-cells  $x^{1+}$ and $x^{-1}$ in braces are repeated $k-1$ times and correspond to the substring $x^{-k+1}$ of the second relator of $\tilde \G_k$.}
    \label{fig:CW_gordon}
\end{figure}

We label the 1-cells $x_1, x_2, \dots, x_{20+2k}$ as in Figure \ref{fig:gordon_matching}.

\begin{figure}[htb]
    \centering
    \begin{footnotesize}
    \def\svgwidth{0.89\textwidth}
     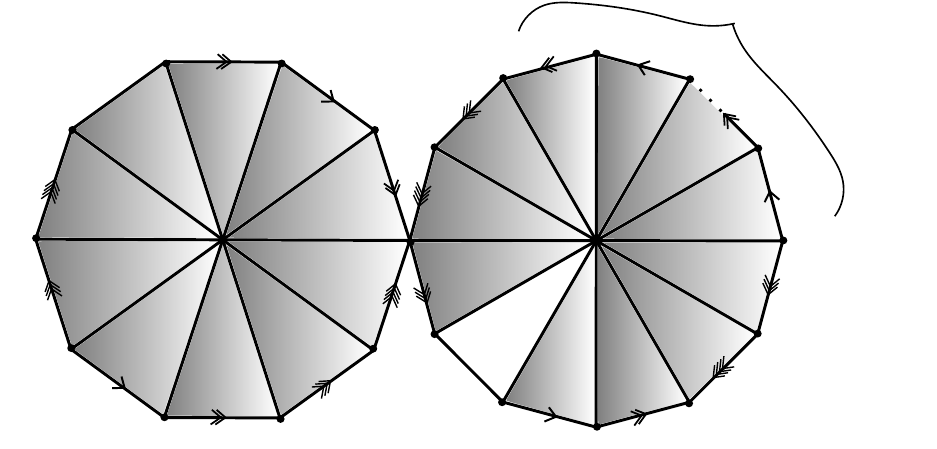
     \end{footnotesize}
    \caption{Labeling of the 1-cells of $K_{\tilde \G_k}'$. The matching $M_k$ in $\X(K_{\tilde \G_k}')$ is represented with red arrows.}
    \label{fig:gordon_matching}
\end{figure}
\

The presentation $\Q_0$ has generators that coincide with the set of 1-cells $x_1, x_2, \dots, x_{20+2k}$ and relators $ r_1, r_2, \dots,$ $r_{20+2k}$ that  correspond either to a non-critical 1-cell (that is, relators $x_8,  x_{11}, x_{14}, x_{16+2k}$) or to the attaching map of a 2-cell.
Concretely,

\begin{footnotesize}
\begin{equation*}
\begin{split}\Q_{0} 
&= \langle x_{1}, \textcolor{darkcerulean}{x_2}, \dots, x_7, x_9, x_{10}, x_{12}, x_{13}, x_{15}, \dots, \textcolor{darkcerulean}{x_{15+2k}}, \dots, \textcolor{darkcerulean}{x_{18+2k}}, x_{19+2k}, x_{20+2k}~|~
%primera relacion
 \mathbf{x_8},  \mathbf{x_{11}}, \mathbf{x_{14}},  \mathbf{x_{16+2k}},\\
&\mathbf{x_5}x_2^{-1}x_6^{-1}, 
\mathbf{x_6}x_1^{-1}x_7^{-1}, 
\mathbf{x_7}x_2^{-1}x_8^{-1}, 
x_8\mathbf{x_1^{-1}}x_9^{-1}, 
\mathbf{x_9}x_4^{-1}x_{10}^{-1}, 
\mathbf{x_{10}}x_3^{-1}x_{11}^{-1}, 
\textcolor{darkcerulean}{x_{11}x_1x_{12}^{-1}}, 
\mathbf{x_{12}}x_2x_{13}^{-1}, 
\mathbf{x_{13}}x_3x_{14}^{-1}, \\
&\textcolor{darkcerulean}{x_{14} x_4x_5^{-1}},
% segunda relacion
\mathbf{x_{15}}x_4^{-1}x_{16}^{-1},
\mathbf{x_{16}}x_3^{-1}x_{17}^{-1},
\mathbf{x_{17}}x_2^{-1}x_{18}^{-1},
\mathbf{x_{18}}x_1^{-1}x_{19}^{-1},
\dots,
\mathbf{x_{13 + 2k}}x_2^{-1}x_{14 + 2k}^{-1},
\mathbf{x_{14 + 2k}} x_1^{-1}x_{15 + 2k}^{-1},\\
& x_{15 + 2k}\mathbf{x_3}x_{16 + 2k}^{-1},
x_{16 + 2k}\mathbf{x_4}x_{17 +  2k}^{-1},
\mathbf{x_{17 +  2k}}x_2^{-1}  x_{18 +  2k}^{-1},
x_{18 +  2k}x_1^{-1} \mathbf{x_{19 +  2k}^{-1}},
\textcolor{darkcerulean}{x_{19 +  2k}x_4^{-1} x_{20 +  2k}^{-1}},
\mathbf{x_{20 +  2k}}x_3^{-1}x_{15}^{-1}
\rangle\\
\end{split}
\end{equation*}
\end{footnotesize}

For each relator, we display in bold either the generator associated to the 1-cell to which the underlying 2-cell is matched or the generator that is not critical.

The critical 1-cells are the ones linked to the generators $x_2, x_{15+2k}, x_{18+2k}$, whereas critical 2-cells are in correspondence with the relators $x_{11}x_1x_{12}^{-1}, x_{14} x_4x_5^{-1}, x_{19 +  2k}x_4^{-1} x_{20 +  2k}^{-1}$ (in \textcolor{darkcerulean}{blue}).
After performing the reductions associated to the non-critical 1-cells (or equivalently, the collapse to a point of the associated spanning tree on the 1-skeleton), we obtain

\begin{footnotesize}
\begin{equation*}
\begin{split}\Q_{0}  
\sim_{Q^{**}}
&\langle x_{1}, \textcolor{darkcerulean}{x_2}, \dots, x_7, x_9, x_{10}, x_{12}, x_{13}, x_{15}, \dots, \textcolor{darkcerulean}{x_{15+2k}}, x_{17+2k}, \textcolor{darkcerulean}{x_{18+2k}}, x_{19+2k}, x_{20+2k}~|~\\
&\mathbf{x_5}x_2^{-1}x_6^{-1}, 
\mathbf{x_6}x_1^{-1}x_7^{-1}, 
\mathbf{x_7}x_2^{-1}, 
\mathbf{x_1^{-1}}x_9^{-1}, 
\mathbf{x_9}x_4^{-1}x_{10}^{-1}, 
\mathbf{x_{10}}x_3^{-1}, 
\textcolor{darkcerulean}{x_1x_{12}^{-1}}, 
\mathbf{x_{12}}x_2x_{13}^{-1}, 
\mathbf{x_{13}}x_3, \\
&\textcolor{darkcerulean}{ x_4x_5^{-1}},
% segunda relacion
\mathbf{x_{15}}x_4^{-1}x_{16}^{-1},
\mathbf{x_{16}}x_3^{-1}x_{17}^{-1},
\mathbf{x_{17}}x_2^{-1}x_{18}^{-1},
\mathbf{x_{18}}x_1^{-1}x_{19}^{-1},
\dots,
%x_{19}x_2^{-1}x_{20}^{-1}
\mathbf{x_{13 + 2k}}x_2^{-1}x_{14 + 2k}^{-1},
%x_{20}x_1^{-1}x_{21}^{-1}
\mathbf{x_{14 + 2k}} x_1^{-1}x_{15 + 2k}^{-1},\\
& x_{15 + 2k}\mathbf{x_3},
\mathbf{x_4}x_{17 +  2k}^{-1},
\mathbf{x_{17 +  2k}}x_2^{-1}  x_{18 +  2k}^{-1},
x_{18 +  2k}x_1^{-1} \mathbf{x_{19 +  2k}^{-1}},
\textcolor{darkcerulean}{x_{19 +  2k}x_4^{-1} x_{20 +  2k}^{-1}},
\mathbf{x_{20 +  2k}}x_3^{-1}x_{15}^{-1}
\rangle.\\
\end{split}
\end{equation*}
\end{footnotesize}

Now, we perform iteratively the reductions associated to the internal collapses involving pairs of 1-cells and 2-cells. At each step, we replace a generator by its equivalent expression induced by the associated matched pair of cells, and then we remove the corresponding generator and relator. The generator and relator involved at each stage of the procedure are marked in bold.

\begin{footnotesize}
\begin{equation*}
\begin{split}\tilde \G_k 
\sim_{Q^{**}}
&\langle x_{1}, \textcolor{darkcerulean}{x_2}, \dots, x_7, x_9, x_{10}, x_{12}, x_{13}, x_{15}, \dots, \textcolor{darkcerulean}{x_{15+2k}}, x_{17+2k}, \textcolor{darkcerulean}{x_{18+2k}}, x_{1
9+2k},  \mathbf{x_{20+2k}}~|~\\
&x_5x_2^{-1}x_6^{-1}, 
x_6^{-1}x_7^{-1}, 
x_7x_2^{-1}, 
x_1^{-1}x_9^{-1}, 
x_9x_4^{-1}x_{10}^{-1}, 
x_{10}x_3^{-1}, 
\textcolor{darkcerulean}{x_1x_{12}^{-1}}, 
x_{12}x_2x_{13}^{-1}, 
x_{13}x_3, \textcolor{darkcerulean}{ x_4x_5^{-1}},\\
&
% segunda relacion
x_{15}x_4^{-1}x_{16}^{-1},
x_{16}x_3^{-1}x_{17}^{-1},
x_{17}x_2^{-1}x_{18}^{-1},
x_{18}x_1^{-1}x_{19}^{-1},
\dots,
%x_{19}x_2^{-1}x_{20}^{-1}
x_{13 + 2k}x_2^{-1}x_{14 + 2k}^{-1},
%x_{20}x_1^{-1}x_{21}^{-1}
x_{14 + 2k} x_1^{-1}x_{15 + 2k}^{-1},\\
& x_{15 + 2k}x_3,
x_4x_{17 +  2k}^{-1},
x_{17 +  2k}x_2^{-1}  x_{18 +  2k}^{-1},
x_{18 +  2k}x_1^{-1} x_{19 +  2k}^{-1},
\textcolor{darkcerulean}{x_{19 +  2k}x_4^{-1} x_{20 +  2k}^{-1}},
\mathbf{x_{20 +  2k}x_3^{-1}x_{15}^{-1}}
\rangle\\
\sim_{Q^{**}}&
\langle x_{1}, \textcolor{darkcerulean}{x_2}, \dots, x_7, x_9, x_{10}, x_{12}, x_{13}, x_{15}, \dots, \textcolor{darkcerulean}{x_{15+2k}}, x_{17+2k}, \textcolor{darkcerulean}{x_{18+2k}}, \mathbf{x_{19+2k}}~|~\\
&x_5x_2^{-1}x_6^{-1}, 
x_6x_1^{-1}x_7^{-1}, 
x_7x_2^{-1}, 
x_1^{-1}x_9^{-1}, 
x_9x_4^{-1}x_{10}^{-1}, 
x_{10}x_3^{-1}, 
\textcolor{darkcerulean}{x_1x_{12}^{-1}}, 
x_{12}x_2x_{13}^{-1}, 
x_{13}x_3, \textcolor{darkcerulean}{ x_4x_5^{-1}},\\
% segunda relacion
&x_{15}x_4^{-1}x_{16}^{-1},
x_{16}x_3^{-1}x_{17}^{-1},
x_{17}x_2^{-1}x_{18}^{-1},
x_{18}x_1^{-1}x_{19}^{-1},
\dots,
%x_{19}x_2^{-1}x_{20}^{-1}
x_{13 + 2k}x_2^{-1}x_{14 + 2k}^{-1},
%x_{20}x_1^{-1}x_{21}^{-1}
x_{14 + 2k} x_1^{-1}x_{15 + 2k}^{-1},\\
& x_{15 + 2k}\mathbf{x_3},
x_4x_{17 +  2k}^{-1},
x_{17 +  2k}x_2^{-1}  x_{18 +  2k}^{-1},
\mathbf{x_{18 +  2k}x_1^{-1} x_{19 +  2k}^{-1}},
\textcolor{darkcerulean}{x_{19 +  2k}x_4^{-1} x_3^{-1}x_{15}^{-1}}
\rangle\\
\sim_{Q^{**}}&
\langle x_{1}, \textcolor{darkcerulean}{x_2}, \dots, x_7, x_9, x_{10}, x_{12}, x_{13}, x_{15}, \dots, \textcolor{darkcerulean}{x_{15+2k}}, \mathbf{x_{17+2k}}, \textcolor{darkcerulean}{x_{18+2k}}~|~\\
&x_5x_2^{-1}x_6^{-1}, 
x_6x_1^{-1}x_7^{-1}, 
x_7x_2^{-1}, 
x_1^{-1}x_9^{-1}, 
x_9x_4^{-1}x_{10}^{-1}, 
x_{10}x_3^{-1}, 
\textcolor{darkcerulean}{x_1x_{12}^{-1}}, 
x_{12}x_2x_{13}^{-1}, 
x_{13}x_3, \textcolor{darkcerulean}{ x_4x_5^{-1}},\\
&% segunda relacion
x_{15}x_4^{-1}x_{16}^{-1},
x_{16}x_3^{-1}x_{17}^{-1},
x_{17}x_2^{-1}x_{18}^{-1},
x_{18}x_1^{-1}x_{19}^{-1},
\dots,
%x_{19}x_2^{-1}x_{20}^{-1}
x_{13 + 2k}x_2^{-1}x_{14 + 2k}^{-1},
%x_{20}x_1^{-1}x_{21}^{-1}
x_{14 + 2k} x_1^{-1}x_{15 + 2k}^{-1},\\
& x_{15 + 2k}x_3,
x_4x_{17 +  2k}^{-1},
\mathbf{x_{17 +  2k}x_2^{-1}  x_{18 +  2k}^{-1}},
\textcolor{darkcerulean}{x_{18 +  2k}x_1^{-1}x_4^{-1} x_3^{-1}x_{15}^{-1}}
\rangle\\
\sim_{Q^{**}}&
\langle x_{1}, \textcolor{darkcerulean}{x_2}, x_3, \mathbf{x_4}, x_5, x_6,  x_7, x_9, x_{10}, x_{12}, x_{13}, x_{15}, \dots, x_{14+2k}, \textcolor{darkcerulean}{x_{15+2k}},  \textcolor{darkcerulean}{x_{18+2k}}~|~\\
&x_5x_2^{-1}x_6^{-1}, 
x_6x_1^{-1}x_7^{-1}, 
x_7x_2^{-1}, 
x_1^{-1}x_9^{-1}, 
x_9x_4^{-1}x_{10}^{-1}, 
x_{10}x_3^{-1}, 
\textcolor{darkcerulean}{x_1x_{12}^{-1}}, 
x_{12}x_2x_{13}^{-1}, 
x_{13}x_3, \textcolor{darkcerulean}{ x_4x_5^{-1}},\\
&
% segunda relacion
x_{15}x_4^{-1}x_{16}^{-1},
x_{16}x_3^{-1}x_{17}^{-1},
x_{17}x_2^{-1}x_{18}^{-1},
x_{18}x_1^{-1}x_{19}^{-1},
\dots,
%x_{19}x_2^{-1}x_{20}^{-1}
x_{13 + 2k}x_2^{-1}x_{14 + 2k}^{-1},
%x_{20}x_1^{-1}x_{21}^{-1}
x_{14 + 2k} x_1^{-1}x_{15 + 2k}^{-1},\\
& x_{15 + 2k}x_3,
\mathbf{x_4x_2^{-1}  x_{18 +  2k}^{-1}},
\textcolor{darkcerulean}{x_{18 +  2k}x_1^{-1}x_4^{-1} x_3^{-1}x_{15}^{-1}}
\rangle\\
\sim_{Q^{**}}&
\langle x_{1}, \textcolor{darkcerulean}{x_2}, \mathbf{x_3}, x_5, x_6, x_7, x_9, x_{10}, x_{12}, x_{13}, x_{15}, \dots, x_{14+2k}, \textcolor{darkcerulean}{x_{15+2k}},  \textcolor{darkcerulean}{x_{18+2k}}~|~\\
&x_5x_2^{-1}x_6^{-1}, ~
x_6x_1^{-1}x_7^{-1}, ~
x_7x_2^{-1}, ~
x_1^{-1}x_9^{-1}, ~
x_9x_2^{-1}  x_{18 +  2k}^{-1}x_{10}^{-1}, ~
x_{10}x_3^{-1}, ~
\textcolor{darkcerulean}{x_1x_{12}^{-1}}, ~
x_{12}x_2x_{13}^{-1}, ~
x_{13}x_3, ~\\
&\textcolor{darkcerulean}{ x_{18 +  2k}x_2x_5^{-1}},
% segunda relacion
x_{15}x_2^{-1}  x_{18 +  2k}^{-1}x_{16}^{-1},
x_{16}x_3^{-1}x_{17}^{-1},
x_{17}x_2^{-1}x_{18}^{-1},
x_{18}x_1^{-1}x_{19}^{-1},
\dots,
%x_{19}x_2^{-1}x_{20}^{-1}
x_{13 + 2k}x_2^{-1}x_{14 + 2k}^{-1},
%x_{20}x_1^{-1}x_{21}^{-1}
\\
& x_{14 + 2k} x_1^{-1}x_{15 + 2k}^{-1}, \mathbf{x_{15 + 2k}x_3},
\textcolor{darkcerulean}{x_{18 +  2k}x_1^{-1}x_2^{-1}  x_{18 +  2k}^{-1} x_3^{-1}x_{15}^{-1}}
\rangle\\
\sim_{Q^{**}}&
\langle x_{1}, \textcolor{darkcerulean}{x_2}, x_5, x_6, x_7, x_9, x_{10}, x_{12}, x_{13}, x_{15}, \dots,  \mathbf{x_{14+2k}},\textcolor{darkcerulean}{x_{15+2k}},  \textcolor{darkcerulean}{x_{18+2k}}~|~\\
&x_5x_2^{-1}x_6^{-1}, 
x_6x_1^{-1}x_7^{-1}, 
x_7x_2^{-1}, 
x_1^{-1}x_9^{-1}, 
x_9x_2^{-1}  x_{18 +  2k}^{-1}x_{10}^{-1}, 
x_{10}x_{15 + 2k}, 
\textcolor{darkcerulean}{x_1x_{12}^{-1}}, 
x_{12}x_2x_{13}^{-1}, 
x_{13}x_{15 + 2k}^{-1}, 
\\
& \textcolor{darkcerulean}{ x_{18 +  2k}x_2x_5^{-1}},
% segunda relacion
x_{15}x_2^{-1}  x_{18 +  2k}^{-1}x_{16}^{-1},
x_{16}x_{15 + 2k}x_{17}^{-1},
x_{17}x_2^{-1}x_{18}^{-1},
\dots,
x_{13 + 2k}x_2^{-1}x_{14 + 2k}^{-1},
\\
&\mathbf{x_{14 + 2k} x_1^{-1}x_{15 + 2k}^{-1}},\textcolor{darkcerulean}{x_{18 +  2k}x_1^{-1}x_2^{-1}  x_{18 +  2k}^{-1} x_{15 + 2k}x_{15}^{-1}}
\rangle\\
\sim_{Q^{**}}&
\langle x_{1}, \textcolor{darkcerulean}{x_2}, x_5, x_6, x_7, x_9, x_{10}, x_{12}, x_{13}, x_{15}, \dots,  \mathbf{x_{13+2k}}, \textcolor{darkcerulean}{x_{15+2k}},  \textcolor{darkcerulean}{x_{18+2k}}~|~\\
&x_5x_2^{-1}x_6^{-1}, 
x_6x_1^{-1}x_7^{-1}, 
x_7x_2^{-1}, 
x_1^{-1}x_9^{-1}, 
x_9x_2^{-1}  x_{18 +  2k}^{-1}x_{10}^{-1}, 
x_{10}x_{15 + 2k}, 
\textcolor{darkcerulean}{x_1x_{12}^{-1}}, 
x_{12}x_2x_{13}^{-1}, 
x_{13}x_{15 + 2k}^{-1}, \\
&\textcolor{darkcerulean}{ x_{18 +  2k}x_2x_5^{-1}},
% segunda relacion
x_{15}x_2^{-1}  x_{18 +  2k}^{-1}x_{16}^{-1},
x_{16}x_{15 + 2k}x_{17}^{-1},
x_{17}x_2^{-1}x_{18}^{-1},
\dots,
x_{12+2k}x_1^{-1}x_{13+2k}^{-1},
%x_{19}x_2^{-1}x_{20}^{-1}
%x_{20}x_1^{-1}x_{21}^{-1}
\\
&\mathbf{x_{13 + 2k}x_2^{-1}x_1^{-1}x_{15 + 2k}^{-1}},
\textcolor{darkcerulean}{x_{18 +  2k}x_1^{-1}x_2^{-1}  x_{18 +  2k}^{-1} x_{15 + 2k}x_{15}^{-1}}
\rangle\\
\sim_{Q^{**}}&
\langle x_{1}, \textcolor{darkcerulean}{x_2}, x_5, x_6, x_7, x_9, x_{10}, x_{12}, x_{13}, x_{15}, \dots , \mathbf{x_{12+2k}}, \textcolor{darkcerulean}{x_{15+2k}},  \textcolor{darkcerulean}{x_{18+2k}}~|~\\
&x_5x_2^{-1}x_6^{-1}, 
x_6x_1^{-1}x_7^{-1}, 
x_7x_2^{-1}, 
x_1^{-1}x_9^{-1}, 
x_9x_2^{-1}  x_{18 +  2k}^{-1}x_{10}^{-1}, 
x_{10}x_{15 + 2k}, 
\textcolor{darkcerulean}{x_1x_{12}^{-1}}, 
x_{12}x_2x_{13}^{-1}, 
x_{13}x_{15 + 2k}^{-1}, \\
&\textcolor{darkcerulean}{ x_{18 +  2k}x_2x_5^{-1}},
% segunda relacion
x_{15}x_2^{-1}  x_{18 +  2k}^{-1}x_{16}^{-1},
x_{16}x_{15 + 2k}x_{17}^{-1},
x_{17}x_2^{-1}x_{18}^{-1},
\dots,
\mathbf{x_{12+2k}x_1^{-1}x_2^{-1}x_1^{-1}x_{15 + 2k}^{-1}},
%x_{19}x_2^{-1}x_{20}^{-1}
%x_{20}x_1^{-1}x_{21}^{-1}
\\
&\textcolor{darkcerulean}{x_{18 +  2k}x_1^{-1}x_2^{-1}  x_{18 +  2k}^{-1} x_{15 + 2k}x_{15}^{-1}}
\rangle\\
\dots\\
\sim_{Q^{**}}&
\langle x_{1}, \textcolor{darkcerulean}{x_2}, x_5, x_6, x_7, x_9, x_{10}, x_{12}, x_{13}, x_{15}, x_{16}, \mathbf{x_{17}}, \textcolor{darkcerulean}{x_{15+2k}},  \textcolor{darkcerulean}{x_{18+2k}}~|~\\
&x_5x_2^{-1}x_6^{-1}, 
x_6x_1^{-1}x_7^{-1}, 
x_7x_2^{-1}, 
x_1^{-1}x_9^{-1}, 
x_9x_2^{-1}  x_{18 +  2k}^{-1}x_{10}^{-1}, 
x_{10}x_{15 + 2k}, 
\textcolor{darkcerulean}{x_1x_{12}^{-1}}, 
x_{12}x_2x_{13}^{-1}, 
x_{13}x_{15 + 2k}^{-1}, \\
&\textcolor{darkcerulean}{ x_{18 +  2k}x_2x_5^{-1}},
% segunda relacion
x_{15}x_2^{-1}  x_{18 +  2k}^{-1}x_{16}^{-1},
x_{16}x_{15 + 2k}x_{17}^{-1},
\mathbf{x_{17}(x_2^{-1}x_1^{-1})^{k-1}x_{15+2k}^{-1}},\\
&
\textcolor{darkcerulean}{x_{18 +  2k}x_1^{-1}x_2^{-1}  x_{18 +  2k}^{-1} x_{15 + 2k}x_{15}^{-1}}
\rangle\\
\sim_{Q^{**}}&
\langle x_{1}, \textcolor{darkcerulean}{x_2}, x_5, x_6, x_7, x_9, x_{10}, x_{12}, x_{13}, x_{15}, \mathbf{x_{16}}, \textcolor{darkcerulean}{x_{15+2k}},  \textcolor{darkcerulean}{x_{18+2k}}~|~\\
&x_5x_2^{-1}x_6^{-1}, 
x_6x_1^{-1}x_7^{-1}, 
x_7x_2^{-1}, 
x_1^{-1}x_9^{-1}, 
x_9x_2^{-1}  x_{18 +  2k}^{-1}x_{10}^{-1}, 
x_{10}x_{15 + 2k}, 
\textcolor{darkcerulean}{x_1x_{12}^{-1}}, 
x_{12}x_2x_{13}^{-1}, 
x_{13}x_{15 + 2k}^{-1}, \\
&\textcolor{darkcerulean}{ x_{18 +  2k}x_2x_5^{-1}},
% segunda relacion
x_{15}x_2^{-1}  x_{18 +  2k}^{-1}x_{16}^{-1},
\mathbf{x_{16}x_{15 + 2k}(x_2^{-1}x_1^{-1})^{k-1}x_{15+2k}^{-1}},
\textcolor{darkcerulean}{x_{18 +  2k}x_1^{-1}x_2^{-1}  x_{18 +  2k}^{-1} x_{15 + 2k}x_{15}^{-1}}
\rangle\\
\end{split}
\end{equation*}

\begin{equation*}
\begin{split}
\sim_{Q^{**}}&
\langle x_{1}, \textcolor{darkcerulean}{x_2}, x_5, x_6, x_7, x_9, x_{10}, x_{12}, x_{13}, \mathbf{x_{15}}, \textcolor{darkcerulean}{x_{15+2k}},  \textcolor{darkcerulean}{x_{18+2k}}~|~\\
&x_5x_2^{-1}x_6^{-1}, 
x_6x_1^{-1}x_7^{-1}, 
x_7x_2^{-1}, 
x_1^{-1}x_9^{-1}, 
x_9x_2^{-1}  x_{18 +  2k}^{-1}x_{10}^{-1}, 
x_{10}x_{15 + 2k}, 
\textcolor{darkcerulean}{x_1x_{12}^{-1}}, 
x_{12}x_2x_{13}^{-1}, 
x_{13}x_{15 + 2k}^{-1}, \\
& \textcolor{darkcerulean}{ x_{18 +  2k}x_2x_5^{-1}},
\mathbf{x_{15}x_2^{-1}  x_{18 +  2k}^{-1}x_{15 + 2k}(x_2^{-1}x_1^{-1})^{k-1}x_{15+2k}^{-1}},
\textcolor{darkcerulean}{x_{18 +  2k}x_1^{-1}x_2^{-1}  x_{18 +  2k}^{-1} x_{15 + 2k}x_{15}^{-1}}
\rangle\\
\sim_{Q^{**}}&\langle x_{1}, \textcolor{darkcerulean}{x_2}, x_5, x_6, x_7, x_9, x_{10}, x_{12}, \mathbf{x_{13}} , \textcolor{darkcerulean}{x_{15+2k}},  \textcolor{darkcerulean}{x_{18+2k}}~|~\\
&x_5x_2^{-1}x_6^{-1}, 
x_6x_1^{-1}x_7^{-1}, 
x_7x_2^{-1}, 
x_1^{-1}x_9^{-1}, 
x_9x_2^{-1}  x_{18 +  2k}^{-1}x_{10}^{-1}, 
x_{10}x_{15 + 2k}, 
\textcolor{darkcerulean}{x_1x_{12}^{-1}}, 
x_{12}x_2x_{13}^{-1}, 
\mathbf{x_{13}x_{15 + 2k}^{-1}}, \\
&\textcolor{darkcerulean}{ x_{18 +  2k}x_2x_5^{-1}},
\textcolor{darkcerulean}{x_{18 +  2k}x_1^{-1}x_2^{-1}  x_{18 +  2k}^{-1} x_{15 + 2k}x_2^{-1}  x_{18 +  2k}^{-1}x_{15 + 2k}(x_2^{-1}x_1^{-1})^{k-1}x_{15+2k}^{-1}}
\rangle\\
 \sim_{Q^{**}}&
\langle x_{1}, \textcolor{darkcerulean}{x_2}, x_5, x_6, x_7, x_9, x_{10}, \mathbf{x_{12}}, \textcolor{darkcerulean}{x_{15+2k}},  \textcolor{darkcerulean}{x_{18+2k}}~|~\\
&x_5x_2^{-1}x_6^{-1}, 
x_6x_1^{-1}x_7^{-1}, 
x_7x_2^{-1}, 
x_1^{-1}x_9^{-1}, 
x_9x_2^{-1}  x_{18 +  2k}^{-1}x_{10}^{-1}, 
x_{10}x_{15 + 2k}, 
\textcolor{darkcerulean}{x_1x_{12}^{-1}}, 
\mathbf{x_{12}x_2x_{15 + 2k}^{-1}}, \\
&\textcolor{darkcerulean}{ x_{18 +  2k}x_2x_5^{-1}},
\textcolor{darkcerulean}{x_{18 +  2k}x_1^{-1}x_2^{-1}  x_{18 +  2k}^{-1} x_{15 + 2k}x_2^{-1}  x_{18 +  2k}^{-1}x_{15 + 2k}(x_2^{-1}x_1^{-1})^{k-1}x_{15+2k}^{-1}}
\rangle\\
\sim_{Q^{**}}&
\langle x_{1}, \textcolor{darkcerulean}{x_2}, x_5, x_6, x_7, x_9, \mathbf{x_{10}}, \textcolor{darkcerulean}{x_{15+2k}},  \textcolor{darkcerulean}{x_{18+2k}}~|~\\
&x_5x_2^{-1}x_6^{-1}, 
x_6x_1^{-1}x_7^{-1}, 
x_7x_2^{-1}, 
x_1^{-1}x_9^{-1}, 
x_9x_2^{-1}  x_{18 +  2k}^{-1}x_{10}^{-1}, 
\mathbf{x_{10}x_{15 + 2k}}, 
\textcolor{darkcerulean}{x_1x_2x_{15 + 2k}^{-1}}, \\
&\textcolor{darkcerulean}{ x_{18 +  2k}x_2x_5^{-1}},
\textcolor{darkcerulean}{x_{18 +  2k}x_1^{-1}x_2^{-1}  x_{18 +  2k}^{-1} x_{15 + 2k}x_2^{-1}  x_{18 +  2k}^{-1}x_{15 + 2k}(x_2^{-1}x_1^{-1})^{k-1}x_{15+2k}^{-1}}
\rangle\\
\sim_{Q^{**}}&
\langle x_{1}, \textcolor{darkcerulean}{x_2}, x_5, x_6, x_7, \mathbf{x_9}, \textcolor{darkcerulean}{x_{15+2k}},  \textcolor{darkcerulean}{x_{18+2k}}~|~\\
&x_5x_2^{-1}x_6^{-1}, 
x_6x_1^{-1}x_7^{-1}, 
x_7x_2^{-1}, 
x_1^{-1}x_9^{-1}, 
\mathbf{x_9x_2^{-1}  x_{18 +  2k}^{-1}x_{15 + 2k}}, 
\textcolor{darkcerulean}{x_1x_2x_{15 + 2k}^{-1}}, \\
&\textcolor{darkcerulean}{ x_{18 +  2k}x_2x_5^{-1}},
\textcolor{darkcerulean}{x_{18 +  2k}x_1^{-1}x_2^{-1}  x_{18 +  2k}^{-1} x_{15 + 2k}x_2^{-1}  x_{18 +  2k}^{-1}x_{15 + 2k}(x_2^{-1}x_1^{-1})^{k-1}x_{15+2k}^{-1}}
\rangle\\
\sim_{Q^{**}}&
\langle \mathbf{x_{1}}, \textcolor{darkcerulean}{x_2}, x_5, x_6, x_7, \textcolor{darkcerulean}{x_{15+2k}},  \textcolor{darkcerulean}{x_{18+2k}}~|~\\
&x_5x_2^{-1}x_6^{-1}, 
x_6x_1^{-1}x_7^{-1}, 
x_7x_2^{-1}, 
\mathbf{x_1^{-1}x_2^{-1}  x_{18 +  2k}^{-1}x_{15 + 2k}}, 
\textcolor{darkcerulean}{x_1x_2x_{15 + 2k}^{-1}}, \\
&\textcolor{darkcerulean}{ x_{18 +  2k}x_2x_5^{-1}},
\textcolor{darkcerulean}{x_{18 +  2k}x_1^{-1}x_2^{-1}  x_{18 +  2k}^{-1} x_{15 + 2k}x_2^{-1}  x_{18 +  2k}^{-1}x_{15 + 2k}(x_2^{-1}x_1^{-1})^{k-1}x_{15+2k}^{-1}}
\rangle\\
\sim_{Q^{**}}&
\langle \textcolor{darkcerulean}{x_2}, x_5, x_6, \mathbf{x_7}, \textcolor{darkcerulean}{x_{15+2k}},  \textcolor{darkcerulean}{x_{18+2k}}~|~\\
&x_5x_2^{-1}x_6^{-1}, 
x_6x_{15+2}^{-1}x_{18+2k}x_2x_7^{-1}, 
\mathbf{x_7x_2^{-1}}, 
\textcolor{darkcerulean}{x_2^{-1}  x_{18 +  2k}^{-1}x_{15 + 2k}x_2x_{15 + 2k}^{-1}}, \\
&\textcolor{darkcerulean}{ x_{18 +  2k}x_2x_5^{-1}},
\textcolor{darkcerulean}{x_{18 +  2k} x_{15 + 2k}^{-1}x_{18 +  2k}x_2x_2^{-1}  x_{18 +  2k}^{-1} x_{15 + 2k}x_2^{-1}  x_{18 +  2k}^{-1}x_{15 + 2k}(x_2^{-1}x_{15+2}^{-1}x_{18+2k}x_2)^{k-1}x_{15+2k}^{-1}}
\rangle\\
\sim_{Q^{**}}&
\langle  \textcolor{darkcerulean}{x_2}, x_5, \mathbf{x_6}, \textcolor{darkcerulean}{x_{15+2k}},  \textcolor{darkcerulean}{x_{18+2k}}~|~\\
&x_5x_2^{-1}x_6^{-1}, 
\mathbf{x_6x_{15+2}^{-1}x_{18+2k}x_2x_2^{-1}}, 
\textcolor{darkcerulean}{x_2^{-1}  x_{18 +  2k}^{-1}x_{15 + 2k}x_2x_{15 + 2k}^{-1}}, \\
&\textcolor{darkcerulean}{ x_{18 +  2k}x_2x_5^{-1}},
\textcolor{darkcerulean}{x_{18 +  2k} x_{15 + 2k}^{-1}x_{18 +  2k}x_2x_2^{-1}  x_{18 +  2k}^{-1} x_{15 + 2k}x_2^{-1}  x_{18 +  2k}^{-1}x_{15 + 2k}(x_2^{-1}x_{15+2}^{-1}x_{18+2k}x_2)^{k-1}x_{15+2k}^{-1}}
\rangle\\
\sim_{Q^{**}}&
\langle \textcolor{darkcerulean}{x_2}, \mathbf{x_5}, \textcolor{darkcerulean}{x_{15+2k}},  \textcolor{darkcerulean}{x_{18+2k}}~|~\\
&\mathbf{x_5x_2^{-1}x_{15+2}^{-1}x_{18+2k}}, 
\textcolor{darkcerulean}{x_2^{-1}  x_{18 +  2k}^{-1}x_{15 + 2k}x_2x_{15 + 2k}^{-1}}, \\
&\textcolor{darkcerulean}{ x_{18 +  2k}x_2x_5^{-1}},
\textcolor{darkcerulean}{x_{18 +  2k} x_{15 + 2k}^{-1}x_{18 +  2k}x_2x_2^{-1}  x_{18 +  2k}^{-1} x_{15 + 2k}x_2^{-1}  x_{18 +  2k}^{-1}x_{15 + 2k}(x_2^{-1}x_{15+2}^{-1}x_{18+2k}x_2)^{k-1}x_{15+2k}^{-1}}
\rangle\\
\sim_{Q^{**}}&
\langle \textcolor{darkcerulean}{x_2}, \textcolor{darkcerulean}{x_{15+2k}},  \textcolor{darkcerulean}{x_{18+2k}}~|~\\
&\textcolor{darkcerulean}{x_2^{-1}  x_{18 +  2k}^{-1}x_{15 + 2k}x_2x_{15 + 2k}^{-1}}, ~\textcolor{darkcerulean}{ x_{18 +  2k}x_2x_2^{-1}x_{15+2}^{-1}x_{18+2k}},\\
&\textcolor{darkcerulean}{x_{18 +  2k} x_{15 + 2k}^{-1}x_{18 +  2k}x_2x_2^{-1}  x_{18 +  2k}^{-1} x_{15 + 2k}x_2^{-1}  x_{18 +  2k}^{-1}x_{15 + 2k}(x_2^{-1}x_{15+2}^{-1}x_{18+2k}x_2)^{k-1}x_{15+2k}^{-1}}
\rangle\\
\end{split}
\end{equation*}
\end{footnotesize}

Therefore, \[\begin{split}\Q_{K'_{\tilde \G_{k}},M_k} \sim_{Q^{**}} \langle x_2, x_{15+2k},  x_{18+2k}~|~& x_{18 +  2k}x_{15+2}^{-1}x_{18+2k},~
x_2x_{15 + 2k}^{-1} x_2^{-1} x_{18 +  2k}^{-1}x_{15 + 2k}, ~ \\
& x_{18 +  2k} x_2^{-1}  x_{18 +  2k}^{-1}x_{15 + 2k}x_2^{-1}(x_{15+2}^{-1}x_{18+2k})^{k-1}x_2x_{15+2k}^{-1}
\rangle
\end{split}\]

\medskip

{\bf Acknowledgements.}
The author is grateful to Eugenio Borghini and Gabriel Minian  for many useful discussions and suggestions during the preparation of this article. This work was partially supported by the EPSRC grant New Approaches to Data Science: Application Driven Topological Data Analysis EP/R018472/1.

\bibliographystyle{plain}
\bibliography{biblio}

\end{document}

%% file: torus.pdf_tex
%% Creator: Inkscape 1.0.2-2 (e86c870879, 2021-01-15), www.inkscape.org
%% PDF/EPS/PS + LaTeX output extension by Johan Engelen, 2010
%% Accompanies image file 'torus.pdf' (pdf, eps, ps)
%%
%% To include the image in your LaTeX document, write
%%   \input{<filename>.pdf_tex}
%%  instead of
%%   \includegraphics{<filename>.pdf}
%% To scale the image, write
%%   \def\svgwidth{<desired width>}
%%   \input{<filename>.pdf_tex}
%%  instead of
%%   \includegraphics[width=<desired width>]{<filename>.pdf}
%%
%% Images with a different path to the parent latex file can
%% be accessed with the `import' package (which may need to be
%% installed) using
%%   \usepackage{import}
%% in the preamble, and then including the image with
%%   \import{<path to file>}{<filename>.pdf_tex}
%% Alternatively, one can specify
%%   \graphicspath{{<path to file>/}}
%% 
%% For more information, please see info/svg-inkscape on CTAN:
%%   http://tug.ctan.org/tex-archive/info/svg-inkscape
%%
\begingroup%
  \makeatletter%
  \providecommand\color[2][]{%
    \errmessage{(Inkscape) Color is used for the text in Inkscape, but the package 'color.sty' is not loaded}%
    \renewcommand\color[2][]{}%
  }%
  \providecommand\transparent[1]{%
    \errmessage{(Inkscape) Transparency is used (non-zero) for the text in Inkscape, but the package 'transparent.sty' is not loaded}%
    \renewcommand\transparent[1]{}%
  }%
  \providecommand\rotatebox[2]{#2}%
  \newcommand*\fsize{\dimexpr\f@size pt\relax}%
  \newcommand*\lineheight[1]{\fontsize{\fsize}{#1\fsize}\selectfont}%
  \ifx\svgwidth\undefined%
    \setlength{\unitlength}{250.3181013bp}%
    \ifx\svgscale\undefined%
      \relax%
    \else%
      \setlength{\unitlength}{\unitlength * \real{\svgscale}}%
    \fi%
  \else%
    \setlength{\unitlength}{\svgwidth}%
  \fi%
  \global\let\svgwidth\undefined%
  \global\let\svgscale\undefined%
  \makeatother%
  \begin{picture}(1,0.99798515)%
    \lineheight{1}%
    \setlength\tabcolsep{0pt}%
    \put(0,0){\includegraphics[width=\unitlength,page=1]{torus.pdf}}%
  \end{picture}%
\endgroup%

%% file: torus_labeled_v2.pdf_tex
%% Creator: Inkscape 1.0.2-2 (e86c870879, 2021-01-15), www.inkscape.org
%% PDF/EPS/PS + LaTeX output extension by Johan Engelen, 2010
%% Accompanies image file 'torus_labeled_v2.pdf' (pdf, eps, ps)
%%
%% To include the image in your LaTeX document, write
%%   \input{<filename>.pdf_tex}
%%  instead of
%%   \includegraphics{<filename>.pdf}
%% To scale the image, write
%%   \def\svgwidth{<desired width>}
%%   \input{<filename>.pdf_tex}
%%  instead of
%%   \includegraphics[width=<desired width>]{<filename>.pdf}
%%
%% Images with a different path to the parent latex file can
%% be accessed with the `import' package (which may need to be
%% installed) using
%%   \usepackage{import}
%% in the preamble, and then including the image with
%%   \import{<path to file>}{<filename>.pdf_tex}
%% Alternatively, one can specify
%%   \graphicspath{{<path to file>/}}
%% 
%% For more information, please see info/svg-inkscape on CTAN:
%%   http://tug.ctan.org/tex-archive/info/svg-inkscape
%%
\begingroup%
  \makeatletter%
  \providecommand\color[2][]{%
    \errmessage{(Inkscape) Color is used for the text in Inkscape, but the package 'color.sty' is not loaded}%
    \renewcommand\color[2][]{}%
  }%
  \providecommand\transparent[1]{%
    \errmessage{(Inkscape) Transparency is used (non-zero) for the text in Inkscape, but the package 'transparent.sty' is not loaded}%
    \renewcommand\transparent[1]{}%
  }%
  \providecommand\rotatebox[2]{#2}%
  \newcommand*\fsize{\dimexpr\f@size pt\relax}%
  \newcommand*\lineheight[1]{\fontsize{\fsize}{#1\fsize}\selectfont}%
  \ifx\svgwidth\undefined%
    \setlength{\unitlength}{316.56747653bp}%
    \ifx\svgscale\undefined%
      \relax%
    \else%
      \setlength{\unitlength}{\unitlength * \real{\svgscale}}%
    \fi%
  \else%
    \setlength{\unitlength}{\svgwidth}%
  \fi%
  \global\let\svgwidth\undefined%
  \global\let\svgscale\undefined%
  \makeatother%
  \begin{picture}(1,0.88349967)%
    \lineheight{1}%
    \setlength\tabcolsep{0pt}%
    \put(0,0){\includegraphics[width=\unitlength,page=1]{torus_labeled_v2.pdf}}%
    \put(0.00747443,0.85147146){\color[rgb]{0,0,0}\makebox(0,0)[lt]{\lineheight{1.25}\smash{\begin{tabular}[t]{l}$v_1$\end{tabular}}}}%
    \put(0.89204335,0.84879653){\color[rgb]{0,0,0}\makebox(0,0)[lt]{\lineheight{1.25}\smash{\begin{tabular}[t]{l}$v_1$\end{tabular}}}}%
    \put(0.89388115,0.01605796){\color[rgb]{0,0,0}\makebox(0,0)[lt]{\lineheight{1.25}\smash{\begin{tabular}[t]{l}$v_1$\end{tabular}}}}%
    \put(0.00697068,0.01070806){\color[rgb]{0,0,0}\makebox(0,0)[lt]{\lineheight{1.25}\smash{\begin{tabular}[t]{l}$v_1$\end{tabular}}}}%
    \put(0.03039829,0.24637027){\color[rgb]{0,0,0}\makebox(0,0)[lt]{\lineheight{1.25}\smash{\begin{tabular}[t]{l}$3$\end{tabular}}}}%
    \put(0.90155147,0.24318038){\color[rgb]{0,0,0}\makebox(0,0)[lt]{\lineheight{1.25}\smash{\begin{tabular}[t]{l}$3$\end{tabular}}}}%
    \put(-0.00239076,0.60528581){\color[rgb]{0,0,0}\makebox(0,0)[lt]{\lineheight{1.25}\smash{\begin{tabular}[t]{l}$x_4$\end{tabular}}}}%
    \put(0.90316662,0.59835185){\color[rgb]{0,0,0}\makebox(0,0)[lt]{\lineheight{1.25}\smash{\begin{tabular}[t]{l}$x_4$\end{tabular}}}}%
    \put(0.26637409,0.8557937){\color[rgb]{0,0,0}\makebox(0,0)[lt]{\lineheight{1.25}\smash{\begin{tabular}[t]{l}$5$\end{tabular}}}}%
    \put(0.64035019,0.85919107){\color[rgb]{0,0,0}\makebox(0,0)[lt]{\lineheight{1.25}\smash{\begin{tabular}[t]{l}$x_6$\end{tabular}}}}%
    \put(0.63348394,0.01029323){\color[rgb]{0,0,0}\makebox(0,0)[lt]{\lineheight{1.25}\smash{\begin{tabular}[t]{l}$x_6$\end{tabular}}}}%
    \put(0.26572722,0.00556816){\color[rgb]{0,0,0}\makebox(0,0)[lt]{\lineheight{1.25}\smash{\begin{tabular}[t]{l}$5$\end{tabular}}}}%
    \put(0.30940874,0.50206314){\color[rgb]{0,0,0}\makebox(0,0)[lt]{\lineheight{1.25}\smash{\begin{tabular}[t]{l}$9$\end{tabular}}}}%
    \put(0.51551851,0.56770377){\color[rgb]{0,0,0}\makebox(0,0)[lt]{\lineheight{1.25}\smash{\begin{tabular}[t]{l}$8$\end{tabular}}}}%
    \put(0,0){\includegraphics[width=\unitlength,page=2]{torus_labeled_v2.pdf}}%
    \put(0.69394472,0.3640935){\color[rgb]{0,0,0}\makebox(0,0)[lt]{\lineheight{1.25}\smash{\begin{tabular}[t]{l}$7$\end{tabular}}}}%
    \put(0,0){\includegraphics[width=\unitlength,page=3]{torus_labeled_v2.pdf}}%
    \put(0.23451158,0.25325121){\color[rgb]{0,0,0}\makebox(0,0)[lt]{\lineheight{1.25}\smash{\begin{tabular}[t]{l}$e_{10}$\end{tabular}}}}%
    \put(0.51797594,0.24282574){\color[rgb]{0,0,0}\makebox(0,0)[lt]{\lineheight{1.25}\smash{\begin{tabular}[t]{l}$2$\end{tabular}}}}%
    \put(0,0){\includegraphics[width=\unitlength,page=4]{torus_labeled_v2.pdf}}%
  \end{picture}%
\endgroup%

%% file: torus_morse_1.pdf_tex
%% Creator: Inkscape 1.0.2-2 (e86c870879, 2021-01-15), www.inkscape.org
%% PDF/EPS/PS + LaTeX output extension by Johan Engelen, 2010
%% Accompanies image file 'torus_morse_1.pdf' (pdf, eps, ps)
%%
%% To include the image in your LaTeX document, write
%%   \input{<filename>.pdf_tex}
%%  instead of
%%   \includegraphics{<filename>.pdf}
%% To scale the image, write
%%   \def\svgwidth{<desired width>}
%%   \input{<filename>.pdf_tex}
%%  instead of
%%   \includegraphics[width=<desired width>]{<filename>.pdf}
%%
%% Images with a different path to the parent latex file can
%% be accessed with the `import' package (which may need to be
%% installed) using
%%   \usepackage{import}
%% in the preamble, and then including the image with
%%   \import{<path to file>}{<filename>.pdf_tex}
%% Alternatively, one can specify
%%   \graphicspath{{<path to file>/}}
%% 
%% For more information, please see info/svg-inkscape on CTAN:
%%   http://tug.ctan.org/tex-archive/info/svg-inkscape
%%
\begingroup%
  \makeatletter%
  \providecommand\color[2][]{%
    \errmessage{(Inkscape) Color is used for the text in Inkscape, but the package 'color.sty' is not loaded}%
    \renewcommand\color[2][]{}%
  }%
  \providecommand\transparent[1]{%
    \errmessage{(Inkscape) Transparency is used (non-zero) for the text in Inkscape, but the package 'transparent.sty' is not loaded}%
    \renewcommand\transparent[1]{}%
  }%
  \providecommand\rotatebox[2]{#2}%
  \newcommand*\fsize{\dimexpr\f@size pt\relax}%
  \newcommand*\lineheight[1]{\fontsize{\fsize}{#1\fsize}\selectfont}%
  \ifx\svgwidth\undefined%
    \setlength{\unitlength}{316.56747653bp}%
    \ifx\svgscale\undefined%
      \relax%
    \else%
      \setlength{\unitlength}{\unitlength * \real{\svgscale}}%
    \fi%
  \else%
    \setlength{\unitlength}{\svgwidth}%
  \fi%
  \global\let\svgwidth\undefined%
  \global\let\svgscale\undefined%
  \makeatother%
  \begin{picture}(1,0.88349967)%
    \lineheight{1}%
    \setlength\tabcolsep{0pt}%
    \put(0,0){\includegraphics[width=\unitlength,page=1]{torus_morse_1.pdf}}%
    \put(0.00747443,0.85147146){\color[rgb]{0,0,0}\makebox(0,0)[lt]{\lineheight{1.25}\smash{\begin{tabular}[t]{l}$v_1$\end{tabular}}}}%
    \put(0.89204335,0.84879653){\color[rgb]{0,0,0}\makebox(0,0)[lt]{\lineheight{1.25}\smash{\begin{tabular}[t]{l}$v_1$\end{tabular}}}}%
    \put(0.89388115,0.01605796){\color[rgb]{0,0,0}\makebox(0,0)[lt]{\lineheight{1.25}\smash{\begin{tabular}[t]{l}$v_1$\end{tabular}}}}%
    \put(0.00697068,0.01070806){\color[rgb]{0,0,0}\makebox(0,0)[lt]{\lineheight{1.25}\smash{\begin{tabular}[t]{l}$v_1$\end{tabular}}}}%
    \put(0.03039829,0.24637027){\color[rgb]{0,0,0}\makebox(0,0)[lt]{\lineheight{1.25}\smash{\begin{tabular}[t]{l}$3$\end{tabular}}}}%
    \put(0.90155147,0.24318038){\color[rgb]{0,0,0}\makebox(0,0)[lt]{\lineheight{1.25}\smash{\begin{tabular}[t]{l}$3$\end{tabular}}}}%
    \put(-0.00239076,0.60528581){\color[rgb]{0,0,0}\makebox(0,0)[lt]{\lineheight{1.25}\smash{\begin{tabular}[t]{l}$x_4$\end{tabular}}}}%
    \put(0.90316662,0.59835185){\color[rgb]{0,0,0}\makebox(0,0)[lt]{\lineheight{1.25}\smash{\begin{tabular}[t]{l}$x_4$\end{tabular}}}}%
    \put(0.26637409,0.8557937){\color[rgb]{0,0,0}\makebox(0,0)[lt]{\lineheight{1.25}\smash{\begin{tabular}[t]{l}$5$\end{tabular}}}}%
    \put(0.64035019,0.85919107){\color[rgb]{0,0,0}\makebox(0,0)[lt]{\lineheight{1.25}\smash{\begin{tabular}[t]{l}$x_6$\end{tabular}}}}%
    \put(0.63348394,0.01029323){\color[rgb]{0,0,0}\makebox(0,0)[lt]{\lineheight{1.25}\smash{\begin{tabular}[t]{l}$x_6$\end{tabular}}}}%
    \put(0.26572722,0.00556816){\color[rgb]{0,0,0}\makebox(0,0)[lt]{\lineheight{1.25}\smash{\begin{tabular}[t]{l}$5$\end{tabular}}}}%
    \put(0.51551851,0.56770377){\color[rgb]{0,0,0}\makebox(0,0)[lt]{\lineheight{1.25}\smash{\begin{tabular}[t]{l}$8$\end{tabular}}}}%
    \put(0,0){\includegraphics[width=\unitlength,page=2]{torus_morse_1.pdf}}%
    \put(0.69394472,0.3640935){\color[rgb]{0,0,0}\makebox(0,0)[lt]{\lineheight{1.25}\smash{\begin{tabular}[t]{l}$7$\end{tabular}}}}%
    \put(0,0){\includegraphics[width=\unitlength,page=3]{torus_morse_1.pdf}}%
    \put(0.51797594,0.24282574){\color[rgb]{0,0,0}\makebox(0,0)[lt]{\lineheight{1.25}\smash{\begin{tabular}[t]{l}$2$\end{tabular}}}}%
    \put(0,0){\includegraphics[width=\unitlength,page=4]{torus_morse_1.pdf}}%
    \put(0.24647427,0.44146688){\color[rgb]{0,0,0}\makebox(0,0)[lt]{\lineheight{1.25}\smash{\begin{tabular}[t]{l}$\tilde{e}_{10}$\end{tabular}}}}%
    \put(0,0){\includegraphics[width=\unitlength,page=5]{torus_morse_1.pdf}}%
  \end{picture}%
\endgroup%

%% file: torus_morse_2.pdf_tex
%% Creator: Inkscape 1.0.2-2 (e86c870879, 2021-01-15), www.inkscape.org
%% PDF/EPS/PS + LaTeX output extension by Johan Engelen, 2010
%% Accompanies image file 'torus_morse_2.pdf' (pdf, eps, ps)
%%
%% To include the image in your LaTeX document, write
%%   \input{<filename>.pdf_tex}
%%  instead of
%%   \includegraphics{<filename>.pdf}
%% To scale the image, write
%%   \def\svgwidth{<desired width>}
%%   \input{<filename>.pdf_tex}
%%  instead of
%%   \includegraphics[width=<desired width>]{<filename>.pdf}
%%
%% Images with a different path to the parent latex file can
%% be accessed with the `import' package (which may need to be
%% installed) using
%%   \usepackage{import}
%% in the preamble, and then including the image with
%%   \import{<path to file>}{<filename>.pdf_tex}
%% Alternatively, one can specify
%%   \graphicspath{{<path to file>/}}
%% 
%% For more information, please see info/svg-inkscape on CTAN:
%%   http://tug.ctan.org/tex-archive/info/svg-inkscape
%%
\begingroup%
  \makeatletter%
  \providecommand\color[2][]{%
    \errmessage{(Inkscape) Color is used for the text in Inkscape, but the package 'color.sty' is not loaded}%
    \renewcommand\color[2][]{}%
  }%
  \providecommand\transparent[1]{%
    \errmessage{(Inkscape) Transparency is used (non-zero) for the text in Inkscape, but the package 'transparent.sty' is not loaded}%
    \renewcommand\transparent[1]{}%
  }%
  \providecommand\rotatebox[2]{#2}%
  \newcommand*\fsize{\dimexpr\f@size pt\relax}%
  \newcommand*\lineheight[1]{\fontsize{\fsize}{#1\fsize}\selectfont}%
  \ifx\svgwidth\undefined%
    \setlength{\unitlength}{316.56747653bp}%
    \ifx\svgscale\undefined%
      \relax%
    \else%
      \setlength{\unitlength}{\unitlength * \real{\svgscale}}%
    \fi%
  \else%
    \setlength{\unitlength}{\svgwidth}%
  \fi%
  \global\let\svgwidth\undefined%
  \global\let\svgscale\undefined%
  \makeatother%
  \begin{picture}(1,0.88349967)%
    \lineheight{1}%
    \setlength\tabcolsep{0pt}%
    \put(0,0){\includegraphics[width=\unitlength,page=1]{torus_morse_2.pdf}}%
    \put(0.00747443,0.85147146){\color[rgb]{0,0,0}\makebox(0,0)[lt]{\lineheight{1.25}\smash{\begin{tabular}[t]{l}$v_1$\end{tabular}}}}%
    \put(0.89204335,0.84879653){\color[rgb]{0,0,0}\makebox(0,0)[lt]{\lineheight{1.25}\smash{\begin{tabular}[t]{l}$v_1$\end{tabular}}}}%
    \put(0.89388115,0.01605796){\color[rgb]{0,0,0}\makebox(0,0)[lt]{\lineheight{1.25}\smash{\begin{tabular}[t]{l}$v_1$\end{tabular}}}}%
    \put(0.00697068,0.01070806){\color[rgb]{0,0,0}\makebox(0,0)[lt]{\lineheight{1.25}\smash{\begin{tabular}[t]{l}$v_1$\end{tabular}}}}%
    \put(0.03039829,0.24637027){\color[rgb]{0,0,0}\makebox(0,0)[lt]{\lineheight{1.25}\smash{\begin{tabular}[t]{l}$3$\end{tabular}}}}%
    \put(0.90155147,0.24318038){\color[rgb]{0,0,0}\makebox(0,0)[lt]{\lineheight{1.25}\smash{\begin{tabular}[t]{l}$3$\end{tabular}}}}%
    \put(-0.00239076,0.60528581){\color[rgb]{0,0,0}\makebox(0,0)[lt]{\lineheight{1.25}\smash{\begin{tabular}[t]{l}$x_4$\end{tabular}}}}%
    \put(0.90316662,0.59835185){\color[rgb]{0,0,0}\makebox(0,0)[lt]{\lineheight{1.25}\smash{\begin{tabular}[t]{l}$x_4$\end{tabular}}}}%
    \put(0.26637409,0.8557937){\color[rgb]{0,0,0}\makebox(0,0)[lt]{\lineheight{1.25}\smash{\begin{tabular}[t]{l}$5$\end{tabular}}}}%
    \put(0.64035019,0.85919107){\color[rgb]{0,0,0}\makebox(0,0)[lt]{\lineheight{1.25}\smash{\begin{tabular}[t]{l}$x_6$\end{tabular}}}}%
    \put(0.63348394,0.01029323){\color[rgb]{0,0,0}\makebox(0,0)[lt]{\lineheight{1.25}\smash{\begin{tabular}[t]{l}$x_6$\end{tabular}}}}%
    \put(0.26572722,0.00556816){\color[rgb]{0,0,0}\makebox(0,0)[lt]{\lineheight{1.25}\smash{\begin{tabular}[t]{l}$5$\end{tabular}}}}%
    \put(0,0){\includegraphics[width=\unitlength,page=2]{torus_morse_2.pdf}}%
    \put(0.69394472,0.3640935){\color[rgb]{0,0,0}\makebox(0,0)[lt]{\lineheight{1.25}\smash{\begin{tabular}[t]{l}$7$\end{tabular}}}}%
    \put(0.51797594,0.24282574){\color[rgb]{0,0,0}\makebox(0,0)[lt]{\lineheight{1.25}\smash{\begin{tabular}[t]{l}$2$\end{tabular}}}}%
    \put(0,0){\includegraphics[width=\unitlength,page=3]{torus_morse_2.pdf}}%
    \put(0.33810315,0.56402592){\color[rgb]{0,0,0}\makebox(0,0)[lt]{\lineheight{1.25}\smash{\begin{tabular}[t]{l}$\tilde{e}_{10}$\end{tabular}}}}%
  \end{picture}%
\endgroup%

%% file: torus_morse_3.pdf_tex
%% Creator: Inkscape 1.0.2-2 (e86c870879, 2021-01-15), www.inkscape.org
%% PDF/EPS/PS + LaTeX output extension by Johan Engelen, 2010
%% Accompanies image file 'torus_morse_3.pdf' (pdf, eps, ps)
%%
%% To include the image in your LaTeX document, write
%%   \input{<filename>.pdf_tex}
%%  instead of
%%   \includegraphics{<filename>.pdf}
%% To scale the image, write
%%   \def\svgwidth{<desired width>}
%%   \input{<filename>.pdf_tex}
%%  instead of
%%   \includegraphics[width=<desired width>]{<filename>.pdf}
%%
%% Images with a different path to the parent latex file can
%% be accessed with the `import' package (which may need to be
%% installed) using
%%   \usepackage{import}
%% in the preamble, and then including the image with
%%   \import{<path to file>}{<filename>.pdf_tex}
%% Alternatively, one can specify
%%   \graphicspath{{<path to file>/}}
%% 
%% For more information, please see info/svg-inkscape on CTAN:
%%   http://tug.ctan.org/tex-archive/info/svg-inkscape
%%
\begingroup%
  \makeatletter%
  \providecommand\color[2][]{%
    \errmessage{(Inkscape) Color is used for the text in Inkscape, but the package 'color.sty' is not loaded}%
    \renewcommand\color[2][]{}%
  }%
  \providecommand\transparent[1]{%
    \errmessage{(Inkscape) Transparency is used (non-zero) for the text in Inkscape, but the package 'transparent.sty' is not loaded}%
    \renewcommand\transparent[1]{}%
  }%
  \providecommand\rotatebox[2]{#2}%
  \newcommand*\fsize{\dimexpr\f@size pt\relax}%
  \newcommand*\lineheight[1]{\fontsize{\fsize}{#1\fsize}\selectfont}%
  \ifx\svgwidth\undefined%
    \setlength{\unitlength}{316.56747653bp}%
    \ifx\svgscale\undefined%
      \relax%
    \else%
      \setlength{\unitlength}{\unitlength * \real{\svgscale}}%
    \fi%
  \else%
    \setlength{\unitlength}{\svgwidth}%
  \fi%
  \global\let\svgwidth\undefined%
  \global\let\svgscale\undefined%
  \makeatother%
  \begin{picture}(1,0.88349967)%
    \lineheight{1}%
    \setlength\tabcolsep{0pt}%
    \put(0,0){\includegraphics[width=\unitlength,page=1]{torus_morse_3.pdf}}%
    \put(0.00747443,0.85147146){\color[rgb]{0,0,0}\makebox(0,0)[lt]{\lineheight{1.25}\smash{\begin{tabular}[t]{l}$v_1$\end{tabular}}}}%
    \put(0.89204335,0.84879653){\color[rgb]{0,0,0}\makebox(0,0)[lt]{\lineheight{1.25}\smash{\begin{tabular}[t]{l}$v_1$\end{tabular}}}}%
    \put(0.89388115,0.01605796){\color[rgb]{0,0,0}\makebox(0,0)[lt]{\lineheight{1.25}\smash{\begin{tabular}[t]{l}$v_1$\end{tabular}}}}%
    \put(0.00697068,0.01070806){\color[rgb]{0,0,0}\makebox(0,0)[lt]{\lineheight{1.25}\smash{\begin{tabular}[t]{l}$v_1$\end{tabular}}}}%
    \put(0.03039829,0.24637027){\color[rgb]{0,0,0}\makebox(0,0)[lt]{\lineheight{1.25}\smash{\begin{tabular}[t]{l}$3$\end{tabular}}}}%
    \put(0.90155147,0.24318038){\color[rgb]{0,0,0}\makebox(0,0)[lt]{\lineheight{1.25}\smash{\begin{tabular}[t]{l}$3$\end{tabular}}}}%
    \put(-0.00239076,0.60528581){\color[rgb]{0,0,0}\makebox(0,0)[lt]{\lineheight{1.25}\smash{\begin{tabular}[t]{l}$x_4$\end{tabular}}}}%
    \put(0.90316662,0.59835185){\color[rgb]{0,0,0}\makebox(0,0)[lt]{\lineheight{1.25}\smash{\begin{tabular}[t]{l}$x_4$\end{tabular}}}}%
    \put(0.26637409,0.8557937){\color[rgb]{0,0,0}\makebox(0,0)[lt]{\lineheight{1.25}\smash{\begin{tabular}[t]{l}$5$\end{tabular}}}}%
    \put(0.64035019,0.85919107){\color[rgb]{0,0,0}\makebox(0,0)[lt]{\lineheight{1.25}\smash{\begin{tabular}[t]{l}$x_6$\end{tabular}}}}%
    \put(0.63348394,0.01029323){\color[rgb]{0,0,0}\makebox(0,0)[lt]{\lineheight{1.25}\smash{\begin{tabular}[t]{l}$x_6$\end{tabular}}}}%
    \put(0.26572722,0.00556816){\color[rgb]{0,0,0}\makebox(0,0)[lt]{\lineheight{1.25}\smash{\begin{tabular}[t]{l}$5$\end{tabular}}}}%
    \put(0,0){\includegraphics[width=\unitlength,page=2]{torus_morse_3.pdf}}%
    \put(0.51651688,0.24720283){\color[rgb]{0,0,0}\makebox(0,0)[lt]{\lineheight{1.25}\smash{\begin{tabular}[t]{l}$2$\end{tabular}}}}%
    \put(0,0){\includegraphics[width=\unitlength,page=3]{torus_morse_3.pdf}}%
    \put(0.42157277,0.51150064){\color[rgb]{0,0,0}\makebox(0,0)[lt]{\lineheight{1.25}\smash{\begin{tabular}[t]{l}$\tilde{e}_{10}$\end{tabular}}}}%
    \put(0,0){\includegraphics[width=\unitlength,page=4]{torus_morse_3.pdf}}%
  \end{picture}%
\endgroup%

%% file: torus_collapse2.pdf_tex
%% Creator: Inkscape 1.0.2-2 (e86c870879, 2021-01-15), www.inkscape.org
%% PDF/EPS/PS + LaTeX output extension by Johan Engelen, 2010
%% Accompanies image file 'torus_collapse2.pdf' (pdf, eps, ps)
%%
%% To include the image in your LaTeX document, write
%%   \input{<filename>.pdf_tex}
%%  instead of
%%   \includegraphics{<filename>.pdf}
%% To scale the image, write
%%   \def\svgwidth{<desired width>}
%%   \input{<filename>.pdf_tex}
%%  instead of
%%   \includegraphics[width=<desired width>]{<filename>.pdf}
%%
%% Images with a different path to the parent latex file can
%% be accessed with the `import' package (which may need to be
%% installed) using
%%   \usepackage{import}
%% in the preamble, and then including the image with
%%   \import{<path to file>}{<filename>.pdf_tex}
%% Alternatively, one can specify
%%   \graphicspath{{<path to file>/}}
%% 
%% For more information, please see info/svg-inkscape on CTAN:
%%   http://tug.ctan.org/tex-archive/info/svg-inkscape
%%
\begingroup%
  \makeatletter%
  \providecommand\color[2][]{%
    \errmessage{(Inkscape) Color is used for the text in Inkscape, but the package 'color.sty' is not loaded}%
    \renewcommand\color[2][]{}%
  }%
  \providecommand\transparent[1]{%
    \errmessage{(Inkscape) Transparency is used (non-zero) for the text in Inkscape, but the package 'transparent.sty' is not loaded}%
    \renewcommand\transparent[1]{}%
  }%
  \providecommand\rotatebox[2]{#2}%
  \newcommand*\fsize{\dimexpr\f@size pt\relax}%
  \newcommand*\lineheight[1]{\fontsize{\fsize}{#1\fsize}\selectfont}%
  \ifx\svgwidth\undefined%
    \setlength{\unitlength}{284.67154916bp}%
    \ifx\svgscale\undefined%
      \relax%
    \else%
      \setlength{\unitlength}{\unitlength * \real{\svgscale}}%
    \fi%
  \else%
    \setlength{\unitlength}{\svgwidth}%
  \fi%
  \global\let\svgwidth\undefined%
  \global\let\svgscale\undefined%
  \makeatother%
  \begin{picture}(1,0.97736901)%
    \lineheight{1}%
    \setlength\tabcolsep{0pt}%
    \put(0,0){\includegraphics[width=\unitlength,page=1]{torus_collapse2.pdf}}%
    \put(0.40885537,0.95033676){\color[rgb]{0,0,0}\makebox(0,0)[lt]{\lineheight{1.25}\smash{\begin{tabular}[t]{l}$v_1$\end{tabular}}}}%
    \put(0.89231699,0.94277295){\color[rgb]{0,0,0}\makebox(0,0)[lt]{\lineheight{1.25}\smash{\begin{tabular}[t]{l}$v_1$\end{tabular}}}}%
    \put(0.88909148,0.01673046){\color[rgb]{0,0,0}\makebox(0,0)[lt]{\lineheight{1.25}\smash{\begin{tabular}[t]{l}$v_1$\end{tabular}}}}%
    \put(0.3922331,0.006192){\color[rgb]{0,0,0}\makebox(0,0)[lt]{\lineheight{1.25}\smash{\begin{tabular}[t]{l}$v_1$\end{tabular}}}}%
    \put(0,0){\includegraphics[width=\unitlength,page=2]{torus_collapse2.pdf}}%
    \put(0.10545961,0.27905173){\color[rgb]{0,0,0}\makebox(0,0)[lt]{\lineheight{1.25}\smash{\begin{tabular}[t]{l}$3$\end{tabular}}}}%
    \put(0.87832463,0.27322581){\color[rgb]{0,0,0}\makebox(0,0)[lt]{\lineheight{1.25}\smash{\begin{tabular}[t]{l}$3$\end{tabular}}}}%
    \put(0.12644428,0.73163724){\color[rgb]{0,0,0}\makebox(0,0)[lt]{\lineheight{1.25}\smash{\begin{tabular}[t]{l}$x_4$\end{tabular}}}}%
    \put(0.89088734,0.66791399){\color[rgb]{0,0,0}\makebox(0,0)[lt]{\lineheight{1.25}\smash{\begin{tabular}[t]{l}$x_4$\end{tabular}}}}%
    \put(0.61242296,0.94906291){\color[rgb]{0,0,0}\makebox(0,0)[lt]{\lineheight{1.25}\smash{\begin{tabular}[t]{l}$x_6$\end{tabular}}}}%
    \put(0.60478738,0.01031982){\color[rgb]{0,0,0}\makebox(0,0)[lt]{\lineheight{1.25}\smash{\begin{tabular}[t]{l}$x_6$\end{tabular}}}}%
    \put(0.46742573,0.66912124){\color[rgb]{0,0,0}\makebox(0,0)[lt]{\lineheight{1.25}\smash{\begin{tabular}[t]{l}$\tilde{e}_{10}$\end{tabular}}}}%
    \put(0,0){\includegraphics[width=\unitlength,page=3]{torus_collapse2.pdf}}%
    \put(0.47092851,0.26970991){\color[rgb]{0,0,0}\makebox(0,0)[lt]{\lineheight{1.25}\smash{\begin{tabular}[t]{l}$2$\end{tabular}}}}%
    \put(0,0){\includegraphics[width=\unitlength,page=4]{torus_collapse2.pdf}}%
  \end{picture}%
\endgroup%

%% file: torus_collapse3.pdf_tex
%% Creator: Inkscape 1.0.2-2 (e86c870879, 2021-01-15), www.inkscape.org
%% PDF/EPS/PS + LaTeX output extension by Johan Engelen, 2010
%% Accompanies image file 'torus_collapse3.pdf' (pdf, eps, ps)
%%
%% To include the image in your LaTeX document, write
%%   \input{<filename>.pdf_tex}
%%  instead of
%%   \includegraphics{<filename>.pdf}
%% To scale the image, write
%%   \def\svgwidth{<desired width>}
%%   \input{<filename>.pdf_tex}
%%  instead of
%%   \includegraphics[width=<desired width>]{<filename>.pdf}
%%
%% Images with a different path to the parent latex file can
%% be accessed with the `import' package (which may need to be
%% installed) using
%%   \usepackage{import}
%% in the preamble, and then including the image with
%%   \import{<path to file>}{<filename>.pdf_tex}
%% Alternatively, one can specify
%%   \graphicspath{{<path to file>/}}
%% 
%% For more information, please see info/svg-inkscape on CTAN:
%%   http://tug.ctan.org/tex-archive/info/svg-inkscape
%%
\begingroup%
  \makeatletter%
  \providecommand\color[2][]{%
    \errmessage{(Inkscape) Color is used for the text in Inkscape, but the package 'color.sty' is not loaded}%
    \renewcommand\color[2][]{}%
  }%
  \providecommand\transparent[1]{%
    \errmessage{(Inkscape) Transparency is used (non-zero) for the text in Inkscape, but the package 'transparent.sty' is not loaded}%
    \renewcommand\transparent[1]{}%
  }%
  \providecommand\rotatebox[2]{#2}%
  \newcommand*\fsize{\dimexpr\f@size pt\relax}%
  \newcommand*\lineheight[1]{\fontsize{\fsize}{#1\fsize}\selectfont}%
  \ifx\svgwidth\undefined%
    \setlength{\unitlength}{310.0227392bp}%
    \ifx\svgscale\undefined%
      \relax%
    \else%
      \setlength{\unitlength}{\unitlength * \real{\svgscale}}%
    \fi%
  \else%
    \setlength{\unitlength}{\svgwidth}%
  \fi%
  \global\let\svgwidth\undefined%
  \global\let\svgscale\undefined%
  \makeatother%
  \begin{picture}(1,0.69230716)%
    \lineheight{1}%
    \setlength\tabcolsep{0pt}%
    \put(0,0){\includegraphics[width=\unitlength,page=1]{torus_collapse3.pdf}}%
    \put(0.41272813,0.66654812){\color[rgb]{0,0,0}\makebox(0,0)[lt]{\lineheight{1.25}\smash{\begin{tabular}[t]{l}$v_1$\end{tabular}}}}%
    \put(0.89774157,0.6564424){\color[rgb]{0,0,0}\makebox(0,0)[lt]{\lineheight{1.25}\smash{\begin{tabular}[t]{l}$v_1$\end{tabular}}}}%
    \put(0.90112243,0.22327872){\color[rgb]{0,0,0}\makebox(0,0)[lt]{\lineheight{1.25}\smash{\begin{tabular}[t]{l}$v_1$\end{tabular}}}}%
    \put(-0.00244123,0.17784591){\color[rgb]{0,0,0}\makebox(0,0)[lt]{\lineheight{1.25}\smash{\begin{tabular}[t]{l}$v_1$\end{tabular}}}}%
    \put(0,0){\includegraphics[width=\unitlength,page=2]{torus_collapse3.pdf}}%
    \put(0.19449584,0.4625717){\color[rgb]{0,0,0}\makebox(0,0)[lt]{\lineheight{1.25}\smash{\begin{tabular}[t]{l}$x_4$\end{tabular}}}}%
    \put(0.89942294,0.40554906){\color[rgb]{0,0,0}\makebox(0,0)[lt]{\lineheight{1.25}\smash{\begin{tabular}[t]{l}$x_4$\end{tabular}}}}%
    \put(0.63862811,0.66748539){\color[rgb]{0,0,0}\makebox(0,0)[lt]{\lineheight{1.25}\smash{\begin{tabular}[t]{l}$x_6$\end{tabular}}}}%
    \put(0.43766428,0.00568569){\color[rgb]{0,0,0}\makebox(0,0)[lt]{\lineheight{1.25}\smash{\begin{tabular}[t]{l}$x_6$\end{tabular}}}}%
    \put(0,0){\includegraphics[width=\unitlength,page=3]{torus_collapse3.pdf}}%
    \put(0.26222313,0.28349907){\color[rgb]{0,0,0}\makebox(0,0)[lt]{\lineheight{1.25}\smash{\begin{tabular}[t]{l}$2$\end{tabular}}}}%
    \put(0.49640725,0.35629195){\color[rgb]{0,0,0}\makebox(0,0)[lt]{\lineheight{1.25}\smash{\begin{tabular}[t]{l}$\tilde{e}_{10}$\end{tabular}}}}%
    \put(0,0){\includegraphics[width=\unitlength,page=4]{torus_collapse3.pdf}}%
  \end{picture}%
\endgroup%

%% file: torus_collapse4.pdf_tex
%% Creator: Inkscape 1.0.2-2 (e86c870879, 2021-01-15), www.inkscape.org
%% PDF/EPS/PS + LaTeX output extension by Johan Engelen, 2010
%% Accompanies image file 'torus_collapse4.pdf' (pdf, eps, ps)
%%
%% To include the image in your LaTeX document, write
%%   \input{<filename>.pdf_tex}
%%  instead of
%%   \includegraphics{<filename>.pdf}
%% To scale the image, write
%%   \def\svgwidth{<desired width>}
%%   \input{<filename>.pdf_tex}
%%  instead of
%%   \includegraphics[width=<desired width>]{<filename>.pdf}
%%
%% Images with a different path to the parent latex file can
%% be accessed with the `import' package (which may need to be
%% installed) using
%%   \usepackage{import}
%% in the preamble, and then including the image with
%%   \import{<path to file>}{<filename>.pdf_tex}
%% Alternatively, one can specify
%%   \graphicspath{{<path to file>/}}
%% 
%% For more information, please see info/svg-inkscape on CTAN:
%%   http://tug.ctan.org/tex-archive/info/svg-inkscape
%%
\begingroup%
  \makeatletter%
  \providecommand\color[2][]{%
    \errmessage{(Inkscape) Color is used for the text in Inkscape, but the package 'color.sty' is not loaded}%
    \renewcommand\color[2][]{}%
  }%
  \providecommand\transparent[1]{%
    \errmessage{(Inkscape) Transparency is used (non-zero) for the text in Inkscape, but the package 'transparent.sty' is not loaded}%
    \renewcommand\transparent[1]{}%
  }%
  \providecommand\rotatebox[2]{#2}%
  \newcommand*\fsize{\dimexpr\f@size pt\relax}%
  \newcommand*\lineheight[1]{\fontsize{\fsize}{#1\fsize}\selectfont}%
  \ifx\svgwidth\undefined%
    \setlength{\unitlength}{194.61233545bp}%
    \ifx\svgscale\undefined%
      \relax%
    \else%
      \setlength{\unitlength}{\unitlength * \real{\svgscale}}%
    \fi%
  \else%
    \setlength{\unitlength}{\svgwidth}%
  \fi%
  \global\let\svgwidth\undefined%
  \global\let\svgscale\undefined%
  \makeatother%
  \begin{picture}(1,0.78710955)%
    \lineheight{1}%
    \setlength\tabcolsep{0pt}%
    \put(0,0){\includegraphics[width=\unitlength,page=1]{torus_collapse4.pdf}}%
    \put(0.02558395,0.74203484){\color[rgb]{0,0,0}\makebox(0,0)[lt]{\lineheight{1.25}\smash{\begin{tabular}[t]{l}$v_1$\end{tabular}}}}%
    \put(0.80126869,0.74706949){\color[rgb]{0,0,0}\makebox(0,0)[lt]{\lineheight{1.25}\smash{\begin{tabular}[t]{l}$v_1$\end{tabular}}}}%
    \put(0.80727723,0.02041866){\color[rgb]{0,0,0}\makebox(0,0)[lt]{\lineheight{1.25}\smash{\begin{tabular}[t]{l}$v_1$\end{tabular}}}}%
    \put(0,0){\includegraphics[width=\unitlength,page=2]{torus_collapse4.pdf}}%
    \put(-0.00388896,0.39223299){\color[rgb]{0,0,0}\makebox(0,0)[lt]{\lineheight{1.25}\smash{\begin{tabular}[t]{l}$x_4$\end{tabular}}}}%
    \put(0.84248532,0.39301266){\color[rgb]{0,0,0}\makebox(0,0)[lt]{\lineheight{1.25}\smash{\begin{tabular}[t]{l}$x_4$\end{tabular}}}}%
    \put(0.42205809,0.74756779){\color[rgb]{0,0,0}\makebox(0,0)[lt]{\lineheight{1.25}\smash{\begin{tabular}[t]{l}$x_6$\end{tabular}}}}%
    \put(0.41505963,0.00905743){\color[rgb]{0,0,0}\makebox(0,0)[lt]{\lineheight{1.25}\smash{\begin{tabular}[t]{l}$x_6$\end{tabular}}}}%
    \put(0,0){\includegraphics[width=\unitlength,page=3]{torus_collapse4.pdf}}%
    \put(0.05061117,0.014988){\color[rgb]{0,0,0}\makebox(0,0)[lt]{\lineheight{1.25}\smash{\begin{tabular}[t]{l}$v_1$\end{tabular}}}}%
    \put(0.3911639,0.39913878){\color[rgb]{0,0,0}\makebox(0,0)[lt]{\lineheight{1.25}\smash{\begin{tabular}[t]{l}$\tilde{e}_{10}$\end{tabular}}}}%
    \put(0,0){\includegraphics[width=\unitlength,page=4]{torus_collapse4.pdf}}%
  \end{picture}%
\endgroup%

%% file: pyramids.pdf_tex
%% Creator: Inkscape 1.2.1 (9c6d41e, 2022-07-14), www.inkscape.org
%% PDF/EPS/PS + LaTeX output extension by Johan Engelen, 2010
%% Accompanies image file 'pyramids.pdf' (pdf, eps, ps)
%%
%% To include the image in your LaTeX document, write
%%   \input{<filename>.pdf_tex}
%%  instead of
%%   \includegraphics{<filename>.pdf}
%% To scale the image, write
%%   \def\svgwidth{<desired width>}
%%   \input{<filename>.pdf_tex}
%%  instead of
%%   \includegraphics[width=<desired width>]{<filename>.pdf}
%%
%% Images with a different path to the parent latex file can
%% be accessed with the `import' package (which may need to be
%% installed) using
%%   \usepackage{import}
%% in the preamble, and then including the image with
%%   \import{<path to file>}{<filename>.pdf_tex}
%% Alternatively, one can specify
%%   \graphicspath{{<path to file>/}}
%% 
%% For more information, please see info/svg-inkscape on CTAN:
%%   http://tug.ctan.org/tex-archive/info/svg-inkscape
%%
\begingroup%
  \makeatletter%
  \providecommand\color[2][]{%
    \errmessage{(Inkscape) Color is used for the text in Inkscape, but the package 'color.sty' is not loaded}%
    \renewcommand\color[2][]{}%
  }%
  \providecommand\transparent[1]{%
    \errmessage{(Inkscape) Transparency is used (non-zero) for the text in Inkscape, but the package 'transparent.sty' is not loaded}%
    \renewcommand\transparent[1]{}%
  }%
  \providecommand\rotatebox[2]{#2}%
  \newcommand*\fsize{\dimexpr\f@size pt\relax}%
  \newcommand*\lineheight[1]{\fontsize{\fsize}{#1\fsize}\selectfont}%
  \ifx\svgwidth\undefined%
    \setlength{\unitlength}{1059.99993896bp}%
    \ifx\svgscale\undefined%
      \relax%
    \else%
      \setlength{\unitlength}{\unitlength * \real{\svgscale}}%
    \fi%
  \else%
    \setlength{\unitlength}{\svgwidth}%
  \fi%
  \global\let\svgwidth\undefined%
  \global\let\svgscale\undefined%
  \makeatother%
  \begin{picture}(1,0.49056605)%
    \lineheight{1}%
    \setlength\tabcolsep{0pt}%
    \put(0,0){\includegraphics[width=\unitlength,page=1]{pyramids.pdf}}%
    \put(0.21260628,0.44772465){\makebox(0,0)[lt]{\lineheight{1.25}\smash{\begin{tabular}[t]{l}$v_1$\end{tabular}}}}%
    \put(0.13352089,0.12705439){\makebox(0,0)[lt]{\lineheight{1.25}\smash{\begin{tabular}[t]{l}$e_1$\end{tabular}}}}%
    \put(0.16778947,0.08207517){\makebox(0,0)[lt]{\lineheight{1.25}\smash{\begin{tabular}[t]{l}$e_2$\end{tabular}}}}%
    \put(0.76387996,0.1517106){\makebox(0,0)[lt]{\lineheight{1.25}\smash{\begin{tabular}[t]{l}$e_2$\end{tabular}}}}%
    \put(0.82340988,0.12128207){\makebox(0,0)[lt]{\lineheight{1.25}\smash{\begin{tabular}[t]{l}$e_1$\end{tabular}}}}%
    \put(0.66476742,0.12855743){\makebox(0,0)[lt]{\lineheight{1.25}\smash{\begin{tabular}[t]{l}$e_3$\end{tabular}}}}%
    \put(0.29222183,0.12128207){\makebox(0,0)[lt]{\lineheight{1.25}\smash{\begin{tabular}[t]{l}$e_3$\end{tabular}}}}%
    \put(0.70021682,0.08273876){\makebox(0,0)[lt]{\lineheight{1.25}\smash{\begin{tabular}[t]{l}$e_4$\end{tabular}}}}%
    \put(0.23277988,0.15237416){\makebox(0,0)[lt]{\lineheight{1.25}\smash{\begin{tabular}[t]{l}$e_4$\end{tabular}}}}%
    \put(0.42078507,0.16955969){\makebox(0,0)[lt]{\lineheight{1.25}\smash{\begin{tabular}[t]{l}$v_2$\end{tabular}}}}%
    \put(0.30264466,0.04220119){\makebox(0,0)[lt]{\lineheight{1.25}\smash{\begin{tabular}[t]{l}$v_3$\end{tabular}}}}%
    \put(0.11703543,0.17995304){\makebox(0,0)[lt]{\lineheight{1.25}\smash{\begin{tabular}[t]{l}$v_3$\end{tabular}}}}%
    \put(0.00449944,0.04644648){\makebox(0,0)[lt]{\lineheight{1.25}\smash{\begin{tabular}[t]{l}$v_2$\end{tabular}}}}%
    \put(0.34083623,0.2914505){\makebox(0,0)[lt]{\lineheight{1.25}\smash{\begin{tabular}[t]{l}$x_4$\end{tabular}}}}%
    \put(0.16781912,0.29145072){\makebox(0,0)[lt]{\lineheight{1.25}\smash{\begin{tabular}[t]{l}$x_3$\end{tabular}}}}%
    \put(0.07336737,0.2419222){\makebox(0,0)[lt]{\lineheight{1.25}\smash{\begin{tabular}[t]{l}$x_2$\end{tabular}}}}%
    \put(0.27472654,0.2419222){\makebox(0,0)[lt]{\lineheight{1.25}\smash{\begin{tabular}[t]{l}$x_1$\end{tabular}}}}%
    \put(0.21019183,0.10199385){\makebox(0,0)[lt]{\lineheight{1.25}\smash{\begin{tabular}[t]{l}$v_4$\end{tabular}}}}%
    \put(0.74567314,0.44772465){\makebox(0,0)[lt]{\lineheight{1.25}\smash{\begin{tabular}[t]{l}$v_1$\end{tabular}}}}%
    \put(0.94677651,0.16955969){\makebox(0,0)[lt]{\lineheight{1.25}\smash{\begin{tabular}[t]{l}$v_3$\end{tabular}}}}%
    \put(0.82722098,0.04220119){\makebox(0,0)[lt]{\lineheight{1.25}\smash{\begin{tabular}[t]{l}$v_2$\end{tabular}}}}%
    \put(0.64576906,0.18203172){\makebox(0,0)[lt]{\lineheight{1.25}\smash{\begin{tabular}[t]{l}$v_2$\end{tabular}}}}%
    \put(0.53190591,0.04503138){\makebox(0,0)[lt]{\lineheight{1.25}\smash{\begin{tabular}[t]{l}$v_3$\end{tabular}}}}%
    \put(0.87107293,0.2914505){\makebox(0,0)[lt]{\lineheight{1.25}\smash{\begin{tabular}[t]{l}$x_1$\end{tabular}}}}%
    \put(0.69863127,0.2914505){\makebox(0,0)[lt]{\lineheight{1.25}\smash{\begin{tabular}[t]{l}$x_2$\end{tabular}}}}%
    \put(0.59794364,0.2419222){\makebox(0,0)[lt]{\lineheight{1.25}\smash{\begin{tabular}[t]{l}$x_3$\end{tabular}}}}%
    \put(0.80063,0.2419222){\makebox(0,0)[lt]{\lineheight{1.25}\smash{\begin{tabular}[t]{l}$x_4$\end{tabular}}}}%
    \put(0.73921332,0.10124234){\makebox(0,0)[lt]{\lineheight{1.25}\smash{\begin{tabular}[t]{l}$v_4$\end{tabular}}}}%
    \put(0.03118005,0.42196995){\makebox(0,0)[lt]{\lineheight{1.25}\smash{\begin{tabular}[t]{l}$K_1$\end{tabular}}}}%
    \put(0.55812282,0.42300929){\makebox(0,0)[lt]{\lineheight{1.25}\smash{\begin{tabular}[t]{l}$K_2$\end{tabular}}}}%
  \end{picture}%
\endgroup%

%% file: pyramids_top2.pdf_tex
%% Creator: Inkscape 1.2.1 (9c6d41e, 2022-07-14), www.inkscape.org
%% PDF/EPS/PS + LaTeX output extension by Johan Engelen, 2010
%% Accompanies image file 'pyramids_top2.pdf' (pdf, eps, ps)
%%
%% To include the image in your LaTeX document, write
%%   \input{<filename>.pdf_tex}
%%  instead of
%%   \includegraphics{<filename>.pdf}
%% To scale the image, write
%%   \def\svgwidth{<desired width>}
%%   \input{<filename>.pdf_tex}
%%  instead of
%%   \includegraphics[width=<desired width>]{<filename>.pdf}
%%
%% Images with a different path to the parent latex file can
%% be accessed with the `import' package (which may need to be
%% installed) using
%%   \usepackage{import}
%% in the preamble, and then including the image with
%%   \import{<path to file>}{<filename>.pdf_tex}
%% Alternatively, one can specify
%%   \graphicspath{{<path to file>/}}
%% 
%% For more information, please see info/svg-inkscape on CTAN:
%%   http://tug.ctan.org/tex-archive/info/svg-inkscape
%%
\begingroup%
  \makeatletter%
  \providecommand\color[2][]{%
    \errmessage{(Inkscape) Color is used for the text in Inkscape, but the package 'color.sty' is not loaded}%
    \renewcommand\color[2][]{}%
  }%
  \providecommand\transparent[1]{%
    \errmessage{(Inkscape) Transparency is used (non-zero) for the text in Inkscape, but the package 'transparent.sty' is not loaded}%
    \renewcommand\transparent[1]{}%
  }%
  \providecommand\rotatebox[2]{#2}%
  \newcommand*\fsize{\dimexpr\f@size pt\relax}%
  \newcommand*\lineheight[1]{\fontsize{\fsize}{#1\fsize}\selectfont}%
  \ifx\svgwidth\undefined%
    \setlength{\unitlength}{992.00006104bp}%
    \ifx\svgscale\undefined%
      \relax%
    \else%
      \setlength{\unitlength}{\unitlength * \real{\svgscale}}%
    \fi%
  \else%
    \setlength{\unitlength}{\svgwidth}%
  \fi%
  \global\let\svgwidth\undefined%
  \global\let\svgscale\undefined%
  \makeatother%
  \begin{picture}(1,0.44354838)%
    \lineheight{1}%
    \setlength\tabcolsep{0pt}%
    \put(0,0){\includegraphics[width=\unitlength,page=1]{pyramids_top2.pdf}}%
    \put(0.19131719,0.25003542){\makebox(0,0)[lt]{\lineheight{1.25}\smash{\begin{tabular}[t]{l}$v_1$\end{tabular}}}}%
    \put(0.0493168,0.36681394){\makebox(0,0)[lt]{\lineheight{1.25}\smash{\begin{tabular}[t]{l}$v_3$\end{tabular}}}}%
    \put(0.32960143,0.36681394){\makebox(0,0)[lt]{\lineheight{1.25}\smash{\begin{tabular}[t]{l}$v_2$\end{tabular}}}}%
    \put(0.32960143,0.070443){\makebox(0,0)[lt]{\lineheight{1.25}\smash{\begin{tabular}[t]{l}$v_3$\end{tabular}}}}%
    \put(0.05117915,0.070443){\makebox(0,0)[lt]{\lineheight{1.25}\smash{\begin{tabular}[t]{l}$v_2$\end{tabular}}}}%
    \put(0.03300836,0.22618895){\makebox(0,0)[lt]{\lineheight{1.25}\smash{\begin{tabular}[t]{l}$x_{6}$\end{tabular}}}}%
    \put(0.60548831,0.22618895){\makebox(0,0)[lt]{\lineheight{1.25}\smash{\begin{tabular}[t]{l}$x_{6}$\end{tabular}}}}%
    \put(0.1768882,0.36681394){\makebox(0,0)[lt]{\lineheight{1.25}\smash{\begin{tabular}[t]{l}$x_5$\end{tabular}}}}%
    \put(0.3515909,0.22618895){\makebox(0,0)[lt]{\lineheight{1.25}\smash{\begin{tabular}[t]{l}$x_{6}$\end{tabular}}}}%
    \put(0.18200539,0.070443){\makebox(0,0)[lt]{\lineheight{1.25}\smash{\begin{tabular}[t]{l}$x_5$\end{tabular}}}}%
    \put(0.18863471,0.15004908){\makebox(0,0)[lt]{\lineheight{1.25}\smash{\begin{tabular}[t]{l}$e_5$\end{tabular}}}}%
    \put(0.26475213,0.22165266){\makebox(0,0)[lt]{\lineheight{1.25}\smash{\begin{tabular}[t]{l}$e_{6}$\end{tabular}}}}%
    \put(0.18607185,0.30462468){\makebox(0,0)[lt]{\lineheight{1.25}\smash{\begin{tabular}[t]{l}$e_{7}$\end{tabular}}}}%
    \put(0.12135067,0.22165266){\makebox(0,0)[lt]{\lineheight{1.25}\smash{\begin{tabular}[t]{l}$e_{8}$\end{tabular}}}}%
    \put(0.76671363,0.24815566){\makebox(0,0)[lt]{\lineheight{1.25}\smash{\begin{tabular}[t]{l}$v_1$\end{tabular}}}}%
    \put(0.62655877,0.36681394){\makebox(0,0)[lt]{\lineheight{1.25}\smash{\begin{tabular}[t]{l}$v_2$\end{tabular}}}}%
    \put(0.90684344,0.36681394){\makebox(0,0)[lt]{\lineheight{1.25}\smash{\begin{tabular}[t]{l}$v_3$\end{tabular}}}}%
    \put(0.90684344,0.070443){\makebox(0,0)[lt]{\lineheight{1.25}\smash{\begin{tabular}[t]{l}$v_2$\end{tabular}}}}%
    \put(0.62842116,0.070443){\makebox(0,0)[lt]{\lineheight{1.25}\smash{\begin{tabular}[t]{l}$v_3$\end{tabular}}}}%
    \put(0.75461686,0.36681394){\makebox(0,0)[lt]{\lineheight{1.25}\smash{\begin{tabular}[t]{l}$x_5$\end{tabular}}}}%
    \put(0.9227845,0.22618895){\makebox(0,0)[lt]{\lineheight{1.25}\smash{\begin{tabular}[t]{l}$x_{6}$\end{tabular}}}}%
    \put(0.75461686,0.070443){\makebox(0,0)[lt]{\lineheight{1.25}\smash{\begin{tabular}[t]{l}$x_5$\end{tabular}}}}%
    \put(0.76238809,0.14665723){\makebox(0,0)[lt]{\lineheight{1.25}\smash{\begin{tabular}[t]{l}$e_{7}$\end{tabular}}}}%
    \put(0.83896992,0.22316476){\makebox(0,0)[lt]{\lineheight{1.25}\smash{\begin{tabular}[t]{l}$e_{6}$\end{tabular}}}}%
    \put(0.76238809,0.29972071){\makebox(0,0)[lt]{\lineheight{1.25}\smash{\begin{tabular}[t]{l}$e_{5}$\end{tabular}}}}%
    \put(0.69405639,0.22316476){\makebox(0,0)[lt]{\lineheight{1.25}\smash{\begin{tabular}[t]{l}$e_{8}$\end{tabular}}}}%
    \put(0.04001356,0.42009077){\makebox(0,0)[lt]{\lineheight{1.25}\smash{\begin{tabular}[t]{l}$K_1 \text{ (top view)}$\end{tabular}}}}%
    \put(0.57129235,0.41903112){\makebox(0,0)[lt]{\lineheight{1.25}\smash{\begin{tabular}[t]{l}$K_2\text{ (top view)}$\end{tabular}}}}%
    \put(0.2898579,0.14728813){\makebox(0,0)[lt]{\lineheight{1.25}\smash{\begin{tabular}[t]{l}$x_1$\end{tabular}}}}%
    \put(0.86294252,0.28748742){\makebox(0,0)[lt]{\lineheight{1.25}\smash{\begin{tabular}[t]{l}$x_1$\end{tabular}}}}%
    \put(0.66939419,0.28748742){\makebox(0,0)[lt]{\lineheight{1.25}\smash{\begin{tabular}[t]{l}$x_2$\end{tabular}}}}%
    \put(0.2898579,0.28748742){\makebox(0,0)[lt]{\lineheight{1.25}\smash{\begin{tabular}[t]{l}$x_4$\end{tabular}}}}%
    \put(0.09683223,0.28748742){\makebox(0,0)[lt]{\lineheight{1.25}\smash{\begin{tabular}[t]{l}$x_3$\end{tabular}}}}%
    \put(0.09683223,0.15088471){\makebox(0,0)[lt]{\lineheight{1.25}\smash{\begin{tabular}[t]{l}$x_2$\end{tabular}}}}%
    \put(0.66965205,0.15010732){\makebox(0,0)[lt]{\lineheight{1.25}\smash{\begin{tabular}[t]{l}$x_3$\end{tabular}}}}%
    \put(0.86446477,0.14745047){\makebox(0,0)[lt]{\lineheight{1.25}\smash{\begin{tabular}[t]{l}$x_4$\end{tabular}}}}%
    \put(0.12122403,0.42475395){\makebox(0,0)[lt]{\lineheight{1.25}\smash{\begin{tabular}[t]{l} \end{tabular}}}}%
    \put(0.67096088,0.42381423){\makebox(0,0)[lt]{\lineheight{1.25}\smash{\begin{tabular}[t]{l}  \end{tabular}}}}%
    \put(0.12404318,0.42287451){\makebox(0,0)[lt]{\lineheight{1.25}\smash{\begin{tabular}[t]{l} \end{tabular}}}}%
  \end{picture}%
\endgroup%

%% file: face_poset_dim3.pdf_tex
%% Creator: Inkscape 1.2.1 (9c6d41e, 2022-07-14), www.inkscape.org
%% PDF/EPS/PS + LaTeX output extension by Johan Engelen, 2010
%% Accompanies image file 'face_poset_dim3.pdf' (pdf, eps, ps)
%%
%% To include the image in your LaTeX document, write
%%   \input{<filename>.pdf_tex}
%%  instead of
%%   \includegraphics{<filename>.pdf}
%% To scale the image, write
%%   \def\svgwidth{<desired width>}
%%   \input{<filename>.pdf_tex}
%%  instead of
%%   \includegraphics[width=<desired width>]{<filename>.pdf}
%%
%% Images with a different path to the parent latex file can
%% be accessed with the `import' package (which may need to be
%% installed) using
%%   \usepackage{import}
%% in the preamble, and then including the image with
%%   \import{<path to file>}{<filename>.pdf_tex}
%% Alternatively, one can specify
%%   \graphicspath{{<path to file>/}}
%% 
%% For more information, please see info/svg-inkscape on CTAN:
%%   http://tug.ctan.org/tex-archive/info/svg-inkscape
%%
\begingroup%
  \makeatletter%
  \providecommand\color[2][]{%
    \errmessage{(Inkscape) Color is used for the text in Inkscape, but the package 'color.sty' is not loaded}%
    \renewcommand\color[2][]{}%
  }%
  \providecommand\transparent[1]{%
    \errmessage{(Inkscape) Transparency is used (non-zero) for the text in Inkscape, but the package 'transparent.sty' is not loaded}%
    \renewcommand\transparent[1]{}%
  }%
  \providecommand\rotatebox[2]{#2}%
  \newcommand*\fsize{\dimexpr\f@size pt\relax}%
  \newcommand*\lineheight[1]{\fontsize{\fsize}{#1\fsize}\selectfont}%
  \ifx\svgwidth\undefined%
    \setlength{\unitlength}{1875.99993896bp}%
    \ifx\svgscale\undefined%
      \relax%
    \else%
      \setlength{\unitlength}{\unitlength * \real{\svgscale}}%
    \fi%
  \else%
    \setlength{\unitlength}{\svgwidth}%
  \fi%
  \global\let\svgwidth\undefined%
  \global\let\svgscale\undefined%
  \makeatother%
  \begin{picture}(1,0.48614074)%
    \lineheight{1}%
    \setlength\tabcolsep{0pt}%
    \put(0,0){\includegraphics[width=\unitlength,page=1]{face_poset_dim3.pdf}}%
    \put(0.0162293,0.4600941){\makebox(0,0)[lt]{\lineheight{1.25}\smash{\begin{tabular}[t]{l}$\mathcal X(K_1)$\end{tabular}}}}%
    \put(0.91640005,0.4600941){\makebox(0,0)[lt]{\lineheight{1.25}\smash{\begin{tabular}[t]{l}$\mathcal X(K_2)$\end{tabular}}}}%
    \put(0.51725076,0.33059046){\makebox(0,0)[lt]{\lineheight{1.25}\smash{\begin{tabular}[t]{l}$\mathcal X(P)$\end{tabular}}}}%
  \end{picture}%
\endgroup%

%% file: internal_dim3.pdf_tex
%% Creator: Inkscape 1.2.1 (9c6d41e, 2022-07-14), www.inkscape.org
%% PDF/EPS/PS + LaTeX output extension by Johan Engelen, 2010
%% Accompanies image file 'internal_dim3.pdf' (pdf, eps, ps)
%%
%% To include the image in your LaTeX document, write
%%   \input{<filename>.pdf_tex}
%%  instead of
%%   \includegraphics{<filename>.pdf}
%% To scale the image, write
%%   \def\svgwidth{<desired width>}
%%   \input{<filename>.pdf_tex}
%%  instead of
%%   \includegraphics[width=<desired width>]{<filename>.pdf}
%%
%% Images with a different path to the parent latex file can
%% be accessed with the `import' package (which may need to be
%% installed) using
%%   \usepackage{import}
%% in the preamble, and then including the image with
%%   \import{<path to file>}{<filename>.pdf_tex}
%% Alternatively, one can specify
%%   \graphicspath{{<path to file>/}}
%% 
%% For more information, please see info/svg-inkscape on CTAN:
%%   http://tug.ctan.org/tex-archive/info/svg-inkscape
%%
\begingroup%
  \makeatletter%
  \providecommand\color[2][]{%
    \errmessage{(Inkscape) Color is used for the text in Inkscape, but the package 'color.sty' is not loaded}%
    \renewcommand\color[2][]{}%
  }%
  \providecommand\transparent[1]{%
    \errmessage{(Inkscape) Transparency is used (non-zero) for the text in Inkscape, but the package 'transparent.sty' is not loaded}%
    \renewcommand\transparent[1]{}%
  }%
  \providecommand\rotatebox[2]{#2}%
  \newcommand*\fsize{\dimexpr\f@size pt\relax}%
  \newcommand*\lineheight[1]{\fontsize{\fsize}{#1\fsize}\selectfont}%
  \ifx\svgwidth\undefined%
    \setlength{\unitlength}{1896bp}%
    \ifx\svgscale\undefined%
      \relax%
    \else%
      \setlength{\unitlength}{\unitlength * \real{\svgscale}}%
    \fi%
  \else%
    \setlength{\unitlength}{\svgwidth}%
  \fi%
  \global\let\svgwidth\undefined%
  \global\let\svgscale\undefined%
  \makeatother%
  \begin{picture}(1,0.70042197)%
    \lineheight{1}%
    \setlength\tabcolsep{0pt}%
    \put(0,0){\includegraphics[width=\unitlength,page=1]{internal_dim3.pdf}}%
    \put(0.76923873,0.33674518){\makebox(0,0)[lt]{\lineheight{1.25}\smash{\begin{tabular}[t]{l}$x_{12}$\end{tabular}}}}%
    \put(0.85982301,0.55006025){\makebox(0,0)[lt]{\lineheight{1.25}\smash{\begin{tabular}[t]{l}$e_{12}$\end{tabular}}}}%
    \put(0.80475935,0.29245381){\makebox(0,0)[lt]{\lineheight{1.25}\smash{\begin{tabular}[t]{l}$e_{16}$\end{tabular}}}}%
    \put(0.11910308,0.56566419){\makebox(0,0)[lt]{\lineheight{1.25}\smash{\begin{tabular}[t]{l}$b_1$\end{tabular}}}}%
    \put(0.12636342,0.35239302){\makebox(0,0)[lt]{\lineheight{1.25}\smash{\begin{tabular}[t]{l}$b_5$\end{tabular}}}}%
    \put(0.39219905,0.36025644){\makebox(0,0)[lt]{\lineheight{1.25}\smash{\begin{tabular}[t]{l}$b_6$\end{tabular}}}}%
    \put(0.61738688,0.29465726){\makebox(0,0)[lt]{\lineheight{1.25}\smash{\begin{tabular}[t]{l}$b_7$\end{tabular}}}}%
    \put(0.26982165,0.55243408){\makebox(0,0)[lt]{\lineheight{1.25}\smash{\begin{tabular}[t]{l}$b_2$\end{tabular}}}}%
    \put(0.0058834,0.65823402){\makebox(0,0)[lt]{\lineheight{1.25}\smash{\begin{tabular}[t]{l}$K_1$\end{tabular}}}}%
    \put(0.00542465,0.41518943){\makebox(0,0)[lt]{\lineheight{1.25}\smash{\begin{tabular}[t]{l}$K_2$\end{tabular}}}}%
    \put(0.0949323,0.55041336){\makebox(0,0)[lt]{\lineheight{1.25}\smash{\begin{tabular}[t]{l}$e_9$\end{tabular}}}}%
    \put(0.3218398,0.54632787){\makebox(0,0)[lt]{\lineheight{1.25}\smash{\begin{tabular}[t]{l}$e_{10}$\end{tabular}}}}%
    \put(0.47245967,0.55854028){\makebox(0,0)[lt]{\lineheight{1.25}\smash{\begin{tabular}[t]{l}$e_{11}$\end{tabular}}}}%
    \put(0.43980799,0.53059609){\makebox(0,0)[lt]{\lineheight{1.25}\smash{\begin{tabular}[t]{l}$b_3$\end{tabular}}}}%
    \put(0.67498229,0.54836328){\makebox(0,0)[lt]{\lineheight{1.25}\smash{\begin{tabular}[t]{l}$b_4$\end{tabular}}}}%
    \put(0.90318067,0.54944367){\makebox(0,0)[lt]{\lineheight{1.25}\smash{\begin{tabular}[t]{l}$x_{11}$\end{tabular}}}}%
    \put(0.15009634,0.32800988){\makebox(0,0)[lt]{\lineheight{1.25}\smash{\begin{tabular}[t]{l}$e_{13}$\end{tabular}}}}%
    \put(0.34401314,0.33034354){\makebox(0,0)[lt]{\lineheight{1.25}\smash{\begin{tabular}[t]{l}$e_{14}$\end{tabular}}}}%
    \put(0.61014253,0.36286494){\makebox(0,0)[lt]{\lineheight{1.25}\smash{\begin{tabular}[t]{l}$e_{15}$\end{tabular}}}}%
    \put(0.19122072,0.11573622){\makebox(0,0)[lt]{\lineheight{1.25}\smash{\begin{tabular}[t]{l}$e_6$\end{tabular}}}}%
    \put(0.35898593,0.15705693){\makebox(0,0)[lt]{\lineheight{1.25}\smash{\begin{tabular}[t]{l}$e_{7}$\end{tabular}}}}%
    \put(0.53588529,0.10767828){\makebox(0,0)[lt]{\lineheight{1.25}\smash{\begin{tabular}[t]{l}$e_{8}$\end{tabular}}}}%
    \put(0.67277628,0.49597327){\makebox(0,0)[lt]{\lineheight{1.25}\smash{\begin{tabular}[t]{l}$e_4$\end{tabular}}}}%
    \put(0.00974239,0.19426491){\makebox(0,0)[lt]{\lineheight{1.25}\smash{\begin{tabular}[t]{l}$P$\end{tabular}}}}%
    \put(0.75209746,0.08529349){\makebox(0,0)[lt]{\lineheight{1.25}\smash{\begin{tabular}[t]{l}$x_6$\end{tabular}}}}%
    \put(0.79867626,0.12278096){\makebox(0,0)[lt]{\lineheight{1.25}\smash{\begin{tabular}[t]{l}$v_4$\end{tabular}}}}%
    \put(0.16545209,0.08075222){\makebox(0,0)[lt]{\lineheight{1.25}\smash{\begin{tabular}[t]{l}$x_5$\end{tabular}}}}%
    \put(0.40431353,0.12498899){\makebox(0,0)[lt]{\lineheight{1.25}\smash{\begin{tabular}[t]{l}$x_8$\end{tabular}}}}%
    \put(0.56346217,0.14538434){\makebox(0,0)[lt]{\lineheight{1.25}\smash{\begin{tabular}[t]{l}$x_7$\end{tabular}}}}%
  \end{picture}%
\endgroup%

%% file: morse_dim3.pdf_tex
%% Creator: Inkscape 1.2.1 (9c6d41e, 2022-07-14), www.inkscape.org
%% PDF/EPS/PS + LaTeX output extension by Johan Engelen, 2010
%% Accompanies image file 'morse_dim3.pdf' (pdf, eps, ps)
%%
%% To include the image in your LaTeX document, write
%%   \input{<filename>.pdf_tex}
%%  instead of
%%   \includegraphics{<filename>.pdf}
%% To scale the image, write
%%   \def\svgwidth{<desired width>}
%%   \input{<filename>.pdf_tex}
%%  instead of
%%   \includegraphics[width=<desired width>]{<filename>.pdf}
%%
%% Images with a different path to the parent latex file can
%% be accessed with the `import' package (which may need to be
%% installed) using
%%   \usepackage{import}
%% in the preamble, and then including the image with
%%   \import{<path to file>}{<filename>.pdf_tex}
%% Alternatively, one can specify
%%   \graphicspath{{<path to file>/}}
%% 
%% For more information, please see info/svg-inkscape on CTAN:
%%   http://tug.ctan.org/tex-archive/info/svg-inkscape
%%
\begingroup%
  \makeatletter%
  \providecommand\color[2][]{%
    \errmessage{(Inkscape) Color is used for the text in Inkscape, but the package 'color.sty' is not loaded}%
    \renewcommand\color[2][]{}%
  }%
  \providecommand\transparent[1]{%
    \errmessage{(Inkscape) Transparency is used (non-zero) for the text in Inkscape, but the package 'transparent.sty' is not loaded}%
    \renewcommand\transparent[1]{}%
  }%
  \providecommand\rotatebox[2]{#2}%
  \newcommand*\fsize{\dimexpr\f@size pt\relax}%
  \newcommand*\lineheight[1]{\fontsize{\fsize}{#1\fsize}\selectfont}%
  \ifx\svgwidth\undefined%
    \setlength{\unitlength}{1224bp}%
    \ifx\svgscale\undefined%
      \relax%
    \else%
      \setlength{\unitlength}{\unitlength * \real{\svgscale}}%
    \fi%
  \else%
    \setlength{\unitlength}{\svgwidth}%
  \fi%
  \global\let\svgwidth\undefined%
  \global\let\svgscale\undefined%
  \makeatother%
  \begin{picture}(1,0.66666667)%
    \lineheight{1}%
    \setlength\tabcolsep{0pt}%
    \put(0,0){\includegraphics[width=\unitlength,page=1]{morse_dim3.pdf}}%
    \put(0.74090666,0.61754172){\makebox(0,0)[lt]{\lineheight{1.25}\smash{\begin{tabular}[t]{l}$v_1$\end{tabular}}}}%
    \put(0.75059981,0.04287679){\makebox(0,0)[lt]{\lineheight{1.25}\smash{\begin{tabular}[t]{l}$v_1$\end{tabular}}}}%
    \put(0.91629028,0.37429346){\makebox(0,0)[lt]{\lineheight{1.25}\smash{\begin{tabular}[t]{l}$v_2$\end{tabular}}}}%
    \put(0.81275364,0.26645032){\makebox(0,0)[lt]{\lineheight{1.25}\smash{\begin{tabular}[t]{l}$v_3$\end{tabular}}}}%
    \put(0.64317898,0.37919542){\makebox(0,0)[lt]{\lineheight{1.25}\smash{\begin{tabular}[t]{l}$v_3$\end{tabular}}}}%
    \put(0.54842852,0.26399934){\makebox(0,0)[lt]{\lineheight{1.25}\smash{\begin{tabular}[t]{l}$v_2$\end{tabular}}}}%
    \put(0.84753126,0.47846012){\makebox(0,0)[lt]{\lineheight{1.25}\smash{\begin{tabular}[t]{l}$x_4$\end{tabular}}}}%
    \put(0.72013563,0.47846012){\makebox(0,0)[lt]{\lineheight{1.25}\smash{\begin{tabular}[t]{l}$x_3$\end{tabular}}}}%
    \put(0.61174554,0.43556797){\makebox(0,0)[lt]{\lineheight{1.25}\smash{\begin{tabular}[t]{l}$x_2$\end{tabular}}}}%
    \put(0.78980165,0.43556797){\makebox(0,0)[lt]{\lineheight{1.25}\smash{\begin{tabular}[t]{l}$x_1$\end{tabular}}}}%
    \put(0.74173752,0.18965867){\makebox(0,0)[lt]{\lineheight{1.25}\smash{\begin{tabular}[t]{l}$x_3$\end{tabular}}}}%
    \put(0.66775662,0.22723463){\makebox(0,0)[lt]{\lineheight{1.25}\smash{\begin{tabular}[t]{l}$x_1$\end{tabular}}}}%
    \put(0.59689421,0.18827393){\makebox(0,0)[lt]{\lineheight{1.25}\smash{\begin{tabular}[t]{l}$x_4$\end{tabular}}}}%
    \put(0.85374039,0.22723463){\makebox(0,0)[lt]{\lineheight{1.25}\smash{\begin{tabular}[t]{l}$x_2$\end{tabular}}}}%
    \put(0.06791421,0.61371529){\makebox(0,0)[lt]{\lineheight{1.25}\smash{\begin{tabular}[t]{l}$S$\end{tabular}}}}%
    \put(0.56728341,0.61371529){\makebox(0,0)[lt]{\lineheight{1.25}\smash{\begin{tabular}[t]{l}$K_M$\end{tabular}}}}%
  \end{picture}%
\endgroup%

%% file: CW_presentation_cells.pdf_tex
%% Creator: Inkscape 1.1 (c68e22c387, 2021-05-23), www.inkscape.org
%% PDF/EPS/PS + LaTeX output extension by Johan Engelen, 2010
%% Accompanies image file 'CW_presentation_cells.pdf' (pdf, eps, ps)
%%
%% To include the image in your LaTeX document, write
%%   \input{<filename>.pdf_tex}
%%  instead of
%%   \includegraphics{<filename>.pdf}
%% To scale the image, write
%%   \def\svgwidth{<desired width>}
%%   \input{<filename>.pdf_tex}
%%  instead of
%%   \includegraphics[width=<desired width>]{<filename>.pdf}
%%
%% Images with a different path to the parent latex file can
%% be accessed with the `import' package (which may need to be
%% installed) using
%%   \usepackage{import}
%% in the preamble, and then including the image with
%%   \import{<path to file>}{<filename>.pdf_tex}
%% Alternatively, one can specify
%%   \graphicspath{{<path to file>/}}
%% 
%% For more information, please see info/svg-inkscape on CTAN:
%%   http://tug.ctan.org/tex-archive/info/svg-inkscape
%%
\begingroup%
  \makeatletter%
  \providecommand\color[2][]{%
    \errmessage{(Inkscape) Color is used for the text in Inkscape, but the package 'color.sty' is not loaded}%
    \renewcommand\color[2][]{}%
  }%
  \providecommand\transparent[1]{%
    \errmessage{(Inkscape) Transparency is used (non-zero) for the text in Inkscape, but the package 'transparent.sty' is not loaded}%
    \renewcommand\transparent[1]{}%
  }%
  \providecommand\rotatebox[2]{#2}%
  \newcommand*\fsize{\dimexpr\f@size pt\relax}%
  \newcommand*\lineheight[1]{\fontsize{\fsize}{#1\fsize}\selectfont}%
  \ifx\svgwidth\undefined%
    \setlength{\unitlength}{427.03338383bp}%
    \ifx\svgscale\undefined%
      \relax%
    \else%
      \setlength{\unitlength}{\unitlength * \real{\svgscale}}%
    \fi%
  \else%
    \setlength{\unitlength}{\svgwidth}%
  \fi%
  \global\let\svgwidth\undefined%
  \global\let\svgscale\undefined%
  \makeatother%
  \begin{picture}(1,0.50058494)%
    \lineheight{1}%
    \setlength\tabcolsep{0pt}%
    \put(0,0){\includegraphics[width=\unitlength,page=1]{CW_presentation_cells.pdf}}%
    \put(0.44181937,0.22067595){\color[rgb]{0,0,0}\makebox(0,0)[lt]{\lineheight{1.25}\smash{\begin{tabular}[t]{l}$o$\end{tabular}}}}%
    \put(-0.00177231,0.22196823){\color[rgb]{0,0,0}\makebox(0,0)[lt]{\lineheight{1.25}\smash{\begin{tabular}[t]{l}$o$\end{tabular}}}}%
    \put(0.82532886,0.47034695){\color[rgb]{0,0,0}\makebox(0,0)[lt]{\lineheight{1.25}\smash{\begin{tabular}[t]{l}$o$\end{tabular}}}}%
    \put(0.82160319,0.03071661){\color[rgb]{0,0,0}\makebox(0,0)[lt]{\lineheight{1.25}\smash{\begin{tabular}[t]{l}$o$\end{tabular}}}}%
    \put(0.2213424,0.00960439){\color[rgb]{0,0,0}\makebox(0,0)[lt]{\lineheight{1.25}\smash{\begin{tabular}[t]{l}$x$\end{tabular}}}}%
    \put(0.54357292,0.47759696){\color[rgb]{0,0,0}\makebox(0,0)[lt]{\lineheight{1.25}\smash{\begin{tabular}[t]{l}$x$\end{tabular}}}}%
    \put(0.22068056,0.48256453){\color[rgb]{0,0,0}\makebox(0,0)[lt]{\lineheight{1.25}\smash{\begin{tabular}[t]{l}$x$\end{tabular}}}}%
    \put(0.94095128,0.23661287){\color[rgb]{0,0,0}\makebox(0,0)[lt]{\lineheight{1.25}\smash{\begin{tabular}[t]{l}$y$\end{tabular}}}}%
    \put(0.54902823,0.0281677){\color[rgb]{0,0,0}\makebox(0,0)[lt]{\lineheight{1.25}\smash{\begin{tabular}[t]{l}$y$\end{tabular}}}}%
    \put(0.24718646,0.22425062){\color[rgb]{0,0,0}\makebox(0,0)[lt]{\lineheight{1.25}\smash{\begin{tabular}[t]{l}$v_{r_1}$\end{tabular}}}}%
    \put(0.71587031,0.22797632){\color[rgb]{0,0,0}\makebox(0,0)[lt]{\lineheight{1.25}\smash{\begin{tabular}[t]{l}$v_{r_2}$\end{tabular}}}}%
    \put(0.34308912,0.38321305){\color[rgb]{0,0,0}\makebox(0,0)[lt]{\lineheight{1.25}\smash{\begin{tabular}[t]{l}$x^{+1}$\end{tabular}}}}%
    \put(0.45528432,0.3835768){\color[rgb]{0,0,0}\makebox(0,0)[lt]{\lineheight{1.25}\smash{\begin{tabular}[t]{l}$x^{+1}$\end{tabular}}}}%
    \put(0.08475024,0.38336369){\color[rgb]{0,0,0}\makebox(0,0)[lt]{\lineheight{1.25}\smash{\begin{tabular}[t]{l}$x^{-1}$\end{tabular}}}}%
    \put(0.65979846,0.48008075){\color[rgb]{0,0,0}\makebox(0,0)[lt]{\lineheight{1.25}\smash{\begin{tabular}[t]{l}$x^{-1}$\end{tabular}}}}%
    \put(0.33236459,0.10399998){\color[rgb]{0,0,0}\makebox(0,0)[lt]{\lineheight{1.25}\smash{\begin{tabular}[t]{l}$x^{-1}$\end{tabular}}}}%
    \put(0.0774233,0.10337909){\color[rgb]{0,0,0}\makebox(0,0)[lt]{\lineheight{1.25}\smash{\begin{tabular}[t]{l}$x^{+1}$\end{tabular}}}}%
    \put(0.87502054,0.38686559){\color[rgb]{0,0,0}\makebox(0,0)[lt]{\lineheight{1.25}\smash{\begin{tabular}[t]{l}$y^{-1}$\end{tabular}}}}%
    \put(0.66622096,0.00471092){\color[rgb]{0,0,0}\makebox(0,0)[lt]{\lineheight{1.25}\smash{\begin{tabular}[t]{l}$y^{-1}$\end{tabular}}}}%
    \put(0.87504627,0.10731027){\color[rgb]{0,0,0}\makebox(0,0)[lt]{\lineheight{1.25}\smash{\begin{tabular}[t]{l}$y^{+1}$\end{tabular}}}}%
    \put(0.45482477,0.10512777){\color[rgb]{0,0,0}\makebox(0,0)[lt]{\lineheight{1.25}\smash{\begin{tabular}[t]{l}$y^{+1}$\end{tabular}}}}%
  \end{picture}%
\endgroup%

%% file: CW_presentation_labeled.pdf_tex
%% Creator: Inkscape 1.1 (c68e22c387, 2021-05-23), www.inkscape.org
%% PDF/EPS/PS + LaTeX output extension by Johan Engelen, 2010
%% Accompanies image file 'CW_presentation_labeled.pdf' (pdf, eps, ps)
%%
%% To include the image in your LaTeX document, write
%%   \input{<filename>.pdf_tex}
%%  instead of
%%   \includegraphics{<filename>.pdf}
%% To scale the image, write
%%   \def\svgwidth{<desired width>}
%%   \input{<filename>.pdf_tex}
%%  instead of
%%   \includegraphics[width=<desired width>]{<filename>.pdf}
%%
%% Images with a different path to the parent latex file can
%% be accessed with the `import' package (which may need to be
%% installed) using
%%   \usepackage{import}
%% in the preamble, and then including the image with
%%   \import{<path to file>}{<filename>.pdf_tex}
%% Alternatively, one can specify
%%   \graphicspath{{<path to file>/}}
%% 
%% For more information, please see info/svg-inkscape on CTAN:
%%   http://tug.ctan.org/tex-archive/info/svg-inkscape
%%
\begingroup%
  \makeatletter%
  \providecommand\color[2][]{%
    \errmessage{(Inkscape) Color is used for the text in Inkscape, but the package 'color.sty' is not loaded}%
    \renewcommand\color[2][]{}%
  }%
  \providecommand\transparent[1]{%
    \errmessage{(Inkscape) Transparency is used (non-zero) for the text in Inkscape, but the package 'transparent.sty' is not loaded}%
    \renewcommand\transparent[1]{}%
  }%
  \providecommand\rotatebox[2]{#2}%
  \newcommand*\fsize{\dimexpr\f@size pt\relax}%
  \newcommand*\lineheight[1]{\fontsize{\fsize}{#1\fsize}\selectfont}%
  \ifx\svgwidth\undefined%
    \setlength{\unitlength}{398.72840701bp}%
    \ifx\svgscale\undefined%
      \relax%
    \else%
      \setlength{\unitlength}{\unitlength * \real{\svgscale}}%
    \fi%
  \else%
    \setlength{\unitlength}{\svgwidth}%
  \fi%
  \global\let\svgwidth\undefined%
  \global\let\svgscale\undefined%
  \makeatother%
  \begin{picture}(1,0.51230743)%
    \lineheight{1}%
    \setlength\tabcolsep{0pt}%
    \put(0,0){\includegraphics[width=\unitlength,page=1]{CW_presentation_labeled.pdf}}%
    \put(0.06067223,0.3556623){\color[rgb]{0,0,0}\makebox(0,0)[lt]{\lineheight{1.25}\smash{\begin{tabular}[t]{l}$3$\end{tabular}}}}%
    \put(0.36275134,0.1194798){\color[rgb]{0,0,0}\makebox(0,0)[lt]{\lineheight{1.25}\smash{\begin{tabular}[t]{l}$3$\end{tabular}}}}%
    \put(0.70469761,0.49300778){\color[rgb]{0,0,0}\makebox(0,0)[lt]{\lineheight{1.25}\smash{\begin{tabular}[t]{l}$3$\end{tabular}}}}%
    \put(0.13798583,0.20838983){\color[rgb]{0,0,0}\makebox(0,0)[lt]{\lineheight{1.25}\smash{\begin{tabular}[t]{l}$14$\end{tabular}}}}%
    \put(0.24107897,0.18044525){\color[rgb]{0,0,0}\makebox(0,0)[lt]{\lineheight{1.25}\smash{\begin{tabular}[t]{l}$13$\end{tabular}}}}%
    \put(0.17077614,0.34575186){\color[rgb]{0,0,0}\makebox(0,0)[lt]{\lineheight{1.25}\smash{\begin{tabular}[t]{l}$15$\end{tabular}}}}%
    \put(0.55442857,0.36275704){\color[rgb]{0,0,0}\makebox(0,0)[lt]{\lineheight{1.25}\smash{\begin{tabular}[t]{l}$12$\end{tabular}}}}%
    \put(0.64920179,0.38074028){\color[rgb]{0,0,0}\makebox(0,0)[lt]{\lineheight{1.25}\smash{\begin{tabular}[t]{l}$11$\end{tabular}}}}%
    \put(0.80356907,0.35888939){\color[rgb]{0,0,0}\makebox(0,0)[lt]{\lineheight{1.25}\smash{\begin{tabular}[t]{l}$10$\end{tabular}}}}%
    \put(0.85962134,0.2197903){\color[rgb]{0,0,0}\makebox(0,0)[lt]{\lineheight{1.25}\smash{\begin{tabular}[t]{l}$9$\end{tabular}}}}%
    \put(0.75819772,0.10096803){\color[rgb]{0,0,0}\makebox(0,0)[lt]{\lineheight{1.25}\smash{\begin{tabular}[t]{l}$8$\end{tabular}}}}%
    \put(0.6247989,0.1729799){\color[rgb]{0,0,0}\makebox(0,0)[lt]{\lineheight{1.25}\smash{\begin{tabular}[t]{l}$7$\end{tabular}}}}%
    \put(0.48199515,0.11908629){\color[rgb]{0,0,0}\makebox(0,0)[lt]{\lineheight{1.25}\smash{\begin{tabular}[t]{l}$2$\end{tabular}}}}%
    \put(0.91544303,0.11774269){\color[rgb]{0,0,0}\makebox(0,0)[lt]{\lineheight{1.25}\smash{\begin{tabular}[t]{l}$2$\end{tabular}}}}%
    \put(0.5707422,0.26119847){\color[rgb]{0,0,0}\makebox(0,0)[lt]{\lineheight{1.25}\smash{\begin{tabular}[t]{l}$5$\end{tabular}}}}%
    \put(0.31328187,0.26499862){\color[rgb]{0,0,0}\makebox(0,0)[lt]{\lineheight{1.25}\smash{\begin{tabular}[t]{l}$4$\end{tabular}}}}%
    \put(0.27053012,0.33055125){\color[rgb]{0,0,0}\makebox(0,0)[lt]{\lineheight{1.25}\smash{\begin{tabular}[t]{l}$e_{16}$\end{tabular}}}}%
    \put(0.43488674,0.2868495){\color[rgb]{0,0,0}\makebox(0,0)[lt]{\lineheight{1.25}\smash{\begin{tabular}[t]{l}$v_1$\end{tabular}}}}%
    \put(0.9231198,0.3684955){\color[rgb]{0,0,0}\makebox(0,0)[lt]{\lineheight{1.25}\smash{\begin{tabular}[t]{l}$x_6$\end{tabular}}}}%
    \put(0.68407871,0.00442082){\color[rgb]{0,0,0}\makebox(0,0)[lt]{\lineheight{1.25}\smash{\begin{tabular}[t]{l}$x_6$\end{tabular}}}}%
    \put(0,0){\includegraphics[width=\unitlength,page=2]{CW_presentation_labeled.pdf}}%
  \end{picture}%
\endgroup%

%% file: gordon_CW.pdf_tex
%% Creator: Inkscape 1.1 (c68e22c387, 2021-05-23), www.inkscape.org
%% PDF/EPS/PS + LaTeX output extension by Johan Engelen, 2010
%% Accompanies image file 'gordon_CW.pdf' (pdf, eps, ps)
%%
%% To include the image in your LaTeX document, write
%%   \input{<filename>.pdf_tex}
%%  instead of
%%   \includegraphics{<filename>.pdf}
%% To scale the image, write
%%   \def\svgwidth{<desired width>}
%%   \input{<filename>.pdf_tex}
%%  instead of
%%   \includegraphics[width=<desired width>]{<filename>.pdf}
%%
%% Images with a different path to the parent latex file can
%% be accessed with the `import' package (which may need to be
%% installed) using
%%   \usepackage{import}
%% in the preamble, and then including the image with
%%   \import{<path to file>}{<filename>.pdf_tex}
%% Alternatively, one can specify
%%   \graphicspath{{<path to file>/}}
%% 
%% For more information, please see info/svg-inkscape on CTAN:
%%   http://tug.ctan.org/tex-archive/info/svg-inkscape
%%
\begingroup%
  \makeatletter%
  \providecommand\color[2][]{%
    \errmessage{(Inkscape) Color is used for the text in Inkscape, but the package 'color.sty' is not loaded}%
    \renewcommand\color[2][]{}%
  }%
  \providecommand\transparent[1]{%
    \errmessage{(Inkscape) Transparency is used (non-zero) for the text in Inkscape, but the package 'transparent.sty' is not loaded}%
    \renewcommand\transparent[1]{}%
  }%
  \providecommand\rotatebox[2]{#2}%
  \newcommand*\fsize{\dimexpr\f@size pt\relax}%
  \newcommand*\lineheight[1]{\fontsize{\fsize}{#1\fsize}\selectfont}%
  \ifx\svgwidth\undefined%
    \setlength{\unitlength}{448.89976682bp}%
    \ifx\svgscale\undefined%
      \relax%
    \else%
      \setlength{\unitlength}{\unitlength * \real{\svgscale}}%
    \fi%
  \else%
    \setlength{\unitlength}{\svgwidth}%
  \fi%
  \global\let\svgwidth\undefined%
  \global\let\svgscale\undefined%
  \makeatother%
  \begin{picture}(1,0.48539862)%
    \lineheight{1}%
    \setlength\tabcolsep{0pt}%
    \put(0,0){\includegraphics[width=\unitlength,page=1]{gordon_CW.pdf}}%
    \put(0.34873693,0.39451725){\color[rgb]{0,0,0}\makebox(0,0)[lt]{\lineheight{1.25}\smash{\begin{tabular}[t]{l}$x^{+1}$\end{tabular}}}}%
    \put(0.21312464,0.43583408){\color[rgb]{0,0,0}\makebox(0,0)[lt]{\lineheight{1.25}\smash{\begin{tabular}[t]{l}$x^{-1}$\end{tabular}}}}%
    \put(0.20991887,0.01044915){\color[rgb]{0,0,0}\makebox(0,0)[lt]{\lineheight{1.25}\smash{\begin{tabular}[t]{l}$x^{-1}$\end{tabular}}}}%
    \put(0.35027985,0.05277859){\color[rgb]{0,0,0}\makebox(0,0)[lt]{\lineheight{1.25}\smash{\begin{tabular}[t]{l}$y^{+1}$\end{tabular}}}}%
    \put(0.67928981,0.01462605){\color[rgb]{0,0,0}\makebox(0,0)[lt]{\lineheight{1.25}\smash{\begin{tabular}[t]{l}$x^{-1}$\end{tabular}}}}%
    \put(0.5486728,0.010901){\color[rgb]{0,0,0}\makebox(0,0)[lt]{\lineheight{1.25}\smash{\begin{tabular}[t]{l}$x^{+1}$\end{tabular}}}}%
    \put(0.67270103,0.42778745){\color[rgb]{0,0,0}\makebox(0,0)[lt]{\lineheight{1.25}\smash{\begin{tabular}[t]{l}$x^{+1}$\end{tabular}}}}%
    \put(0.78356698,0.37169794){\color[rgb]{0,0,0}\makebox(0,0)[lt]{\lineheight{1.25}\smash{\begin{tabular}[t]{l}$x^{-1}$\end{tabular}}}}%
    \put(0.83261576,0.27682313){\color[rgb]{0,0,0}\makebox(0,0)[lt]{\lineheight{1.25}\smash{\begin{tabular}[t]{l}$x^{+1}$\end{tabular}}}}%
    \put(0.83201904,0.16773636){\color[rgb]{0,0,0}\makebox(0,0)[lt]{\lineheight{1.25}\smash{\begin{tabular}[t]{l}$y^{+1}$\end{tabular}}}}%
    \put(0.78167449,0.07326219){\color[rgb]{0,0,0}\makebox(0,0)[lt]{\lineheight{1.25}\smash{\begin{tabular}[t]{l}$y^{-1}$\end{tabular}}}}%
    \put(0,0){\includegraphics[width=\unitlength,page=2]{gordon_CW.pdf}}%
    \put(0.07143516,0.38750165){\color[rgb]{0,0,0}\makebox(0,0)[lt]{\lineheight{1.25}\smash{\begin{tabular}[t]{l}$x^{+1}$\end{tabular}}}}%
    \put(0,0){\includegraphics[width=\unitlength,page=3]{gordon_CW.pdf}}%
    \put(0.00392266,0.29065594){\color[rgb]{0,0,0}\makebox(0,0)[lt]{\lineheight{1.25}\smash{\begin{tabular}[t]{l}$y^{-1}$\end{tabular}}}}%
    \put(0,0){\includegraphics[width=\unitlength,page=4]{gordon_CW.pdf}}%
    \put(-0.00168598,0.15222222){\color[rgb]{0,0,0}\makebox(0,0)[lt]{\lineheight{1.25}\smash{\begin{tabular}[t]{l}$y^{+1}$\end{tabular}}}}%
    \put(0,0){\includegraphics[width=\unitlength,page=5]{gordon_CW.pdf}}%
    \put(0.37065263,0.17012304){\color[rgb]{0,0,0}\makebox(0,0)[lt]{\lineheight{1.25}\smash{\begin{tabular}[t]{l}$y^{-1}$\end{tabular}}}}%
    \put(0,0){\includegraphics[width=\unitlength,page=6]{gordon_CW.pdf}}%
    \put(0.36707246,0.27383953){\color[rgb]{0,0,0}\makebox(0,0)[lt]{\lineheight{1.25}\smash{\begin{tabular}[t]{l}$x^{-1}$\end{tabular}}}}%
    \put(0,0){\includegraphics[width=\unitlength,page=7]{gordon_CW.pdf}}%
    \put(0.45752298,0.37138479){\color[rgb]{0,0,0}\makebox(0,0)[lt]{\lineheight{1.25}\smash{\begin{tabular}[t]{l}$y^{+1}$\end{tabular}}}}%
    \put(0.54658833,0.42202071){\color[rgb]{0,0,0}\makebox(0,0)[lt]{\lineheight{1.25}\smash{\begin{tabular}[t]{l}$x^{-1}$\end{tabular}}}}%
    \put(0.44225628,0.07405494){\color[rgb]{0,0,0}\makebox(0,0)[lt]{\lineheight{1.25}\smash{\begin{tabular}[t]{l}$y^{-1}$\end{tabular}}}}%
    \put(0.46408661,0.16799227){\color[rgb]{0,0,0}\makebox(0,0)[lt]{\lineheight{1.25}\smash{\begin{tabular}[t]{l}$y^{+1}$\end{tabular}}}}%
    \put(0.07815208,0.05186834){\color[rgb]{0,0,0}\makebox(0,0)[lt]{\lineheight{1.25}\smash{\begin{tabular}[t]{l}$x^{+1}$\end{tabular}}}}%
    \put(0.46059157,0.27116084){\color[rgb]{0,0,0}\makebox(0,0)[lt]{\lineheight{1.25}\smash{\begin{tabular}[t]{l}$y^{-1}$\end{tabular}}}}%
    \put(0,0){\includegraphics[width=\unitlength,page=8]{gordon_CW.pdf}}%
    \put(0.4286262,0.18798649){\color[rgb]{0,0,0}\makebox(0,0)[lt]{\lineheight{1.25}\smash{\begin{tabular}[t]{l}$o$\end{tabular}}}}%
    \put(0.3035246,0.42655843){\color[rgb]{0,0,0}\makebox(0,0)[lt]{\lineheight{1.25}\smash{\begin{tabular}[t]{l}$o$\end{tabular}}}}%
    \put(0.04218825,0.34533974){\color[rgb]{0,0,0}\makebox(0,0)[lt]{\lineheight{1.25}\smash{\begin{tabular}[t]{l}$o$\end{tabular}}}}%
    \put(0.05148824,0.09188827){\color[rgb]{0,0,0}\makebox(0,0)[lt]{\lineheight{1.25}\smash{\begin{tabular}[t]{l}$o$\end{tabular}}}}%
    \put(0.29466802,0.01493134){\color[rgb]{0,0,0}\makebox(0,0)[lt]{\lineheight{1.25}\smash{\begin{tabular}[t]{l}$o$\end{tabular}}}}%
    \put(0.50632462,0.40437886){\color[rgb]{0,0,0}\makebox(0,0)[lt]{\lineheight{1.25}\smash{\begin{tabular}[t]{l}$o$\end{tabular}}}}%
    \put(0.74564308,0.39953681){\color[rgb]{0,0,0}\makebox(0,0)[lt]{\lineheight{1.25}\smash{\begin{tabular}[t]{l}$o$\end{tabular}}}}%
    \put(0.84457541,0.22446594){\color[rgb]{0,0,0}\makebox(0,0)[lt]{\lineheight{1.25}\smash{\begin{tabular}[t]{l}$o$\end{tabular}}}}%
    \put(0.74612047,0.04151862){\color[rgb]{0,0,0}\makebox(0,0)[lt]{\lineheight{1.25}\smash{\begin{tabular}[t]{l}$o$\end{tabular}}}}%
    \put(0.51122604,0.03246574){\color[rgb]{0,0,0}\makebox(0,0)[lt]{\lineheight{1.25}\smash{\begin{tabular}[t]{l}$o$\end{tabular}}}}%
    \put(0.40678546,0.34632843){\color[rgb]{0,0,0}\makebox(0,0)[lt]{\lineheight{1.25}\smash{\begin{tabular}[t]{l}$x$\end{tabular}}}}%
    \put(0.14713685,0.42350707){\color[rgb]{0,0,0}\makebox(0,0)[lt]{\lineheight{1.25}\smash{\begin{tabular}[t]{l}$x$\end{tabular}}}}%
    \put(0.15573778,0.01481201){\color[rgb]{0,0,0}\makebox(0,0)[lt]{\lineheight{1.25}\smash{\begin{tabular}[t]{l}$x$\end{tabular}}}}%
    \put(0.62046236,0.44081129){\color[rgb]{0,0,0}\makebox(0,0)[lt]{\lineheight{1.25}\smash{\begin{tabular}[t]{l}$x$\end{tabular}}}}%
    \put(0.82147281,0.32530771){\color[rgb]{0,0,0}\makebox(0,0)[lt]{\lineheight{1.25}\smash{\begin{tabular}[t]{l}$x$\end{tabular}}}}%
    \put(0.62110205,0.00392671){\color[rgb]{0,0,0}\makebox(0,0)[lt]{\lineheight{1.25}\smash{\begin{tabular}[t]{l}$x$\end{tabular}}}}%
    \put(0.25986406,0.21442355){\color[rgb]{0,0,0}\makebox(0,0)[lt]{\lineheight{1.25}\smash{\begin{tabular}[t]{l}$v_{r_1}$\end{tabular}}}}%
    \put(0.65873217,0.21488191){\color[rgb]{0,0,0}\makebox(0,0)[lt]{\lineheight{1.25}\smash{\begin{tabular}[t]{l}$v_{r_2}$\end{tabular}}}}%
    \put(0.00260835,0.22224566){\color[rgb]{0,0,0}\makebox(0,0)[lt]{\lineheight{1.25}\smash{\begin{tabular}[t]{l}$y$\end{tabular}}}}%
    \put(0.40059679,0.09451001){\color[rgb]{0,0,0}\makebox(0,0)[lt]{\lineheight{1.25}\smash{\begin{tabular}[t]{l}$y$\end{tabular}}}}%
    \put(0.4346,0.11125994){\color[rgb]{0,0,0}\makebox(0,0)[lt]{\lineheight{1.25}\smash{\begin{tabular}[t]{l}$y$\end{tabular}}}}%
    \put(0.8126757,0.11683503){\color[rgb]{0,0,0}\makebox(0,0)[lt]{\lineheight{1.25}\smash{\begin{tabular}[t]{l}$y$\end{tabular}}}}%
    \put(0.43376461,0.32716194){\color[rgb]{0,0,0}\makebox(0,0)[lt]{\lineheight{1.25}\smash{\begin{tabular}[t]{l}$y$\end{tabular}}}}%
    \put(0,0){\includegraphics[width=\unitlength,page=9]{gordon_CW.pdf}}%
    \put(0.77734179,0.468256){\color[rgb]{0,0,0}\makebox(0,0)[lt]{\lineheight{1.25}\smash{\begin{tabular}[t]{l}$k-1 \text{ times}$\end{tabular}}}}%
  \end{picture}%
\endgroup%

%% file: gordon_matching.pdf_tex
%% Creator: Inkscape 1.1 (c68e22c387, 2021-05-23), www.inkscape.org
%% PDF/EPS/PS + LaTeX output extension by Johan Engelen, 2010
%% Accompanies image file 'gordon_matching.pdf' (pdf, eps, ps)
%%
%% To include the image in your LaTeX document, write
%%   \input{<filename>.pdf_tex}
%%  instead of
%%   \includegraphics{<filename>.pdf}
%% To scale the image, write
%%   \def\svgwidth{<desired width>}
%%   \input{<filename>.pdf_tex}
%%  instead of
%%   \includegraphics[width=<desired width>]{<filename>.pdf}
%%
%% Images with a different path to the parent latex file can
%% be accessed with the `import' package (which may need to be
%% installed) using
%%   \usepackage{import}
%% in the preamble, and then including the image with
%%   \import{<path to file>}{<filename>.pdf_tex}
%% Alternatively, one can specify
%%   \graphicspath{{<path to file>/}}
%% 
%% For more information, please see info/svg-inkscape on CTAN:
%%   http://tug.ctan.org/tex-archive/info/svg-inkscape
%%
\begingroup%
  \makeatletter%
  \providecommand\color[2][]{%
    \errmessage{(Inkscape) Color is used for the text in Inkscape, but the package 'color.sty' is not loaded}%
    \renewcommand\color[2][]{}%
  }%
  \providecommand\transparent[1]{%
    \errmessage{(Inkscape) Transparency is used (non-zero) for the text in Inkscape, but the package 'transparent.sty' is not loaded}%
    \renewcommand\transparent[1]{}%
  }%
  \providecommand\rotatebox[2]{#2}%
  \newcommand*\fsize{\dimexpr\f@size pt\relax}%
  \newcommand*\lineheight[1]{\fontsize{\fsize}{#1\fsize}\selectfont}%
  \ifx\svgwidth\undefined%
    \setlength{\unitlength}{449.36876192bp}%
    \ifx\svgscale\undefined%
      \relax%
    \else%
      \setlength{\unitlength}{\unitlength * \real{\svgscale}}%
    \fi%
  \else%
    \setlength{\unitlength}{\svgwidth}%
  \fi%
  \global\let\svgwidth\undefined%
  \global\let\svgscale\undefined%
  \makeatother%
  \begin{picture}(1,0.47864008)%
    \lineheight{1}%
    \setlength\tabcolsep{0pt}%
    \put(0,0){\includegraphics[width=\unitlength,page=1]{gordon_matching.pdf}}%
    \put(0.77757417,0.46151535){\color[rgb]{0,0,0}\makebox(0,0)[lt]{\lineheight{1.25}\smash{\begin{tabular}[t]{l}$k-1 \text{ times}$\end{tabular}}}}%
    \put(0,0){\includegraphics[width=\unitlength,page=2]{gordon_matching.pdf}}%
    \put(0.08611797,0.38797357){\color[rgb]{0,0,0}\makebox(0,0)[lt]{\lineheight{1.25}\smash{\begin{tabular}[t]{l}$x_1$\end{tabular}}}}%
    \put(0.3602651,0.38620927){\color[rgb]{0,0,0}\makebox(0,0)[lt]{\lineheight{1.25}\smash{\begin{tabular}[t]{l}$x_1$\end{tabular}}}}%
    \put(0.08755241,0.04297284){\color[rgb]{0,0,0}\makebox(0,0)[lt]{\lineheight{1.25}\smash{\begin{tabular}[t]{l}$x_1$\end{tabular}}}}%
    \put(0.5565781,0.00710604){\color[rgb]{0,0,0}\makebox(0,0)[lt]{\lineheight{1.25}\smash{\begin{tabular}[t]{l}$x_1$\end{tabular}}}}%
    \put(0.83519172,0.26972789){\color[rgb]{0,0,0}\makebox(0,0)[lt]{\lineheight{1.25}\smash{\begin{tabular}[t]{l}$x_1$\end{tabular}}}}%
    \put(0.67936935,0.41979398){\color[rgb]{0,0,0}\makebox(0,0)[lt]{\lineheight{1.25}\smash{\begin{tabular}[t]{l}$x_1$\end{tabular}}}}%
    \put(0.21800791,0.43231154){\color[rgb]{0,0,0}\makebox(0,0)[lt]{\lineheight{1.25}\smash{\begin{tabular}[t]{l}$x_2$\end{tabular}}}}%
    \put(0.37859035,0.27566323){\color[rgb]{0,0,0}\makebox(0,0)[lt]{\lineheight{1.25}\smash{\begin{tabular}[t]{l}$x_2$\end{tabular}}}}%
    \put(0.20549035,0.00313809){\color[rgb]{0,0,0}\makebox(0,0)[lt]{\lineheight{1.25}\smash{\begin{tabular}[t]{l}$x_2$\end{tabular}}}}%
    \put(0.78215794,0.36353306){\color[rgb]{0,0,0}\makebox(0,0)[lt]{\lineheight{1.25}\smash{\begin{tabular}[t]{l}$x_2$\end{tabular}}}}%
    \put(0.55419381,0.42431567){\color[rgb]{0,0,0}\makebox(0,0)[lt]{\lineheight{1.25}\smash{\begin{tabular}[t]{l}$x_2$\end{tabular}}}}%
    \put(0.66840156,0.00456866){\color[rgb]{0,0,0}\makebox(0,0)[lt]{\lineheight{1.25}\smash{\begin{tabular}[t]{l}$x_2$\end{tabular}}}}%
    \put(0.83389731,0.15692524){\color[rgb]{0,0,0}\makebox(0,0)[lt]{\lineheight{1.25}\smash{\begin{tabular}[t]{l}$x_3$\end{tabular}}}}%
    \put(0.4553901,0.36326921){\color[rgb]{0,0,0}\makebox(0,0)[lt]{\lineheight{1.25}\smash{\begin{tabular}[t]{l}$x_3$\end{tabular}}}}%
    \put(0.46096779,0.16467421){\color[rgb]{0,0,0}\makebox(0,0)[lt]{\lineheight{1.25}\smash{\begin{tabular}[t]{l}$x_3$\end{tabular}}}}%
    \put(0.35350378,0.04208745){\color[rgb]{0,0,0}\makebox(0,0)[lt]{\lineheight{1.25}\smash{\begin{tabular}[t]{l}$x_3$\end{tabular}}}}%
    \put(0.00539642,0.14485915){\color[rgb]{0,0,0}\makebox(0,0)[lt]{\lineheight{1.25}\smash{\begin{tabular}[t]{l}$x_3$\end{tabular}}}}%
    \put(0.45336344,0.063597){\color[rgb]{0,0,0}\makebox(0,0)[lt]{\lineheight{1.25}\smash{\begin{tabular}[t]{l}$x_4$\end{tabular}}}}%
    \put(0.77418759,0.05877755){\color[rgb]{0,0,0}\makebox(0,0)[lt]{\lineheight{1.25}\smash{\begin{tabular}[t]{l}$x_4$\end{tabular}}}}%
    \put(0.45941797,0.26159587){\color[rgb]{0,0,0}\makebox(0,0)[lt]{\lineheight{1.25}\smash{\begin{tabular}[t]{l}$x_4$\end{tabular}}}}%
    \put(0.38300131,0.17135025){\color[rgb]{0,0,0}\makebox(0,0)[lt]{\lineheight{1.25}\smash{\begin{tabular}[t]{l}$x_4$\end{tabular}}}}%
    \put(0.00199064,0.27804752){\color[rgb]{0,0,0}\makebox(0,0)[lt]{\lineheight{1.25}\smash{\begin{tabular}[t]{l}$x_4$\end{tabular}}}}%
    \put(0.33746123,0.23453412){\color[rgb]{0,0,0}\makebox(0,0)[lt]{\lineheight{1.25}\smash{\begin{tabular}[t]{l}$x_5$\end{tabular}}}}%
    \put(0.30149221,0.31289216){\color[rgb]{0,0,0}\makebox(0,0)[lt]{\lineheight{1.25}\smash{\begin{tabular}[t]{l}$x_6$\end{tabular}}}}%
    \put(0.22992942,0.34041387){\color[rgb]{0,0,0}\makebox(0,0)[lt]{\lineheight{1.25}\smash{\begin{tabular}[t]{l}$x_7$\end{tabular}}}}%
    \put(0.16102324,0.32334918){\color[rgb]{0,0,0}\makebox(0,0)[lt]{\lineheight{1.25}\smash{\begin{tabular}[t]{l}$x_8$\end{tabular}}}}%
    \put(0.10109204,0.27862666){\color[rgb]{0,0,0}\makebox(0,0)[lt]{\lineheight{1.25}\smash{\begin{tabular}[t]{l}$x_9$\end{tabular}}}}%
    \put(0.08448729,0.20435577){\color[rgb]{0,0,0}\makebox(0,0)[lt]{\lineheight{1.25}\smash{\begin{tabular}[t]{l}$x_{10}$\end{tabular}}}}%
    \put(0.12614412,0.12644858){\color[rgb]{0,0,0}\makebox(0,0)[lt]{\lineheight{1.25}\smash{\begin{tabular}[t]{l}$x_{11}$\end{tabular}}}}%
    \put(0.20311457,0.08930503){\color[rgb]{0,0,0}\makebox(0,0)[lt]{\lineheight{1.25}\smash{\begin{tabular}[t]{l}$x_{12}$\end{tabular}}}}%
    \put(0.27930981,0.11267958){\color[rgb]{0,0,0}\makebox(0,0)[lt]{\lineheight{1.25}\smash{\begin{tabular}[t]{l}$x_{13}$\end{tabular}}}}%
    \put(0.33166207,0.16100397){\color[rgb]{0,0,0}\makebox(0,0)[lt]{\lineheight{1.25}\smash{\begin{tabular}[t]{l}$x_{14}$\end{tabular}}}}%
    \put(0.47654245,0.23513019){\color[rgb]{0,0,0}\makebox(0,0)[lt]{\lineheight{1.25}\smash{\begin{tabular}[t]{l}$x_{15}$\end{tabular}}}}%
    \put(0.5088666,0.30587541){\color[rgb]{0,0,0}\makebox(0,0)[lt]{\lineheight{1.25}\smash{\begin{tabular}[t]{l}$x_{16}$\end{tabular}}}}%
    \put(0.5754055,0.3382934){\color[rgb]{0,0,0}\makebox(0,0)[lt]{\lineheight{1.25}\smash{\begin{tabular}[t]{l}$x_{17}$\end{tabular}}}}%
    \put(0.64277016,0.34087374){\color[rgb]{0,0,0}\makebox(0,0)[lt]{\lineheight{1.25}\smash{\begin{tabular}[t]{l}$x_{18}$\end{tabular}}}}%
    \put(0.69989952,0.31416964){\color[rgb]{0,0,0}\makebox(0,0)[lt]{\lineheight{1.25}\smash{\begin{tabular}[t]{l}$x_{13+2k}$\end{tabular}}}}%
    \put(0.71561057,0.20282294){\color[rgb]{0,0,0}\makebox(0,0)[lt]{\lineheight{1.25}\smash{\begin{tabular}[t]{l}$x_{15+2k}$\end{tabular}}}}%
    \put(0.68822591,0.14342007){\color[rgb]{0,0,0}\makebox(0,0)[lt]{\lineheight{1.25}\smash{\begin{tabular}[t]{l}$x_{16+2k}$\end{tabular}}}}%
    \put(0.66506359,0.09850998){\color[rgb]{0,0,0}\makebox(0,0)[lt]{\lineheight{1.25}\smash{\begin{tabular}[t]{l}$x_{17+2k}$\end{tabular}}}}%
    \put(0.59210411,0.08360813){\color[rgb]{0,0,0}\makebox(0,0)[lt]{\lineheight{1.25}\smash{\begin{tabular}[t]{l}$x_{18+2k}$\end{tabular}}}}%
    \put(0.52662982,0.10864324){\color[rgb]{0,0,0}\makebox(0,0)[lt]{\lineheight{1.25}\smash{\begin{tabular}[t]{l}$x_{19+2k}$\end{tabular}}}}%
    \put(0.48628426,0.17794096){\color[rgb]{0,0,0}\makebox(0,0)[lt]{\lineheight{1.25}\smash{\begin{tabular}[t]{l}$x_{20+2k}$\end{tabular}}}}%
    \put(0,0){\includegraphics[width=\unitlength,page=3]{gordon_matching.pdf}}%
    \put(0.74474108,0.25959465){\color[rgb]{0,0,0}\makebox(0,0)[lt]{\lineheight{1.25}\smash{\begin{tabular}[t]{l}$x_{14+2k}$\end{tabular}}}}%
    \put(0,0){\includegraphics[width=\unitlength,page=4]{gordon_matching.pdf}}%
  \end{picture}%
\endgroup%